\RequirePackage{fix-cm}
\documentclass[final]{svjour3}
\usepackage{amsmath}
\usepackage{amssymb}
\smartqed

\usepackage{latexsym}
\usepackage{srcltx}
\usepackage{graphics}
\usepackage{graphicx}
\usepackage{epsfig}
\usepackage{GGGraphics}
\usepackage[ruled,vlined]{algorithm2e}

\newcommand{\re}{{\mathbb R}}
\newcommand{\n}{{\mathbb N}}

\newcommand{\z}{{\mathbb Z}}
\newcommand{\cA}{{\cal{A}}}
\newcommand{\cO}{{\cal{O}}}
\newcommand{\cV}{{\cal{V}}}
\newcommand{\cL}{{\cal{L}}}
\newcommand{\cW}{{\cal{W}}}
\newcommand{\cB}{{\cal{B}}}
\newcommand{\cT}{{\cal{T}}}
\newcommand{\cH}{{\cal{H}}}
\newcommand{\cS}{{\cal{S}}}
\newcommand{\cR}{{\cal{R}}}
\newcommand{\cQ}{{\cal{Q}}}
\newcommand{\cP}{P}
\newcommand{\cM}{{\cal{M}}}
\newcommand{\cU}{{\cal{U}}}

\newcommand{\e}{{\rm e}}
\newcommand{\dt}{\tau}
\newcommand{\dtt}{}
\newcommand{\ds}{}
\newcommand{\eps}{\varepsilon}

\newcommand{\tp}{{\rm T}}
\newcommand{\logm}{\rm log}

\newtheorem{prop}{Proposition}
\newtheorem{cor}{Corollary}

\newtheorem{defi}{Definition}

\begin{document}

\title{Polytope Lyapunov functions for stable and for stabilizable LSS
}

\author{Nicola Guglielmi\thanks{The research of the first author is supported by Italian INDAM G.N.C.S.} 
\and Linda Laglia
\and
Vladimir Protasov\thanks{The research of the third author is supported by the RFBR grants No 13-01-00642 and No 14-01-00332,
and by the grant of the Dynasty foundation.}}
\institute{Nicola Guglielmi and Linda Laglia \at Department of Pure and Applied Mathematics, \\ University of L'Aquila, Italy \\ \email{guglielm@univaq.it, laglia@univaq.it} \and
Vladimir Protasov \at Department of Mechanics and Mathematics, Moscow State University, Vorobyovy Gory, Moscow, Russia 119992 \\ \email{v-protassov@yandex.ru}}

\date{}

\journalname{
}

\maketitle

\begin{abstract}

We present a new approach for constructing polytope Lyapunov functions for continuous-time linear switching systems
(LSS). This allows us to decide the stability of LSS and to compute the Lyapunov exponent with a good precision
in relatively high dimensions. The same technique is also extended for stabilizability of positive systems by evaluating
a polytope concave Lyapunov function (``antinorm'') in the cone. 
The method is based on a suitable discretization of the underlying continuous system and provides both a lower and an upper bound
for the Lyapunov exponent.
The absolute error in the Lyapunov exponent computation
is estimated from above and proved to be  linear in the dwell time.
The practical efficiency of the new method is demonstrated  in several examples and in the list of numerical experiments with randomly
 generated matrices of dimensions up to $10$ (for general linear systems) and up to $100$ (for positive systems).
The development of the method is based on several theoretical results proved in the paper: the existence of monotone invariant 
norms and antinorms for positively irreducible systems, the equivalence of all contractive norms for stable systems and the
linear convergence theorem.

\smallskip

\smallskip

\begin{flushleft}
\noindent  \textbf{AMS 2010} {\em subject
classification: } 15A60, 15-04, 15A18, 90C90
\end{flushleft}

\end{abstract}


\section{Introduction}

The stability of linear switching systems (LSS) has been studied in the literature in great
detail. We consider continuous-time LSS which is the following linear system of ODE
on the vector-valued function $x: [0, +\infty) \to \re^d$:
\begin{equation}\label{main}
\left\{
\begin{array}{l}
\dot x (t) \ = \ A(t)\, x(t)\,  ;\\
x(0) \, = \, x_0\, ;\\
A(t) \in \cA \, ,\ t \ge 0\, .
\end{array}
\right.
\end{equation}
Here $A(\cdot )$ is a {\em control function}, also called {\em switching law} which is a summable function that takes values on a given compact set~$\cA$ of $d\times d$ matrices. Since the range of the function
$A(\cdot)$ is compact, the summability of~$A(\cdot)$
is equivalent to its measurability. The set of  control functions on an interval $[a,b]$ will be denoted by $\cU\, [a,b]$.
We use the short notation~$\cU\, [0, +\infty) = \cU$. The space of all summable functions will be denoted as usual by $L_1$.

The {\em Lyapunov exponent} $ \sigma (\cA)$ is the infimum
of numbers $\alpha$ such that $\|x(t)\| \, \le \, C e^{\, \alpha t}$
for every trajectory of~(\ref{main}). The system, or the corresponding family of matrices~$\cA$, is {\em stable} if
 $\|x(t)\| \to 0$ as $t \to +\infty$ for every trajectory of~(\ref{main}).
 Obviously, if $ \sigma < 0$, then the system is stable, and, conversely,
 the stability implies that $ \sigma \le 0$. This small gap between the necessary and sufficient
 condition can be handled: actually the system is stable if and only if $ \sigma < 0$
 (see, for instance, \cite{B,O}).

 It is well known that $\sigma(\cA) = \sigma ({\rm co}\, (\cA))$, where $\, {\rm co}\, (\cdot )$ denotes the convex hull.
 In particular, the family~${\rm co}\, (\cA)$
 is stable if so is the family~$\cA$~\cite{B}. We call a {\em generalized trajectory}
 of system~(\ref{main}) a trajectory $x(\cdot )$ corresponding to a control function $A(\cdot)$
 with values from the convex hull ${\rm co}\, (\cA)$. So, any trajectory is a generalized trajectory as well.
 If $\cA$ is a convex set, then the converse is true, and the notions of trajectory and of generalized trajectory coincide.
Let $I$ denote the $d\times d$ identity matrix and $\cA + s I \, = \, \{A + sI \ | \ A \in \cA\}$, where $s$ is a number. The following equality is checked directly:
\begin{equation}\label{plus}
  \sigma \, (\cA\,  + \, s \, I) \ = \  \sigma \, (\cA) \, + \, s \  .
\end{equation}

This means that, by simple shift of the matrices, the comparison of the Lyapunov exponent with a given number is equivalent  to
the comparison with zero, i.e., to deciding the stability. Hence, the computation of the Lyapunov exponent is
reduced to the stability problem by means of the double division principle. The most popular approach to prove the stability of LSS is by constructing a Lyapunov function~$f(x)$, a positive homogeneous function on~$\re^d$
that decreases along any trajectory of the system. See~\cite{Gurv,L,MP1} for the general theory of Lyapunov functions for LSS.
In most of applications,  the quadratic function $f(x) = \sqrt{x^TMx}$ (CQLF -- {\em common quadratic Lyapunov function})
appears to be quite efficient. Here $M$ is a symmetric positive definite matrix, and the Lyapunov function property
is equivalent to the following system of matrix inequalities: $\, A^TM + MA \, \prec \, 0\, , \  A \in \cA$.
The CQLF can be found by solving the corresponding SDP (semidefinite programming problem), which can be efficiently done
using standard computer software, mostly in dimensions $d \le 20$ or slightly more. The main disadvantage of this approach is that in some 
examples the precision of this method is not satisfactory.
The SDP system may have no solutions even if the system is very stable, say, when  $\sigma(\cA) = -1$. The reason is that quadratic functions are not dense (or, in other terms, not {\em universal}) in the set of all
Lyapunov functions. The examples of stable LSS that have no CQLF are well known. Due to the compactness argument,
each of those systems have an irreducible error in the Lyapunov exponent computation by CQLF. That is why, other classes of functions have been used in the literature to construct Lyapunov functions: positive polynomials of higher degree, SOS, piecewise-quadratic, piecewise-linear, etc. (see surveys~\cite{LA,YS}). In contrast to quadratic functions, all those classes are dense which implies their universality, i.e., every stable LSS has a Lyapunov function from those classes. However,
in many cases, this advantage is rather
  theoretical because the constructing of such Lyapunov functions is hard, even in relatively small dimensions~$d$.

The class of {\em polytope} functions (also referred to as {\em polyhedral} or {\em piecewise-linear}) is the simplest one and it drew much attention in the literature. The first polytope algorithms originated in late eighties with Molchanov and Pyatnitskii~\cite{MP1,MP2} and
Barabanov~\cite{B-poly}.  Then this method was developed in various directions by Amato, Ambrosino, Ariola, Blanchini, Miani, 
Julian, Guivant, Desages, Polanski, Shorten, Yfoulis and others  (see ~\cite{AAA,BM1,BM2,Bob,JGD,Mar,Pol1,Pol2,YS}).  
The polytope function can be easily defined by faces of the corresponding level polyhedron~$P$ (unit ball):
\begin{equation}\label{polyh}
f(x) \ = \ \max_{i = 1, \ldots , N}\, \bigl( v^*_i\, , \,  x\bigr)\ ,
\end{equation}
where $\{v_i^*\}_{i=1}^N$ are given vectors (normals to the hyperfaces of~$P$). The Lyapunov function
property (to decrease along any trajectory) becomes the following property of vertices: for each vertex~$v$
of~$P$ the vectors $Av \, , A \in \cA$, starting at~$v$ are all directed inside the polytope~$P$~\cite{MP2,Pol1}.
For the function~$f$ given in the form~(\ref{polyh}), this condition is hard to verify in high dimensions, because it requires
finding all vertices. On the other hand, if $P$ is given in the dual form, merely by the list of its
vertices~$P = {\rm co}\, \{v_j\}_{j=1}^k$, then this condition is checked  easier, just by solving corresponding LP (linear programming)
problems. In this case, however, the function~$f$ loses its explicit form: 
the evaluation of~$f(x)$ at a given point $x \in \re^d$ requires solving an LP problem.
The main challenge in the design of the polytope Lyapunov functions is to chose the location of vertices in a proper way.
In~\cite{Mar,Pol1,Y,YS} this is done by placing all vertices on a given system of ray directions. First, one construct a system of rays that  actually form an $\eps$-net on the unit sphere in~$\re^d$.
Then one selects a vertex in each ray in order to fulfill the Lyapunov property of the polytope~$P$.
 In~\cite{Mar,Pol1} this is done by solving LP problems. In~\cite{Y,YS} the authors introduce iterative ray-gridding approach
  and demonstrate its efficiency in examples of dimensions $d=2$ and $d=3$. Unfortunately, in higher dimensions the number of vertices grows dramatically, which makes those methods hardly applicable. An  $\eps$-net  on the unit sphere contains,  roughly, $\, C \, \eps^{\, 1 - d}$ points~(see, for instance~\cite{Chen,LSW}). Therefore, already in~$\re^4$,  asking for a precision $\eps = 0.01$ to approximate the Lyapunov exponent requires to deal with millions of vertices.

In this paper, we develop a new method to design the polytope Lyapunov function. We use some ideas from our recent work~\cite{GP}
as well as \cite{GWZ,P0,GZ1,GZ2}, where we analyse discrete-time LSS. For them, the stability  depends on the joint spectral radius of matrices, see \S 2.2 for a brief overview. In~\cite{GWZ,P0,GP} we developed  a method of computation of the joint spectral radius. For the vast majority of finite sets of matrices, it finds the exact value. The method works efficiently for general finite sets of matrices of dimensions up to~$20$ although it can reach higher dimension if one accepts a longer computation. 
For sets of nonnegative matrices, it works much faster and it is applicable for dimensions up to several hundreds.
 The main idea is to construct iteratively a common polytope Lyapunov norm, see \S 2.3 for more details. Here we use this argument  to analyze the continuous-time LSS. First, we discretize the system with a properly chosen dwell time~$\tau > 0$. Then we apply the algorithm from~\cite{GP} to the matrices $e^{\, \tau \cA} = \bigl\{e^{\, \tau A} \ \bigl| \ A \in \cA  \bigr\}$ and construct a corresponding
polytope~$P$. Then we use the piecewise linear norm generated by this polytope as a Lyapunov function for the continuous-time LSS obtaining
both a lower bound and an upper bound for the Lyapunov exponent.

The new method can be shortly summarized as follows:

1) the vertices of the polytope~$P$ are generated iteratively. The starting vertices are chosen in a special way: they are  leading eigenvectors of a chosen product~$\Pi$ of matrices from $e^{\, \tau \, \cA}$,
and of cyclic permutations of this product;

2)  the iterations are realized not by the matrices of the family $\cA$ but by shifted matrices
$A + s I$ (formula~(\ref{plus})), where the parameter $s$ is chosen by solving an optimization problem.
\smallskip

As a result,  we obtain a polytope~$P$ that defines a Lyapunov function and
localize the Lyapunov exponent~$\sigma$ in a segment $[\beta  ,\alpha ]$. The precision of this method is
estimated. In particular, we prove that the length of this segment~$\gamma = \alpha - \beta$
decreases linearly with $\tau$ and converges to zero as $\tau \rightarrow 0$. Then we consider numerical examples in several dimensions.
In dimension $d = 5$ (Example \ref{ex:2}), to compute the Lyapunov exponent with an absolute error $\gamma \le 0.25$ we need $\tau$ 
about $0.025$ and a polytope with $10000$ vertices; to compute it with an absolute error $\gamma \le 0.1$ we need $\tau$ about $0.01$ and
a polytope with $20000$ vertices.

In a separate section we consider positive systems, i.e., systems with all trajectories inside a
given cone $K \subset \re^d$. Such systems have been intensively studied in the literature~\cite{AR,FM,FMC,FV1,GSM,SH}.
We start with several theoretical results, the main of which is the theorem on the existence of
monotone invariant norm for a positively irreducible system (Theorem~\ref{th10}). Then we modify the polytope algorithm for
positive systems and estimate its accuracy. The algorithm is written for finitely many matrices in the case $K = \re^d_+$ (i.e., all matrices are Metzler). In numerical examples (Section 5), we see that it works much faster than in the general case in dimensions up to~$100$. 
In dimension $d = 25$ (see Example \ref{ex:6}), to compute the Lyapunov exponent with an absolute error $\gamma \le 0.02$ we need 
$\tau$ about $0.02$ and a polytope with only $210$ vertices.
In dimension $d = 100$ (see Section \ref{sec:stat}), to compute the Lyapunov exponent with an absolute error $\gamma \le 0.1$ we need 
$\tau$ about $0.004$ and a polytope with less than $300$ vertices.

All numerical results are presented in Section~5 and  compared with the CQLF method. In small dimensions (up to~$10$)
our algorithm is more expensive than CQLF, but gives better accuracy. For positive systems, its complexity
grows moderately with the dimension, and for dimensions up to $d = 100$ it still gives good results
(with absolute error about $\gamma = 0.1 - 0.3$), while the CQLF method becomes inapplicable.

The last part of the paper deals with stabilizability of positive systems. The system is called stabilizable if
there is at least one switching low with stable trajectories. The largest possible exponent of growth is
called lower Lyapunov exponent of the system and is denoted as~$\check \sigma (\cA)$, see Section~4 for more details. Stabilizability of positive systems was studied in~\cite{BS,FV1,FV2,LA,SDP,XW}.   An advantage of our method is that it is easily extended
to the stabilizability problem and to computing the lower Lyapunov exponent. To this end we have to consider
concave Lyapunov functions (``antinorms'') instead of convex ones and the so-called ``infinite polytopes'' instead of usual ones. We begin with theoretical results and show the existence of invariant antinorm  for an arbitrary system with positively irreducible matrices (Theorem~\ref{th20}). This allows us to estimate the accuracy of the polytope method for computing the
lower Lyapunov exponent, which also turns out to be linear in the dwell time~$\tau$. In numerical examples presented in Section~5 the algorithm decides stabilizability and computes the lower Lyapunov exponent in dimensions up to~$100$.

The paper is organized as follows. In Section~2 we start with a short summary of results on the joint spectral radius, extremal and invariant norms, and bounds for the Lyapunov exponent. Then we prove the main theoretical result of that section, Theorem~\ref{th8} on the linear upper bound for the precision of the Lyapunov exponent computation. Afterwards we present Algorithm~(R) for computing the Lyapunov exponent
for general finite sets of matrices, estimate its convergence rate and prove the conditions to terminate within finite time  
(Theorem~\ref{th40}).  Section~3 deals with LSS that are positive with respect to a given cone~$K \subset \re^d_+$. We start with the main theoretical result, Theorem~\ref{th10} on monotone invariant norm in the cone. We use it to prove Theorem~\ref{th8-} providing an upper bound for the Lyapunov exponent. Then we describe Algorithm~(P) for computing the Lyapunov exponent of a positive system. Section~4 is concerned
with the stabilizability of positive systems and starts with introducing notions of antinorm and of infinite polytope. Then we present Theorem~\ref{th20} on the existence of a monotone invariant antinorm in a cone.
Finally we derive lower and upper bounds for the lower Lyapunov exponent $\check \sigma (A)$, estimate their distance 
(Theorem~\ref{th8+}) and present Algorithm~(L) for deciding the stabilizability and approximating $\sigma (\cA)$.
The criterion of convergence of Algorithm~(L) and estimates of its precision are proved in Theorem~\ref{th40+}.
In Section~5 we present numerical examples and some statistics of the implementation of our algorithms
to randomly generated matrices of various dimensions.

Throughout the paper, unless we explicitly state differently, a norm of a vector and of a matrix is Euclidean.
For a matrix $A$ and for a set of matrices $\cA$, we denote $e^{\, A} = \sum_{k=0}^{\infty}\frac{1}{k!}A^k\, , \
e^{\, \cA} = \{e^{\, A} | \ A \in \cA\}\, , \ \cA^k \, = \, \{A_k\ldots A_1 \ | \ A_i \in \cA\, , \, i = 1, \ldots , k\}$.

\section{Stability of general linear switching systems}

In this section we provide some general theoretical results concerning stability of switching systems.

\subsection{Extremal and invariant norms}

The main approach to establish the stability of LSS is to compute a Lyapunov function~$f(x)$,
a positive homogeneous continuous function on~$\re^d$ such that, for every trajectory~$x(\cdot)$ of the system,
 the function $f(x(t))$ strictly  decreases in $t$.
Such a function is usually called (joint) Lyapunov function of the family~$\cA$.
The existence of  Lyapunov function implies the stability. A converse statement
is also true, even in the following strong sense: for an arbitrary stable LSS there exists
a {\em convex} symmetric Lyapunov function~\cite{O,MP1}. The symmetry means that~$f(-x) = f(x), \, x \in \re^d$.
 Since a symmetric convex positively homogeneous function on~$\re^d$ is a norm, one can say that
 there is a {\em Lyapunov norm}, i.e., a norm that possesses property of Lyapunov function. If the family~$\cA$ generating
 the LSS is irreducible, this result can be  strengthened to the existence of extremal and invariant (Barabanov) norms.

\begin{defi}\label{d5}
A norm $\|\cdot\|$ is called extremal for a set $\cA$ if
for every trajectory of~(\ref{main})
we have $\|x(t)\| \, \le \, e^{\,  \sigma \, t}\|x(0)\|\, , \ t \ge 0$.

An extremal norm is called  invariant if for every $x_0 \in \re^d$ there exists a generalized trajectory
$\bar x (t)$ such that $\bar x (0) = x_0$ and  $\|x(t)\| \, = \, e^{\,  \sigma \, t}\, \|x_0\|\, , \ t \ge 0$.
\end{defi}

Since every point $x(\tau)$ can be considered as a starting point of a new trajectory
(after the shift of the argument $t' = t - \tau$), it follows that {\em  for an extremal norm
the function $e^{-\,  \sigma \, t}\, \|x(t) \|$ is non-increasing in $t$ on every trajectory. For an invariant norm,
this function is identically constant on some generalized trajectory, and for every point $x_0 \in \re^d$ there is
such a trajectory starting in it}. In particular, for $ \sigma = 0$ we have
\begin{cor}\label{c20}
In case~$ \sigma (\cA) \, = \, 0 \, $ a norm is extremal for $\cA$ if and only if it is non-increasing in~$t$ on every trajectory
of~(\ref{main}). An extremal norm is invariant if and only if for every $x_0 \in \re^d$ there exists a generalized trajectory
$\bar x (t)$ with $\bar x (0) = x_0$ on which this norm is identically constant.
\end{cor}

If we take a unit ball~$B$ of that norm, we see that a norm is extremal if and only if
every trajectory starting on the unit sphere $\partial \,B$ never leaves the ball~$B$.
This norm is invariant if for each point of the sphere there exists a generalized trajectory starting at this point
 that eternally  goes on the sphere.

 A set of matrices $\cA$ is called {\em irreducible} if these matrices  do not share a nontrivial common invariant subspace.
 The following theorem originated with  N.Barabanov in~\cite{B}.


\noindent 
\textbf{Theorem A}. {\em An irreducible set of matrices  possesses an invariant norm.}

\smallskip

The proof is in~\cite{B}.
Clearly, if a family~$\cA$ is stable, i.e., $\sigma = \sigma (\cA) < 0$, then an extremal norm
of the family~$\cA - \sigma I$ is a Lyapunov norm for~$\cA$. Lyapunov norms can be characterized
geometrically in terms of the vector field on the unit sphere.
For each $A \in \cA$, we consider the following vector field on~$\re^d$: to every point $x \in \re^d$
we associate a vector $Ax$ starting at~$x$. For a given convex set $G$, we say that
the vector $Ax$  at the point $x \in G$ {\em is directed inside} $G$, if
there is a number $\eta > 0$ such that $x \, + \, \eta \, Ax \, \in \, {\rm int}\, G$.
It was shown in~\cite{O,MP1} that a norm~$f(\cdot)$ with a unit ball~$G \subset \re^d$ is Lyapunov
for a given LSS if and only if, at every point $x \in \partial \, G$ the vector $Ax$
is directed inside~$G$, for each $A \in \cA$. The following result is  corollary of Theorem~A, but it was
derived much earlier, in works~\cite{MP1,O}:

\noindent 

\textbf{Theorem B}. 
{\em A family of matrices~$\cA$ is stable if and only if
there exists a convex body~${G \subset \re^d}$ symmetric about the origin such that
at every point~$x \in \partial \, G$, the vector $Ax$ is directed inside~$G, \ A \in \cA$.}

\smallskip

\subsection{Discretization and the joint spectral radius}

The idea is that of discretizing (\ref{main}) and imposing that the switching instants are
multiple of a dwell time $\tau$. This allows us to express the solution of the discretized system 
as a product of matrices $\{ B=e^{\tau A}\}$, $A \in \cA$.

A discrete linear switching system is the following system of difference equations on
a sequence $\{x_k\}_{k =0}^{\infty} \subset \re^d$:
\begin{equation}\label{main-discr}
\left\{
\begin{array}{l}
 x_{k+1}  \ = \ B_{k+1} \, x_k\,  ;\\
x_0 \ \in \ \re^d \quad \mbox{is given}\, ;\\
B_k \in \cB \, ,\ k \in \n \cup \{0\} \, ,
\end{array}
\right.
\end{equation}
where $\cB$ is a given compact set of matrices. For an arbitrary sequence $B_k \in \cB\, , \ k = 1, 2, \ldots $
and an initial point $x_0$, a unique solution $\{x_k\}_{k=0}^{\infty}$ is called trajectory of the system. The system is
{\em stable} if $x_k \to 0, \, k \to \infty$, for every trajectory. The role of Lyapunov exponent for discrete systems
is played by the joint spectral radius (JSR) of the set~$\cB$ (see e.g the monograph \cite{J} for an extensive treatise on the JSR).
\begin{defi}\label{d15}
For a given compact set of matrices~$\cB$, the joint spectral radius~$\rho(\cB)$ is
$$
\rho(\cB)\ = \ \lim_{k \to \infty}\max_{B_i \in \cB\, , \, i = 1,\ldots , k}\, \bigl\|B_k\ldots B_1\bigr\|^{\, 1/k}\, .
$$
\end{defi}
This limit exists for every compact set o matrices~$\cB$ and does not depend on the matrix norm~\cite{RS}.
For properties and for more references on numerous applications of JSR see~\cite{GP,GZ3}. The discrete system
is stable if and only if $\rho(\cB) < 1$~\cite{B-discr}. One of the ways to analyse stability
is to discretize the continuous system~(\ref{main}) with dwell time~$\tau > 0$ to the form~(\ref{main-discr})
by setting~$x_k = x(k\tau )\, , \, B = e^{\, \tau A}\, \, A \in \cA$.
This system represents only those trajectories of the continuous system corresponding
to piecewise-constant control functions $A(\cdot)$ with the step size~$\tau$.
Hence, if there exists $\tau > 0$
for which the discrete system is
unstable, i.e., $\rho(e^{\, \tau \cA}) \ge 1$, then the continuous system is also unstable.
We need several properties of JSR that are formulated below. The first one was established in~\cite{RS}:
\smallskip

\noindent 
\textbf{Proposition C}. 
{\em Let~$\cB$ be a compact matrix family and $\lambda$ be a positive number.
If there exists a symmetric convex body~$G$ such that $B (G) \subset \lambda \, G\, ,\ B \in \cB$, 
then $\rho(\cB) \le \lambda$. If $\rho(\cB) < \lambda$, then such a convex body exists.   }


The following property of the joint spectral radius is a special case of~\cite[Proposition 2]{P4}:

\noindent 
\textbf{Proposition D}. 
{\em For every compact set of matrices~$\cB$ and for every point $x_0\in \re^d$
 that does not belong to their proper common invariant linear subspace, there is a constant $C = C(x_0\, \cB) > 0$
 such that $\, \max\limits_{B_i \in \cB\, , \, i = 1, \ldots , k}\, \|B_k\, \cdots \, B_1 \, x_0\| \, \ge \,
 C\, \rho^k\, \, k \in \n$, where $\rho = \rho (\cB)$. }


\subsection{The polytope norm method for discrete-time systems}

Our method of computing of the Lyapunov exponent with a polytope norm is based on
the corresponding method for discrete-time systems developed in~\cite{P0,GWZ,GP}. Below, we give a  short summary of
some results of that work needed in the subsequent sections.

The main idea of the method of JSR computation with a polytope norm is to
find the {\em spectrum maximizing product} (s.m.p.), i.e., a product~$\Pi$ of matrices from~$\cB$ for which the value
$[\rho(\Pi)]^{\, 1/n}$ is maximal among all products of matrices from~$\cB$, where $n = n(\Pi)$ is the length of~$\Pi$.
This is done as follows: first we fix some reasonably large $l \in \n$ and check all products of lengths~$n \le l$ finding
a product~$\Pi$ with the maximal value~$[\rho(\Pi)]^{\, 1/n}$. We denote this value by~$\rho_l$. This product is considered as a candidate for s.m.p. Then
the algorithm iteratively builds a polytope~$P$ for which $\, B\, P \, \subset \, \rho_l P\, , \ B \in \cB$. If it terminates within
finitely many iterations, then the polytope~$P$ is extremal, $\Pi$ is an s.m.p., and $\rho(\cB) = \rho_l$.
A theoretic criterion for termination of the algorithm within finite time is formulated in terms of dominant products.
We consider the normalized family $\tilde \cB \, = \, \{\tilde B = \rho_l^{-1}B | \  B \in \cB\}$. By $\tilde \Pi$
we denote the product of matrices from~$\tilde \cB$ corresponding to~$\Pi$.
\begin{defi}\label{d16}
A product $\Pi\in \cB^n$ is called dominant for the family
$\cB$  if there is $q< 1$ such that the spectral radius of every
product of operators of the normalized family $\tilde \cB$, that is
not a power of~$\widetilde \Pi$ nor a power of its cyclic
permutations, is smaller than $\, q$.
\end{defi}
Thus, any dominant product is an s.m.p., but, in general, not vice versa.
As is shown in~\cite[theorem 4]{GP} the algorithm terminates within finite time if and only if
the product~$\Pi$ is dominant for~$\cB$.

\subsection{Lower and upper bounds for the Lyapunov exponent}

Let $\cA$ be a compact family of matrices. For a given number $\tau > 0$, we define the value $\beta (\cA , \tau ) = \beta (\tau) = \tau^{-1}\ln \rho (e^{\tau \cA})$, and for a given polytope~$P \subset \re^d$ symmetric about the origin,  we define the value $\alpha (\cA , P) \, = \, \alpha (P)$ as
\begin{eqnarray}
\alpha (P) & = & \inf \ \Bigl\{ \alpha \in \re \ \Bigl| \
\mbox{for each vertex}\ v \in P \ \mbox{and} \ A \in \cA, \Bigr.
\nonumber
\\
& & \Bigl. \qquad \mbox{the vector} \ (A - \alpha I)v \  \mbox{is directed inside} \, P\, \Bigr\}\, .
\label{def-alpha}
\end{eqnarray}
The following observation is simple, but crucial for the further results:
\begin{prop}\label{p50}
For an arbitrary compact family~$\cA$, for each number $\tau > 0$ and for a polytope~$P$, we have
\begin{equation}\label{lu}
\beta (\tau) \ \le \  \sigma \ \le \ \alpha (P)\, .
\end{equation}
\end{prop}
 \begin{proof} 
Take arbitrary $\eps > 0$ and a vector $x_0 \in \re^d$ that does not belong to
 a common invariant subspace of~$\cA$. Consider the set of trajectories with~$x(0) = x_0$ corresponding to
 piecewise-constant  control functions~$A(t)$
 with~$k$ steps of size~$\tau$.
By Proposition~D, the maximal value of~$\|x(k \tau )\|$ over such trajectories is
 \begin{equation}\label{twoparts}
 \max\limits_{A_{d_1}, \ldots , A_{d_k} \in \cA}\|e^{\tau A_k}\cdots e^{\tau A_1}x_0\| \ \ge \
 C_1 \, [\rho(e^{\tau \cA})]^k,
 \end{equation}
  where the constant~$C_1$ does not depend on~$k$.
  On the other hand,
$\|x(k \tau) \| \, \le \, C_2 e^{(\sigma + \eps)k \tau}$.
Combining this with~(\ref{twoparts})
and taking the limit as~$k \to \infty$, we get $\, e^{(\sigma + \eps)\tau} \, \ge \,
\rho(e^{\tau \cA})$, which for $\eps \to 0$ yields $\sigma \ge \frac{\ln \, \rho(e^{\tau \cA})}{\tau}= \beta(\tau)$.
Take now some~$\alpha \in \re$. If for every $A \in \cA$ and for each vertex~$v \in P$,
the vector $(A -\alpha I)v$ is directed inside~$P$,
then there is $\eta > 0$ such that $v  + \eta (A -\alpha I)v \in {\rm int} \, P\, , \ A \in \cA$, for each vertex~$v \in P$.
 Hence, the same is true for every convex combination~$x$ of vertices:
 $x  + \eta (A -\alpha I)x \in {\rm int} \, P$, and hence, for every~$x \in \partial \, P$,
 the vector $(A -\alpha I)x$ is directed inside~$P$.
 Consequently, $\sigma (\cA - \alpha I) < 0$, and so $\sigma (\cA) < \alpha$. Taking infimum over all~$\alpha$,
 we arrive at the right hand side inequality of~(\ref{lu}). \qed
\end{proof}

If, for a polytope~$P$, we have $e^{\, \tau \, A}\, P \, \subset \, \lambda \, P\, , \ A \in \cA$,
then Proposition~C yields that $\lambda \ge \rho(e^{\tau \cA})$.
If this inclusion  holds true for $\lambda = \rho(e^{\tau \cA})$,
then the polytope~$P$ is called {\em extremal} for the family~$e^{\tau \cA}$. Clearly, if we have an extremal polytope available, then we know the value of JSR. This is the base of the algorithm exact JSR computation from~\cite{GWZ,GP}.
 In many cases, however, the extremality
property is a too strong requirement, and one can use the following weaker version:
\begin{defi}\label{d17}
Given $\eps \ge 0$,
a polytope~$P$ is called {\em $\eps$-extremal} for a family $e^{\tau \cA}$ if
$$
 e^{\tau A} \, P \ \subset \ e^{\tau \eps } \, \rho (e^{\tau \cA})\, P\, , \qquad  A \in \cA\, .
$$
\end{defi}
Since $\rho (e^{\tau \cA}) \, = \, e^{\tau \beta (\tau)}$, we see that
the  $\eps$-extremalily is equivalent to the inclusion
\begin{equation}\label{varep}
 e^{\tau A} \, P \ \subset \ e^{\tau (\beta + \eps )} \, P\ , \qquad  A \in \cA\, .
\end{equation}

Before we formulate the main results of this subsection, we need to clarify the irreducibility assumption. 
The problem is the irreducibility of the set~$\cA$  does not a priory imply the irreducibility of the set of 
exponents~$e^{\, \tau \cA}$. For some rare cases of the parameter $\tau$, the set~$e^{\, \tau \cA}$ may obtain 
common invariant subspaces. 
To avoid this difficulty, we introduce the notion of admissible numbers~$\tau$. Fix an arbitrary small number $\delta > 0$. 
For a given matrix $A$, we denote
by ${\rm sp}\, (A)$ the set of its eigenvalues. Consider the union of the two following sets:
$$
\left\{ \, \frac{2\pi n}{|{\rm Im}\, (\lambda_i - \lambda_j)|}\, , \  \lambda_i, \lambda_j \in {\rm sp}\, (A)\, , \
{\rm Im}\, (\lambda_i) \ne {\rm Im}\, (\lambda_j) \, , \ n \in \n \, \right\}
$$
and
$$
\left\{ \, \frac{\pi n}{|{\rm Im}\, (\lambda_i)|}\, , \  \lambda_i \in {\rm sp}\, (A)\, , \
\lambda_i\notin \re  \, , \ n \in \n \, \right\}\, 
$$  
and intersect this union with the segment $[0,2]$. We obtain the set~$\cT_0(A)$. 
Let $\cT_0 (\delta, A)$ be some open subset of 
the segment~$[0,2]$ of measure $\delta/2$ that contains $\cT_0(A)$. Finally, 
let $\cT(\delta, A) = \cup_{k\ge 0}2^{-k}\cT_0(\delta, \cA)$.  

Thus, $\cT (\delta, A)$ is an open subset of the segment $[0,2]$ of measure~$\delta$. 
This measure can be chosen arbitrarily small.
If a matrix set~$\cA$ is irreducible, then it has a finite irreducible subset~$\cA_f$. 
If $\cA$ is finite, then we take $\cA_f = \cA$, otherwise we take an arbitrary finite irreducible subset~$\cA_f \subset \cA$.
We fix $\cA_f$ and write $\cT(\delta, \cA)$ for the finite union $\, \cup_{A \in \cA_f} \cT(\delta, A)$. 
A number $\tau \in (0,2]$ is called {\em admissible} for $\cA$ if it does not belong to the set $\cT(\delta, \cA)$.  
Thus, all numbers except for those from a set $\cT(\delta, \cA)$ of arbitrarily small measure are admissible. 
So, a generic number~$\tau > 0$ is admissible.  
 \begin{lemma}\label{l5}
 If a set of matrices~$\cA$ is irreducible, then for any admissible number~$\tau \in (0, 2]$, 
 the set $e^{\, \tau \cA}$ is irreducible. 
 \end{lemma}
\begin{proof} 
If for every~$A \in \cA$, the matrix~$e^{\, \tau A}$ has the same invariant subspaces as~$A$, 
then the family~$e^{\, \tau \cA}$ is irreducible. Otherwise, if some matrix~$A \in \cA$ gets a new invariant subspace after taking its exponent, then either some of its complex eigenvalues~$\lambda_i \in {\rm sp}\, (A)$ becomes real, or 
two its different eigenvalues~$\lambda_i, \lambda_j \in {\rm sp}\, (A)$ become equal. The former means that 
${\rm Im}\, \bigl( e^{\,\tau \lambda_i} \bigr) = \sin \bigl( {\rm Im}\, (\tau \lambda_i)\, \bigr) = 0$ and hence 
$\tau {\rm Im}\, (\lambda_i) = \pi n\, , \, n \in \z$; the latter means that 
$ e^{\,\tau \lambda_i}  =    e^{\,\tau \lambda_j} $, and hence $\tau \, {\rm Im}\, (\lambda_i - \lambda_j) = 2\pi n, \, n \in \z$. In both cases we have $\tau \in \cT (\delta, A)$, and $\tau$ is not admissible. \qed
\end{proof}

In what follows we always assume that the number $\tau$
is admissible, and hence the set $e^{\, \tau \cA}$ is irreducible. 
The double inequality~(\ref{lu}) localizes the Lyapunov exponent to the
segment $[\beta(\tau) \, , \, \alpha(P)]$. The following theorem estimates the length of this
segment in case the polytope~$P$ is $\eps$-extremal.
\begin{theorem}\label{th8}
For every compact irreducible family~$\cA$,  there is a positive constant~$C = C(\cA)$ such that
for all $ \, \eps \ge  0$ and admissible~$\tau \in (0,1)$,  we have
$$
\alpha (P) \ - \ \beta (\tau) \ \le \ C\tau \, + \, \eps\, ,
$$
whenever~$P$ is $\eps$-extremal for~$e^{\tau \cA}$.
\end{theorem}
Thus, by inequality~(\ref{lu}), every dwell time~$\tau > 0$  gives the lower bound~$\beta(\tau)$
for the Lyapunov exponent, and that dwell time with an $\eps$-extremal polytope~$P$ give the lower bound~$\alpha(P)$.
Theorem~\ref{th8} ensures that the precision of these bounds is linear in~$\tau$ and $\eps$, provided 
$\cA$ is irreducible and $\tau$ is admissible. In particular, for $\eps = 0$, 
we have
\begin{cor}\label{c5}
If the polytope~$P$ is extremal for~$e^{\tau \cA}$, then $\alpha (P) - \beta (\tau) \le C\tau$.
\end{cor}

To prove Theorem~\ref{th8} we begin with several auxiliary facts. First,
for an arbitrary compact set of matrices~$\cA$ there is a constants~$C$ such that
\begin{equation}\label{const1}
\Bigl\|\, e^{\, \tau \, A}\, \, - \,  \bigl(\, I \, + \, \tau\, A \, \bigr) \, \Bigl\|
\ < \ C\, \tau^2\, \qquad \mbox{ for every }\ A \in \cA \, , \  \tau \, \in (0, 1)\, .
\end{equation}
For the proof, it suffices to write the Taylor expansion of the matrix exponent and to
estimate the norm of the rest  $\sum_{k = 2}^{\infty}\, \frac{\tau^k}{k!}\, A^k$.

We make use of the following measure of irreducibility suggested by  Kozyakin and Pokrovsky in~\cite{KozP,Koz}.
Denote $\mathbf{O}(x) = {\rm co}\, \bigl\{\pm \, \Pi_k x \ \bigl| \  \Pi_k \in \cA^k \, , \, k = 0, \ldots , d-1\bigr\}$.
This is the symmetrized convex hull of the orbit of point~$x$ by products of length~$ \le k$
of matrices  from~$\cA$. Consider the value
\begin{equation}\label{koz}
h(\cA) \ = \  \max \, \Bigl\{ \, r  \ge 0 \ \Bigl| \  \forall x \in \re^d \, , \, \|x\| = 1 \, , \quad
B(0, r) \subset \mathbf{O}(x)   \  \Bigr\}    \, .
\end{equation}
Thus, $h(\cA)$ is the radius of the biggest Euclidean ball  contained in the convex hull
of the set $\mathbf{O}(x)$, for each
point $x$ from the unit sphere. The following lemma was proved in~\cite{Koz}:
\begin{lemma}\label{l15}
The family~$\cA$ is reducible if and only if~$h(\cA) = 0$.
\end{lemma}
\begin{proof} 
If $h(\cA) = 0$, then by  the compactness it follows that there
exists a point $x\, , \, \|x\| = 1$ for which the set $\mathbf{O}(x)$
has an empty interior.
Consider a sequence $\{L_i\}_{i \in \n}$ of subspaces of $\re^d$ defined recursively as follows:
$L_1 = {\rm span}\, (x)\, , \ L_{i+1} = {\rm span}\, \bigl( L_i\, , \, \cup_{A \in \cA}, AL_i\bigr)\, , \, i \in \n$.
Clearly, this is an embedded sequence, i.e., $L_i \subset L_{i+1}$. If this inclusion is strict
for all $i = 1, \ldots , d-1$, then the dimensions of subspaces strictly increase each time, and hence
${\rm dim}\, L_d \ge  d$. On the other hand, $L_d$ is a linear span of the set~$\mathbf{O}(x)$, therefore, its dimension is smaller than~$d$. Thus, for some $i \le d-1$ we have $L_i = L_{i+1}$, and hence the subspaces $L_n$ coincide for all~$n \ge i$.
In particular, $L_d = L_{d+1}$. Consequently, $AL_d \subset L_d$ for all $A \in \cA$, and so $\cA$ is reducible.

 Conversely, if~$\cA$ has a proper common invariant subspace~$L$, then, for every~$x \in L$ the set  $\mathbf{O}(x)$ lies in~$L$, and hence, does not contain any ball. \qed
\end{proof}

For a given family $\cA$, we consider the class of {\em contractive} norms $\|\cdot \|_c$ in $\re^d$ for which
$\|A\|_c \le 1\, , \ A \in A$. The following result shows that for an irreducible family~$\cA$,
all contractive norms are equivalent.

\begin{lemma}\label{l25}
Let $\cA$ be a compact family of matrices. If a norm~$\|\cdot \|_{c}$ is contractive for~$\cA$,
then, after a multiplication of this norm by a constant, we have
$$
h(\cA)\, \|x\|_c \ \le \ \|x\| \ \le \ \|x\|_c\, , \qquad x \in \re^d\, .
$$
\end{lemma}
\begin{proof}  
Let the minimum of the norm~$\|\cdot\|_c$ on the Euclidean sphere is attained
at a point~$z$. Multiplying the norm~$\| \cdot \|_c$ by a constant, we assume that~$\|z\|_c =1$.
Thus, $\|x\|_c \ge \|x\|$ for all~$x$. Since the norm $\|\cdot \|_c$ is contractive,
it follows that the set~$\mathbf{O}(z)$ is contained in the unit ball of the norm~$\|\cdot \|_c$. By Lemma~\ref{l15}, this
set contains the Euclidean ball of radius~$h(\cA)$. Whence, $\|x\|  \ge h(\cA)\|x\|_c , \, x \in \re^d$. \qed
\end{proof}

Let us denote $H(\cA)$ the minimum of the function~$h(e^{\, t \cA})$ taken over 
admissible points $t \in [1,2]$ . This minimum is attained
at some point~$t_0 \in [1,2]$, because $h(e^{\, t \cA})$ is a continuous function of~$t$
and the set of admissible points~$t \in [1,2]$ is compct. If~$\cA$ is
 irreducible, then so is~$e^{\, t A}$,
 and hence~$H(\cA) = h(e^{\,  t_0 \cA}) > 0$.
 Thus, $H(\cA)$ is strictly positive for irreducible~$\cA$.

For a given $\tau \in (0,1)$, we consider the family~$\, e^{\tau \cA}$. By Lemma~\ref{l25},
all contractive norms of this family are equivalent. We are going to prove that
they are uniformly equivalent for all admissible $\tau \in (0,1)$.

\begin{prop}\label{p20}
Let~$\cA$ be a compact irreducible family of matrices,
then for each admissible~$\tau \in (0,1)$ and for every contractive  norm~$\|\, \cdot \, \|_{\, \tau}$
of the family~$e^{\, \tau \, \cA}$, after a multiplication of this norm by a constant, we have
\begin{equation}\label{norm1}
\, H(\cA)\ \| x \|_{\, \tau} \quad \le \quad \|x \| \quad \le \quad \, \|x\|_{\, \tau}\, ,
\qquad  x \in \re^d\, .
\end{equation}
\end{prop}
\begin{proof} 
There is a natural number~$n$ such that~$2^n\tau \in [1,2]$.
By definition, the number~$2^n \tau$ is admissible, hence the set~$e^{\, 2^n\tau \cA}$ is irreducible. 
Since~$\|e^{\, \tau A}\|_{\tau} \le 1$ for every~$A \in \cA$, we have~$\|e^{\, 2^n \tau A}\|_{\tau} =  \|(e^{\, \tau A})^{\, 2^n}\|_{\tau}\le 1$, hence the norm $\|\cdot \|_{\tau}$ is contractive for the family~$e^{\, 2^n \tau \cA}$ as well.
Applying Lemma~\ref{l25} to this family, we obtain
$$
h\, (e^{\, 2^n \tau \cA})\, \|x\|_{\tau} \ \le \ \|x\| \ \le \ \|x\|_{\tau}\, .
$$
It remains to note that $h\, (e^{\, 2^n \tau \cA}) \ge H(\cA)$, because $2^n \tau \in [1,2]$. \qed
\end{proof}

Thus, all contractive norms of the families~$e^{\, \tau \, \cA}$ are equivalent, uniformly for all
admissible $\tau \in (0,1)$, to the Euclidean norm. 

\subsection*{Proof of Theorem~\ref{th8}.}  

Without loss of generality, passing from the family~$\cA$ to $\cA - \beta I$, it can be assumed that
$ \beta  (\tau) = 0$. We estimate
$\alpha (P)$ from above by showing the existence of a constant~$C$ depending only on the family~$\cA$
such that for every vertex~$v \in P$ and for each~$A \in \cA$,
the vector~$(\, A \, - \, (C\tau + \eps)\, I)\, v$ is directed inside~$P$. This will imply $\alpha \le C\tau + \eps$.

Denote~$\cA' = \cA - \eps I$ and consider arbitrary~$A' \in \cA'$.
By~(\ref{const1}) the distance between points~$e^{\, \tau A'}v$ and $(I + \tau A')v$
is smaller than~$C\tau^2$. This distance can be measured in the norm~$\|\, \cdot\, \|_P$, where
it is smaller than~$C\tau^2$, with another constant~$C$, which depends neither on~$\tau$ nor on~$P$. Indeed,
by the $\eps$-extremality assumption, the norm~$\|\, \cdot\, \|_{P}$ is contractive for~$e^{\, \tau \cA'}$,
and hence, by Proposition~\ref{p20}, is equivalent to the Euclidean norm, uniformly in~$\tau \in (0,1)$.
Since for each vertex~$v \in P$, we have $\|\, e^{\, \tau A'}v\, \|_{P} \, \le \, \|v\|_P = 1$,
from the triangle inequality it follows that 
$$\|(I + \tau A')v\|_{P}\, < \, 1 \, + \, C\tau^2.$$
Therefore, the point \ $\, y \, = \, 1/(1+C\tau^2)\, (I + \tau A')\, v\, $ belongs to~${\rm int}\, P$,
and hence the vector from the point~$v$ to~$y$ is directed  inside~$P$.
On the other hand, $\, y - v \, = \, \tau/(1 + C\tau^2)\, \bigl( A' \, - \, C\tau \, I \, \bigr)\, v$,
consequently the vector~$( A'  -  C\tau I ) v = ( A  -  (C\tau + \eps) I ) v$  is directed  inside the polytope,
which concludes the proof.
\qed

\subsection{Algorithm~(R) for computing the Lyapunov exponent
 and for constructing the polytope Lyapunov function}

Proposition~\ref{p50} and Theorem~\ref{th8} suggest the following method of approximate computation
of the Lyapunov exponent $\sigma(\cA)$:
\smallskip

1) choose a dwell time $\tau >$, and compute the joint spectral radius $\rho(e^{\, \tau \cA})$;

2) choose $\eps > 0$ and construct an $\eps$-extremal polytope~$P$ for the family~$e^{\, \tau \cA}$.
\smallskip

Then we localize the Lyapunov exponent $\sigma (\cA)$ on the segment $[\beta , \alpha]$, whose length
tends to zero  with a linear rate in~$\tau$ and $\eps$ as $\tau , \eps \to 0$.

For a given finite  irreducible family~$\cA = \{A_1, \ldots , A_m\}$ or for a polytope
family of matrices ${\rm co} \, (\cA)$, the algorithm approximates the Lyapunov exponent~$\sigma(\cA)$
by giving its lower and upper  bounds, and produces the corresponding polytope Lyapunov norm.
The main idea is the iterative construction of an $\eps$-extremal
polytope for a prescribed $\eps >0$. We begin with a brief description of the algorithm.
\smallskip

{\em Initialization}.  We choose a small admissible number $\tau > 0$, a small number $\nu \ge 0$,  and a reasonably large natural~$l$.
Among all products of length~$\le l$ of the matrices $e^{\, \tau A_j}\, , \ j = 1, \ldots , m$,
we select a {\em starting} product
$\Pi \, = \, e^{\, \tau A_{d_n}}\cdots e^{\, \tau A_{d_1}}\, , \ n \le l$ for which the
value $[\rho (\Pi)]^{1/n}$ ($n$ is the length of~$\Pi$) is as large is possible.
This can be done by mere exhaustion of the set of all products of lengths~$n = 1, \ldots , l$,
or by Gripenberg's algorithm~\cite{G},
or by the recent algorithm \cite{CicPro}. We denote~$\rho_l = [\rho(\Pi)]^{1/n},
\, \beta_l(\tau) = \tau^{-1} \ln \rho_l$. If we performed a complete exhaustion and the
value~$\rho_l = [\rho(\Pi)]^{1/n}$ is maximal among all products of lengths at most~$\Pi$, then we call
$\Pi$ the {\em maximal} product.
Observe that $\beta_l \le \beta$ and if we take only maximal products, then $\beta_l \to \beta$ as $l \to \infty$.
If the leading eigenvalue~$\lambda_{\max}$
of~$\Pi$ is real, we assume
it is positive, the case of a negative
eigenvalue is considered in the same way. In this case we denote by $v_1$ the leading eigenvector of~$\Pi$
(if it is not unique, we take any of them). If $\lambda_{\max}\notin \re$, then we set $v_1 = v + \bar v$,
where $v$ is the leading eigenvector. Then we normalize the family $\cA$ as
\begin{equation}
\tilde \cA = \cA - (\beta_l + \nu)I
\label{eq:normalized}
\end{equation}
where we recall that $\nu$ is a suitable small positive number.
\smallskip

{\em The first part. Construction of the polytope~$P$.}
We start with the set $\cV_0 = \{v_1, \ldots , v_n\}$ and with the corresponding polytope
(possibly, not full-dimensional) $P_0 = {\rm co}_s (\cV_0)$, where
$v_i = e^{\, \tau A_{d_{i-1}}}\cdots e^{\, \tau A_{d_{1}}}v_1$. At the $k$th step, $k \ge 1$,
we have a finite set of points ${\mathcal V}_{k-1}$ and a polytope $P_{k-1} = {\rm co}_s ({\mathcal V}_{k-1})$.
We take a point~$v \in {\mathcal V}_{k-1}$ added in the previous step, and for each $j = 1, \ldots , m$,
check if $e^{\tau \tilde A_j}v \in P_{k-1}$, by solving the corresponding LP problem. If
the answer is affirmative, then we go to the next $j$; if $j=m$, we go to the next point from~${\mathcal V}_{k-1}$.
Otherwise, if $e^{\tau \tilde A_j}v \notin P_{k-1}$, we update the set ${\mathcal V}_{k-1}$ by adding the point $v$, update, respectively,
the polytope $P_{k-1}$ by adding two new vertices $\{v, -v\}$,
and then go the next~$j$ and to the next point~$v$. After we exhaust all points
of the initial set~${\cV}_{k-1}$, we start the $(k+1)$st step, and so on. This process
terminates after some~$N$th step, when $\mathcal{V}_N = {\cV}_{N-1}$, i.e., no new vertices are added
to the polytope~$P_{N-1}$. In this case, $e^{\tau \tilde A}P_{N}\subset P_N$ for all $\, \tilde A \in \cA - (\beta_l + \nu)I$.
This means that~$P_N$ is $\eps$-extremal for the family~$\cA$ with
\begin{equation}\label{vareps}
\eps \quad = \quad \beta_l(\tau) \ - \ \beta(\tau) \ + \ \nu\, .
\end{equation}
If the first part of the algorithm does not terminate within finite time, then we either increase~$l$ or increase~$\nu$ and go
to the Initialization.
\smallskip

{\em The second  part. Deriving the lower and upper bounds for $\sigma (\cA)$}.
Thus, the first part of the algorithm produces an $\eps$-extremal polytope~$P_N$.
We compute $\alpha(P_N)$ by definition, as infimum of numbers~$\alpha$ such that
the vector $(A - \alpha I)w $ is directed inside $P_N$, for each vertex $w \in P_N$ and for every $A \in \cA$.
This is done by taking a small~$\delta> 0$ and solving the following LP problem:
$$
\left\{
\begin{array}{l}
\alpha \ \to \ \inf \\
w + \delta (A - \alpha I)w \, \in \, P_N,\\
\, w \in P_N, \ A \in \cA\, .
\end{array}
\right.
$$
 Proposition~\ref{p50} yields
\begin{equation}\label{result}
\beta_l(\tau) \ \le \ \sigma(\cA) \ \le \ \alpha(P_N)\, .
\end{equation}
Thus, as a result of Algorithm~(R), we obtain the lower bound $\beta_l(\tau)$ and
the upper bound~$\alpha(P_N)$ for the Lyapunov exponent.
The polytope~$P_N$ constitutes the Lyapunov norm for the family~$\cA$.
If $\alpha (P_N) < 0$, then we conclude that the system is stable and its joint Lyapunov function
is defined by the polytope~$P_N$.

\begin{algorithm}
\DontPrintSemicolon
\KwData{$\cB_{\dtt} = \e^{\dt \tilde \cA}$ (see (\ref{eq:normalized}))\; 
The starting product $\Pi$ (candidate s.m.p.) of length $n \le l$ for which the value $[\rho(\Pi)]^{1/n}$ is as large as possible for all products of length at most~$l$.
}
\KwResult{the polytope $\cP_{\dtt}$}
\Begin{
\nl Compute the leading eigenvector of $\Pi$ and of~$n$ its cyclic permutations.
We obtain a system of vectors $\{v_1,\ldots,v_n\}$ \;
\nl Set $\cV_0 := \{ v_j \}_{j=1}^{n}$ and $\cR_0 = \cV_0$\;
\nl Set $i=0$\;
\nl Set ${\rm term} = 0$\;
\nl \While{${\rm term} \neq 1$} {
\nl $\cW_{i+1} = \cB_{\dtt}\, \cR_{i}$\;
\nl Set $\cS_i = \emptyset$\;
\nl Set $m_i = {\rm Cardinality}(\cW_{i+1})$\;
\nl Set $n_i = {\rm Cardinality}(\cV_i)$\;
\For{$\ell=1,\ldots,m_i$}{
\nl Let $z$ the $\ell$-th element of $\cR_i$\;
\nl Check whether  $z \in {\rm co}_{s} (\cV_i)$, $\cV_i := \{ v_j \}_{j=1}^{n_i}$, i.e. solve the LP problem\;
$
\begin{array}{rcl}
\min & & f_\ell=\sum\limits_{j=1}^{n_i} \left( \lambda_j + \mu_j \right)
\\[0.25cm]
{\rm subject \ to}
     & & \sum\limits_{j=1}^{n_i} \left( \lambda_j v_j  + \mu_i (-v_j)
     \right) \!=\! z
\\[0.25cm]
{\rm and}
     & & \lambda_j \ge 0, \ \mu_j \ge 0, \quad
     j=1,\ldots,n_i.
\end{array}
$\;
\nl \If{$f_\ell > 1$}{
\nl $\cS_i = \cS_i \cup z$\;
}
}
\eIf{$\cS_i = \emptyset$}{
\nl Set ${\rm term} = 1$\;}{
\nl $\cR_i = \cS_i$\;
\nl $\cV_{i+1} = \cV_{i} \cup \cR_{i}$\;}
\nl Set $i=i+1$\;
}
\nl Set $N=i$\;
\nl Return $\cP := \cP_{N} = {\rm co}_{s} (\cV_N)$ (extremal polytope)\;
}
\caption{Algorithm (R), part 1. Constructing an $\eps$-extremal polytope~$P$ \label{algoR}}

\end{algorithm}

\smallskip

\begin{algorithm}
\DontPrintSemicolon
\KwData{$\cA, \cP_{N}, \cV_{N}$ (system of vertices of $\cP_N$), $\delta$ (small positive stepsize)}
\KwResult{$\alpha$}
\Begin{
\For{$i=1,\ldots,m$}{
\nl Solve the LP problems (w.r.t. $\{t_v,s_v\}$, $\alpha_i$)
\begin{eqnarray}
\label{LP.r}
\begin{array}{rl}
\min & \alpha_i
\\[0.1cm]
{\rm s.t.}
     & w + \delta (A_i - \alpha_i I) w \le \sum\limits_{v \in \cV_{N}}\,
		 t_v\,v - s_v\, v \quad \forall w \in \cV_{N}
\\[0.1cm]
{\rm and}
     & \sum\limits_{v \in \cV_{N}}\, t_v + s_v \le 1, \qquad t_v, s_v \ge 0 \quad
		   \forall v \in \cV_{N}
\end{array}
\nonumber
\end{eqnarray}
}
\nl Return $\alpha(\cP_N) := \max\limits_{1 \le i \le m} \alpha_i$ \;
}
\caption{Algorithm~(R), part 2.  Computing  the best upper bound~$\alpha(P)$ \label{algoR2}}
\end{algorithm}

\begin{theorem}\label{th40}
Algorithm~(R) terminates within finite time if one of the following conditions is satisfied:

1) $\nu \, > \, \beta(\tau)\, - \, \beta_l(\tau)$;

2) $\nu = 0$, the product $\Pi$ is dominant for the family $e^{\tau \cA}$ and
its leading eigenvalue $\lambda_{\max}$ is unique and simple.

\noindent In case 1) the distance between the lower and upper bounds in~(\ref{result})
does not exceed  $\beta_l - \beta + \nu + C\tau$; in case 2) it does not exceed $C\tau$,
where $C = C(\cA)$ is a constant independent of $\tau$.
\end{theorem}
\begin{proof}
In the case 1) we have $\rho(e^{\tau \tilde A}) = e^{ (\beta - \beta_l - \nu)\tau} < 1$.
Hence, products of matrices from the family $e^{\tau \tilde A}$ tend to zero as their lengths
tend to infinity.  Therefore, for each~$r > 0$, there is $k = k(r)$ such that all points $v$
appearing in the $k$th step of Algorithm~(R) are inside the ball~$B(0, r)$. On the other hand,
the family~$e^{\tau \tilde A}$ is irreducible, and hence, for every $k \ge d$, the polytope $P_{k-1}$ 
has a nonempty interior, i.e., contains some ball~$B(0,r)$. This means that all points~$v$ generated in 
$k$th step ($k = k(r)$) are inside~$P_{k-1}$, i.e., the first part of Algorithm~(R) terminates within finite time.

In case 2) we have $\rho(e^{\tau \tilde A}) = 1$, the spectrum maximizing product~$\Pi$ is dominant,
and its leading eigenvalue~$\lambda_{\max}$ is unique  and simple. By theorem~4 of~\cite{GP} Algorithm~(R) terminates
within finite time.

Since the polytope~$P_N$ is $\eps$-extremal for $\eps = \beta_l -  \beta +  \nu$,
the upper bound for the difference $\alpha(P_N) - \beta_l(\tau)$ follows from Theorem~\ref{th8}. \qed
\end{proof}

\begin{cor}\label{c50}
If the starting product $\Pi$ is always maximal (i.e. for all $\tau$), then
the distance between the lower and upper bounds in~(\ref{result}) tends to zero as
$l\to \infty$ and $\nu \to 0\, , \, \tau \to 0$.
\end{cor}
\begin{proof} 
Since for maximal products, we have $\beta_l \to \beta$ as $l \to \infty$, the corollary follows
by applying Theorem~\ref{th40}. \qed
\end{proof}

\subsection{An illustrative example in dimension $2$.}
\label{ex:simpleG}

Let ${\cA}=\{ A_1, A_2 \}$ with 
\begin{eqnarray*}
A_1 & = & \left(\begin{array}{rr}
   0.34657\ldots &  0.78539\ldots \\
  -0.78539\ldots &  0.34657\ldots
\end{array} \right)
\\
A_2 & = & \left(\begin{array}{rr}
   0.60459\ldots &  1.20919\ldots \\
  -1.20919\ldots & -0.60459\ldots
\end{array} \right).
\end{eqnarray*}
For $\dt = 1$ we set $\tilde\cB = \{ \tilde{B}_1, \tilde{B}_2 \} = \{ e^{A_1}, e^{A_2} \}$ with
\begin{eqnarray*}
\tilde{B}_1 = \left(\begin{array}{rr}
 1 & 1
\\
-1 & 1
\end{array} \right), \quad
\tilde{B}_2 = \left(\begin{array}{rr}
 1 & 1
\\
-1 & 0
\end{array} \right)
\end{eqnarray*}
i.e. $A_1=\logm(\tilde{B}_1)$ and $A_2=\logm(\tilde{B}_2)$.

By means of Algorithm (R) we are able to prove that the product of degree equal to $7$,
$P = \tilde{B}_1^2\,\tilde{B}_2\,\tilde{B}_1^3\, \tilde{B}_2$ is spectrum maximizing, so that
$\rho(\tilde{\cB}) = \rho(P)^{1/7} = 13.65685424\ldots$, giving the lower bound
$\beta=0.373463076\ldots$.
Then we set $\cB = e^{\cA - \beta I} = \{ B_1, B_2 \}$ 
with $B_1 = \tilde{B}_1/\rho(\tilde{\cB}), B_2 = \tilde{B}_2/\rho(\tilde{\cB})$
and apply Algorithm (R), part 1. 
 
As a result we obtain the polytope norm in Figure \ref{fig1} whose unit ball $\cP_\dt$ is a polytope with $16$ vertices.
\begin{figure}[ht]
\centering
\global\def\path{#1}\input{exgen.inp} \\[2mm]
      \caption{Polytope norm for the illustrative example with $\dt=1$. In red the vectors
			$B_1 v$ and in blue the vectors $B_2 v$, for $v \in V_{\dt}$,
			vertices of $\cP_{\dt}$.}
      \label{fig1}
\end{figure}

\begin{figure}[ht]
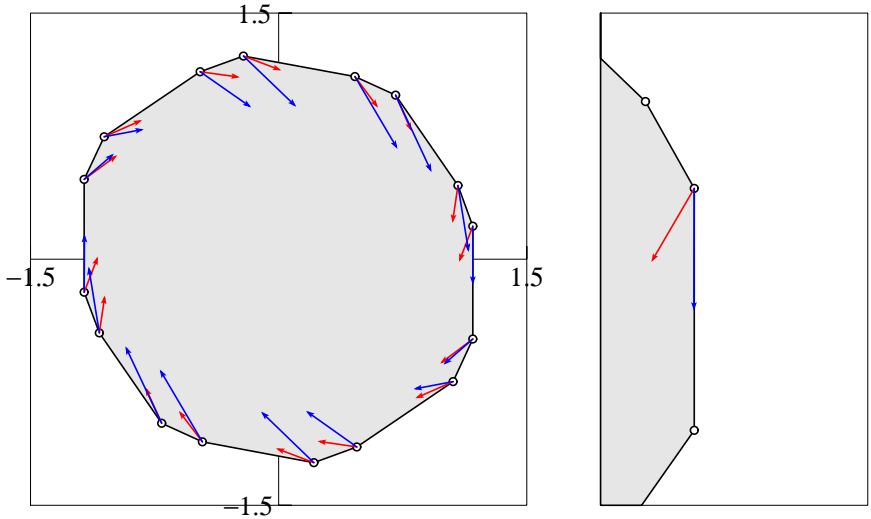

\centering
\global\def\path{#1}\input{exgen2.inp} \hskip 0.5cm \global\def\path{#1}\input{exgen3.inp} \\[2mm]
      \caption{Left picture. In red the vectors $(A_1-\alpha I) v$ and in blue the vectors $(A_2-\alpha I) v$,
			for $v \in V_{\dt}$, vertices of $\cP_{\dt}$. Right picture: zoom of the vectorfield (in blue) tangent to 
			the boundary of the polytope}
      \label{fig1b}
\end{figure}

Applying Algorithm (R), part 2, we obtain the optimal shift $\gamma = 0.433445\ldots$ so that we have the estimate
$$
\beta = 0.373463076\ldots \le \sigma \le 0.80690807\ldots = \alpha.
$$
Figure \ref{fig1b} illustrates the fact that the computed polytope $\cP_\dt$ is positively invariant for the shifted family
$\cA - \alpha I$. 

As we expect $\alpha   = 0.80690807\ldots$ cannot be improved since one of the vectorfields $(\cA - \alpha I) v$ is tangential to the boundary of the polytope (see Figure \ref{fig1b} (right)), according to the fact that we have solved the optimization problem in Algorithm (R), part 2.

Note that we can easily increase the accuracy of the approximation.
For example, choosing the smaller dwell time $\dt=1/8$ we obtain a polytope with $80$ vertices which gives the following interval
$$
\beta = 0.385225559\ldots \le \sigma \le 0.438159379\ldots = \alpha.
$$

\section{Stability of positive linear switching systems}

In this section we analyze positive continuous-time LSS. Usually, they are defined in the
literature as systems with all trajectories~$x(t)$ in the positive orthant~$\re^d_+$, provided~$x(0) \in \re^d_+$.
This is equivalent to say that all matrices~$A \in \cA$ are Metzler, i.e., all off-diagonal entries of~$A$ are nonnegative.
Such LSS are applied, for example, in the consensus problem of multiagent systems and in cooperative systems.
Their properties have been thoroughly analyzed in the literature, see \cite{AR,FM,FV2,SH} and references therein.

For the sake of generality, we consider LSS  that are positive with respect to an arbitrary cone~$K \subset \re^d_+$, rather than the special case~$K = \re^d_+$. For criteria on LSS to be positive with respect to some cone and for special properties of such systems,
see~\cite{EMT,P2,RSS,Van,W}.
To avoid confusions with the standard notation for positive LSS (positive matrix, Metzler matrix, etc.)
that are used in the literature in the case $K = \re^d_+$, in this section we deal with families of linear operators
instead of families of matrices.
Only in the  case $K = \re^d_+$, we assume the basis in~$\re^d$ to be fixed, and deal with corresponding matrices
 $A \in \cA$. First of all, we
formulate and prove Theorem~\ref{th10} on the existence of a monotone invariant norm for a positive system.
This result strengthens Theorem~A for systems positive with respect to a cone~$K$: it relaxes the irreducibility assumption
for operators from~$\cA$ and states the monotonicity of the invariant norm. Then we use this fact to establish special analogues of Theorem~B and of  Theorem~\ref{th8} for positive systems. This enables us to
derive a modification of Algorithm~(R) for positive LSS, which works
more efficiently and under  weaker assumptions.

\subsection{Invariant cones and $K$-Metzler operators}

We begin with extending well-known notions and results on positive systems to the case of arbitrary cone~$K$,
then we formulate the main result of this subsection, Theorem~\ref{th10}.

Let $K \subset \re^d$ be a cone. In the sequel every cone is assumed to be convex, closed, solid, pointed, and
with an apex at the origin. The dual cone~$K^*$ is defined in a standard way:
\begin{equation}\label{dual}
K^* \quad = \quad \bigl\{\, y \in \re^d \quad \bigl| \quad \inf_{x \in K}(y, x) \, \ge \, 0 \, \bigr\}\, .
\end{equation}
By $\partial_K \, M$ and ${\rm int}_K\, M$ we denote the boundary and the interior respectively
of a set $M \subset K$ in the topology of the cone~$K$.
\begin{defi}\label{d10}
Let a cone $K \subset \re^d$ be given. A linear operator $A$ in $\re^d$ is called Metzler with respect to $K \subset \re^d$
(or, in short notation, $K$-Metzler) if there is $h > 0$ such that $(I+hA)K \subset K$.
\end{defi}
 A vector $x$ is $K$-nonnegative ($x \, \ge_K \, 0$) if it belongs to this cone, and an operator $A$ is $K$-nonnegative ($A \, \ge_K \, 0$) if it leaves the cone $K$ invariant.
If $I+hA \, \ge_K \, 0$, then $I+tA\, \ge_K \, 0$ for all $t \in (0, h]$. Indeed,
$\, I + tA \, = \, \frac{h-t}{h}\, I\, + \,    \frac{t}{h}\bigl(I+hA \bigr)\, \ge_K \, 0$,
since the both terms are $K$-nonnegative.
In the sequel of this section we assume a cone $K$ to be fixed, and write ``nonnegative'' and ``Metzler''
instead of ``$K$-nonnegative'' and ``$K$-Metzler'' respectively.
We start with two simple lemmas that are
well-known for the case $K = \re^d_+$.

\begin{lemma}\label{l10}
Let $K$ be an arbitrary cone. If an operator $A$ is Metzler, then the operator $e^{\, tA}$ is nonnegative for every $t > 0$.
\end{lemma}
\begin{proof} 
We have $e^{\, tA}\,= \, \lim\limits_{n \to \infty}\bigl(I \, + \, \frac{t}{n}\, A \bigr)^n \, \ge_K \, 0$,
because $\, I \, + \, \frac{t}{n}\, A\, \ge_K\, 0$ for all large $n$. \qed
\end{proof}

\begin{lemma}\label{l20}
Let $K$ be an arbitrary cone. If a compact set $\cA$
consists of Metzler operators, then for every trajectory of~(\ref{main}) such that $x_0 \in K$
we have $x(t) \in K\, , \ t \ge 0$.
\end{lemma}
\begin{proof}  
Fix an arbitrary $t > 0$. Every control function
can be approximated on the segment $[0,t]$
by piecewise-constant functions $ A^{(n)}(\cdot)$ with the nodes $\bigl\{\frac{kt}{n}\, , \, k = 1, \ldots , n-1\bigr\}$
so that $\|A^{(n)} - A\|_{L_1[0,t]} \to 0$ as $n \to \infty$. Whence, $\|x^{(n)} - x\|_{C[0,t]}\, \to \, 0$
as $n \to \infty$. Since
$$
 x^{(n)}(t)\  = \ e^{\frac{t}{n}\, A\bigl(\frac{t(n-1)}{n}\bigr)}\cdots
e^{\frac{t}{n}\, A\bigl(\frac{t}{n}\bigr)}e^{\frac{t}{n}\, A\bigl(0\bigr)}\, x_0
$$
and all the exponents in this product are $K$-nonnegative (Lemma~\ref{l10}), it follows that $x^{(n)}(t) \in K$.
The limit passage as $n \to \infty$ concludes the proof. \qed
\end{proof}

\begin{cor}\label{c10}
If a compact set $\cA$
consists of Metzler operators, then for every control function $A(\cdot)$
inequality $y_0 \ge_K x_0$ implies $y(t) \ge_K x(t)$ for every $t \ge 0$.
\end{cor}
\begin{proof} 
Is by applying Lemma~\ref{l20} to the function $y(t) - x(t)$. \qed 
\end{proof}

For $K$-positive families, the irreducibility condition imposed in the main results of Section~2 can be
relaxed to {\em $K$-irreducibility.}  Let us first introduce some further notation.
A {\em face} of a cone~$K$ is the intersection of~$K$ with a hyperplane passing through the apex.
The apex is a face of dimension~$0$, this is a {\em trivial} face, all others are nontrivial.
All generatrices are faces of dimension~$1$. A {\em face plane} is a linear span of a face.

\begin{defi}\label{d20}
A $\, K$-Metzler operator is called irreducible with respect to~$K$ (in short, $K$-irreducible, or positively irreducible,
if the cone $K$) is fixed
if it has no invariant subspace among the nontrivial face planes of~$K$.
A family of $K$-Metzler operators~$\cA$ is $K$-irreducible if there is no nontrivial face plane of~$K$ invariant for all operators from~$\cA$.
\end{defi}
Thus, the $K$-irreducibility property is much weaker than just irreducibility. A $K$-irreducible family of operators
may have common invariant subspaces, but not among the proper faces of~$K$. In particular, in dimension~$d \ge 3$
there are no irreducible operators, while $K$-irreducible ones, of course, exist.
The $K$-irreducibility may be verified by the
following simple criterion.
\begin{prop}[~\cite{Van}] \label{p15}
A Metzler operator is irreducible with respect to a given cone~$K$ if and only if it does not have
eigenvectors on the boundary of~$K$.
\end{prop}
\begin{remark}\label{r20}
{\rm In case~$K = \re^d_+$, the $K$-irreducibility, or positive irreducibility,
 means that the operators have no common invariant coordinate
subspaces (subspace spanned by several vectors of the canonical basis), or, which is the same, the matrices are not similar via a permutation to  block upper triangular matrices (with more than one block).}
\end{remark}

\begin{defi}[monotone norm]
A norm on a cone~$K$ is called {\em monotone} if for every $x, y \in K$ the inequality
 $x \ge_K y$  implies $\|x\| \ge  \|y\|$. 
\label{def:monorm}
\end{defi}

The aim of this subsection is to sharpen Barabanov's theorem
 for Metzler operators with a cone~$K$. We prove that in this case there exists an invariant
 norm that is  {\em monotone} with respect to~$K$. Moreover, the irreducibility
 assumption can be now weakened to $K$-irreducibility. Thus, even if the operators
 share common invariant subspaces, they have an invariant norm, unless one of those subspaces
 is a face plane for~$K$. The proof of the first assertion (monotonicity) is rather simple, it can be derived from
 Barabanov's theorem.  The second part (relaxing the irreducibility condition) is more delicate.
 To realize it we need actually to derive an independent proof, not relying on Theorem~A, although
 using some ideas of its proof.

 The extremal norm on a cone $K$ for $K$-Metzler operators is defined in the same way as
 in Definition~\ref{d5}. The only difference is that now we consider only those trajectories
 starting in the cone~$K$ (and hence, entirely lying in~$K$). The definition of invariant norm
 on a cone~$K$ also stays the same, we only write
 $x_0 \in K$ instead of $x_0 \in \re^d$.
 \begin{theorem}\label{th10}
 Every  $K$-irreducible set of Metzler operators possesses an invariant monotone norm on the cone~$K$.
 \end{theorem}
 \begin{remark}\label{r30}
{\rm This fact, in comparison with Barabanov's theorem (Theorem~A) applied to positive operators, has two advantages:
it ensures the existence of a} $K$-monotone {\em invariant norm and, which is more important, it relaxes the assumptions on the set of operators to $K$-irreducibilty.}
\end{remark}

 The proof of Theorem~\ref{th10} is fairly technical. It is placed in Appendix and split into four steps.
 In the first two steps we construct an extremal norm on~$K$, using the compactness argument and involving the
 $K$-irreducibility assumption. In the last two steps we use convex optimal control theory to show the existence of a
 generalized trajectory on the unit sphere, which means that this extremal norm is invariant.


Theorem~\ref{th10} implies, in particular, an analogue of Theorem~B for $K$-positive systems.
To formulate it we need to extend definitions of some notation from Section~2 to this case.

A  convex set $G \subset K$ is called {\em monotone} with respect to~$K$ if
$x \in G, y \le_K x \, \Rightarrow \, y \in G$.
 For a given monotone convex set $G \subset K$, we say that
the vector $Ax$  at the point $x \in G$ {\em is directed inside} $G$, if
there is a number $\eta > 0$ such that $x \, + \, \eta \, Ax \, \in \, {\rm int}_K\, G$.
The proof of the following fact is the same as the proof of Theorem~B
(applying Theorem~\ref{th10} instead of Theorem~A), and we omit it
\begin{prop}\label{p30}
A family of $K$-Metzler operators~$\cA$ is stable if and only if
 there exists a convex monotone body~$G \subset \re^d$ such that
 at every point~$x \in \partial_K \, G$ the vector $Ax$ is directed inside~$G$, $A \in \cA$.
\end{prop}

\subsection{Monotone polytopes and corresponding Lyapunov functions}

 Let $K$ be a cone. For a given set $M \subset K$,
  we denote its {\em monotone convex hull} as
\begin{equation}
{\rm co}_{-}(M)\ = \ \bigl( \, {\rm co}(M)\, - \, K\, \bigr)\, \cap \, K\ = \
\bigl\{ \, x \, \in \, K\ \bigl| \ x = y \, - \, z\, , \ y \, \in \,
{\rm co}\, (M)\, , \   z \in K\bigr\}\,
\label{eq:cominus}
\end{equation}
A monotone convex hull of a finitely many points is called a {\em monotone polytope} or a {\em $K$-polytope}.
Each of those points is a vertex of~$P$, unless it is in a monotone convex hull of the
remaining points. Thus, a monotone polytope is a monotone convex hull of its vertices.
In contrast to usual polytopes, a monotone polytope may have one vertex and be full-dimensional.

Every monotone polytope defines a monotone norm on~$K$. We use this norm as a joint
Lyapunov function of operators from~$\cA$ on the cone~$K$, which gives us bounds
for the Lyapunov exponent. Those bounds are similar to
those defined in subsection~2.2. We use the same lower bound~$\beta (\tau) = \tau^{-1}\ln \rho (e^{\tau A})$.
The upper bound $\alpha(P)$ is also defined in the same way, by formula~(\ref{def-alpha}), but
only for a monotone polytope~$P$.

\begin{prop}\label{p50-}
For an arbitrary compact family~$\cA$ of $K$-Metzler operators, for each number $\tau > 0$ and a
monotone polytope~$P$, we have
\begin{equation}\label{lu-}
\beta (\tau) \ \le \  \sigma \ \le \ \alpha (P)\, .
\end{equation}
\end{prop}
\begin{proof} 
The lower bound has already been proved in Proposition~\ref{p50}.  The
upper bound needs a proof, because~$P$ is not a (usual) convex hull of its vertices any more.
Take an arbitrary~$\alpha \in \re$. If for every~$A \in \cA$ and for each vertex~$v \in P$,
 the vector $(A -\alpha I)v$ is directed inside~$P$,
then there is $\eta > 0$ such that $v  + \eta (A -\alpha I)v \in {\rm int}_K \, P\, , \ A \in \cA$,
for each vertex~$v \in P$.
Rewriting this inclusion in the form $\eta (A  + (\eta^{-1} -\alpha) I)v \in {\rm int}_K \, P$,
we see that the operator $A  + (\eta^{-1} -\alpha )I$ is $K$-positive, whenever $\eta$ is small enough.

Hence, this inclusion holds for every convex combination~$x$ of vertices of~$P$, and
for all points $y \le_K x$, i.e., for all $y \in P$.
Thus, $y  + \eta (A -\alpha I)y \in {\rm int} \,P$,
and therefore, the vector $(A -\alpha I)y$ is directed inside~$P$, for every~$y \in \partial_K \, P$.
Proposition~\ref{p30} yields $\, \sigma (\cA - \alpha I) < 0$, and so $\sigma (\cA) < \alpha$. \qed 
\end{proof}

The notions of extremal and  {\em $\eps$-extremal} polytope are extended
to monotone polytopes in a straightforward manner.

\begin{theorem}\label{th8-}
For every compact irreducible family~$\cA$ of Metzler operators,  there is a constant~$C$ such that
for all~$\tau \in (0,1)$ and $ \, \eps > 0$ we have
$$
\alpha (P) \ - \ \beta (\tau) \ \le \ C\tau \, + \, \eps\, ,
$$
whenever~$P$ is $\eps$-extremal monotone polytope for the family $e^{\tau \cA}$.
\end{theorem}
The proof is actually the same as for Theorem~\ref{th8}, but with the use of modified
parameter of irreducibility $h_K(\cA)$. This value is defined for an arbitrary family~$\cA$ of $K$-positive operators
as follows:
 \begin{eqnarray}
h_K(\cA) & = & 
\max \, \Bigl\{ \, r  \ge 0 \, \Bigl| \,  \forall x \in K \, , \, \|x\| = 1 \, , \, \Bigr.
\nonumber
\\
& &
\Bigl. B_K(0, r) \subset {\rm co}_- \, \{ \Pi_k x \  | \
\Pi_k \in \cA^k \, , \  k = 0, \ldots , d-1\}  \,  \Bigr\}    \, .
\label{koz-}
\end{eqnarray}
\begin{lemma}\label{l15-}
A family~$\cA$ of $K$-positive operators is  $K$-reducible if and only if~$h_K(\cA) = 0$.
\end{lemma}
\begin{proof} 
If $h_K(\cA) = 0$, then by compactness it follows that there
exists a point $x \in K\, , \, \|x\| = 1$ for which
 the set ${\rm co}_-\, \{ \Pi_k x \, , \,  \Pi_k \in \cA^k \, , \, k = 0, \ldots , d-1\}$
has an empty interior.  Therefore, this set is contained in a proper face of~$K$.
Let~$L$ be a minimal by inclusion face containing this set.  As in the proof of Lemma~\ref{l15}
we show that $AL \subset L$ for all~$A \in \cA$. Hence~$\cA$ is $K$-reducible. The proof of the converse is straightforward. \qed 
\end{proof}

Lemma~\ref{l25} is extended for monotone norms and for the parameter~$h_K(\cA)$
without any change. Then we need the following observation:
\begin{lemma}\label{l40}
If a family~$\cA$ of Metzler operators is $K$-irreducible, then so is the
family~$e^{t \cA}$, for each $t > 0$.
\end{lemma}
\begin{proof}
 If~$\cA$ is
 $K$-irreducible, then so is the family $\cA' = h I + \cA$, for every $h > 0$.
 If $h$ is large enough, then every operator $A' \in \cA'$ is  $K$-positive,
 and hence $e^{ t A'} \, \ge_K \, I +  t A' \ge_K \, t \, A'$.
 Therefore, the $K$-irreducibility of the family $t \cA'$ implies that of the family $e^{t A'}$.
 Hence, the family~$e^{t \cA} = e^{- t h}e^{\tau \cA'}$ is $K$-irreducible. \qed
\end{proof}

Thus, for positive systems we do not need admissible numbers and do not use Lemma~\ref{l5}. 
Then, for a family~$\cA$ of Metzler operators,
 we denote
$$H_K(\cA) = \min\limits_{t \in [1,2]}h_K(e^{\, t \cA}).$$
This minimum is attained
at some point~$t_0 \in [1,2]$, because $h_K(e^{\, t \cA})$ is a continuous function of~$t$.
If $\cA$ is $K$-irreducible, then, by Lemma~\ref{l40}, so is the family~$e^{\, t_0 \cA}$,
 and therefore $H_K(\cA) = h_K(e^{\,  t_0 \cA}) > 0$.
 Thus, $H_K(\cA)$ is strictly positive for $K$-irreducible~$\cA$.
Then we establish a complete analogue of Proposition~\ref{p20} for $K$-primitive
families and for the parameter~$H_K(\cA)$. The rest of the proof of Theorem~\ref{th8-}
is literally the same as the proof of  Theorem~\ref{th8}.

\subsection{Algorithm~(P) for computing Lyapunov exponents
and constructing polytope norms of positive systems}

We are now ready to present an algorithm for computing the Lyapunov exponent and constructing
a polytope Lyapunov function specially for positive systems.
The corresponding algorithm will be referred to as Algorithm~(P) (P stands for {\em positive}).

For the sake of simplicity, we consider only the case~$K = \re^d_+$, i.e., we deal with a set of Metzler matrices~$\tilde A$,
although the same construction is applicable for other cones, for instance, for the positive semidefinite cone
(with the corresponding replacement of LP problems by semidefinite problems).
Algorithm~(P) is very similar to Algorithm~(R), we do not therefore give its detailed
presentation, but describe the differences from Algorithm~(R) only.
\smallskip

1) Algorithm~(R) is applicable for all irreducible sets of operators,
while Algorithm~(P) is applicable for all $K$-irreducible sets of $K$-Metzler operators.
In case $K = \re^d_+$, we obtain a positively irreducible set of Metzler matrices.
\smallskip

2) By the Krein--Rutman theorem~\cite{KR}, we have $\lambda_{\max} > 0$, and $v_1 \in K$. So, we do not have to consider cases when $\lambda_{\max}$ is negative or complex.
\smallskip

3) The main difference is that Algorithm~(P) constructs a {\em monotone polytope}~$P$.
Thus, in each step we have a monotone polytope $P_i = {\rm co}_{-}(\mathcal{V}_i)$. Everywhere we replace the symmetrized convex hull ${\rm co }_{s}(\cdot )$ by the  monotone convex hull
${\rm co }_{-}(\cdot ) $ (see (\ref{eq:cominus})). In particular the LP problem at line 11
of Algorithm~(P) is replaced by: \\

\noindent
{\bf 11}. Check whether  $z \in {\rm co}_{-} (\cV_i)$, $\cV_i := \{ v_j \}_{j=1}^{n_i}$, i.e. solve the LP problem
$$
\begin{array}{rcl}
\min & & f_\ell=\sum\limits_{j=1}^{n_i} \lambda_j
\\[0.25cm]
{\rm subject \ to}
     & & \sum\limits_{j=1}^{n_i} \lambda_j v_j  \!=\! z
\\[0.25cm]
{\rm and}
     & & \lambda_j \ge 0, \quad j=1,\ldots,n_i.
\end{array}
$$

The rest of the algorithm is the same as for Algorithm~(R).
The corresponding LP problems of Algorithm~(P), are described in~\cite{GP}.

The whole procedure gives a monotone polytope Lyapunov norm in~$K$ generated by the monotone polytope~
$P_N = {\rm co}_{-}(\cV_N)$ and a lower bound and upper bounds~(\ref{result}) for the Lyapunov exponent.

The upper bound is obtained by Algorithm \ref{algoP2}, which is similar to the previously
described Algorithm \ref{algoR2}.

\begin{algorithm}
\DontPrintSemicolon
\KwData{$\cA, \cP_{N}, \cV_{N}$ (system of vertices of $\cP_N$)}
\KwResult{$\alpha$}
\Begin{
\For{$i=1,\ldots,m$}{
\nl Solve the LP problems (w.r.t. $\{t_v\}$, $\alpha_i$)
\begin{eqnarray}
\label{LP.p}
\begin{array}{rl}
\min & \alpha_i
\\[0.1cm]
{\rm s.t.}
     & w + \delta (A_i - \alpha_i I) w \le \sum\limits_{v \in \cV_{N}}\, t_v\,
		 v \quad \forall w \in \cV_{N}
\\[0.1cm]
{\rm and}
     & \sum\limits_{v \in \cV_{N}}\, t_v \le 1, \qquad t_v \ge 0 \quad \forall v \in \cV_{N}
\end{array}
\nonumber
\end{eqnarray}
}
\nl Return $\alpha(\cP_N) := \max\limits_{1 \le i \le m} \alpha_i$ \;
}
\caption{Algorithm~(P), part 2. Computing the best upper bound \label{algoP2}}
\end{algorithm}

Theorem~\ref{th40} and Corollary~\ref{c50} hold true for Algorithm~(P) without any change, and their proofs stay the same for this case.

\begin{remark}\label{r40}
{\rm Numerical experiments (Section~5) show a very high efficiency of Algorithm~(P). While Algorithm~(R) finds the Lyapunov exponent with a satisfactory accuracy in dimensions~$d \le 10$ on a standard laptop, Algorithm~(P) does the same for positive systems of dimensions up to $100$ and higher. The number of vertices of the polytopes constructed by Algorithm~(P) is significantly smaller.
The reason is that the positive convex hull is regularly  much larger than the usual convex hull, and hence
Algorithm~(P) sorts out much more redundant vertices.
}
\end{remark}

\section{Stabilizability of positive systems}

 The {\em lower Lyapunov exponent} $\check \sigma (\cA)$ is the infimum
of numbers $\alpha$, for which there exists a control function $A(\cdot) \in \cU$ such that
 every corresponding trajectory of~(\ref{main}) satisfies
$\|x(t)\| \, \le \, C \, e^{\, \alpha t}$.  The system is {\em stabilizable} if
there is a control function $A(\cdot) \in \cU$ such that
$\|x(t)\| \to 0$ as $t \to +\infty$ for every corresponding trajectory. The stabilizability is equivalent to the condition
 $\check \sigma < 0$~\cite{LA,SDP}.

The following analogue of equality~(\ref{plus}) is true for the lower Lyapunov exponent:
\begin{equation}\label{plus-}
  \check \sigma \, (\cA\,  + \, s \, I) \ = \ \check \sigma \, (\cA) \, + \, s \ .
\end{equation}

Although the stabilizability issue is a very difficult problem, for positive systems it can often be
efficiently solved. Therefore we restrict our attention to positive systems. Besides, the stabilizability
of positive systems was the subject of an extensive literature (see~\cite{FV2,LA,SDP} and references therein).

Thus, we study stabilizability of $K$-positive systems, where~$K \subset \re^d$ is an arbitrary cone.
In case of the positive orthant~$K = \re^d_+$, we obtain the stabilizability of positive (in the usual sense) systems
of Metzler matrices. For other cones~$K$, such as polyhedral cones, positive semidefinite cones, etc.,
this problem also makes sense.
We begin by introducing the concept of Lyapunov antinorm on cones, which turns out to be natural
for characterizing  stabilizability. Some important properties of those
antinorms, in particular, an analogue of Theorems~A and~\ref{th10} (the existence of invariant antinorm)
and~Theorem~B (a geometric criterion of stabilizability)  are established in subsections 4.1 and 4.2.
Then we derive lower and upper bounds for $\check \sigma (\cA)$ by means of infinite polytopes on cones,
and estimate the distance between them. Applying these results we present Algorithm~(L) which
estimates the lower Lyapunov exponent and constructs the corresponding polytope antinorm on the cone.

\subsection{Antinorms on cones}

It is not difficult to formulate  analogues
to the notions of extremal and invariant norms
for stabilizable systems. In case $\check \sigma = 0$, it would be natural
 to define a norm to be extremal, if it is   non-decreasing in~$t$ on every trajectory~$x(t)$ of the system.
However, simple examples show that
extremal norms may not exist,  even for an irreducible pair of positive $2\times 2$-matrices.
 It was first observed in~\cite{BS} that stabilizability does not imply the
existence of convex Lyapunov function. In the proof of Theorems~\ref{th10},
an extremal norm is constructed as a pointwise supremum of some convex functionals. This is natural, because
the operation of taking supremum respects the convexity. For the lower Lyapunov exponent, the
supremum has to be replaced by infimum, but this operation does not preserve convexity.
Therefore, one might suggest to consider concave functions rather than convex.  However,
 positive homogeneous concave functions on~$\re^d$ do not exist.  On the other hand, such functions
exist on any cone~$K \subset \re^d$, and this makes theoretically possible to apply them for $K$-positive systems.  We are going to show that   stabilizable positive systems defined by Metzler operators on an arbitrary cone~$K$
 do always have concave Lyapunov functions.
\begin{defi}\label{d30}
An antinorm on a cone~$K$ is a nontrivial nonnegative concave homogeneous functional on~$K$.
An antinorm is called positive if it is positive at all points $x \in K\setminus \{0\}$.
\end{defi}
The concept of antinorm originated in~\cite{P1} to analyze random positive systems.
It was applied to discrete-time stabilizable positive systems in~\cite{GP}.
In contrast to norms, an antinorm is always monotone on the cone.
\begin{lemma}\label{l50}
Any antinorm~$f$ on a cone $K$ is monotone, i.e., $x \ge_K y \, \Rightarrow
\, f(x)\ge f(y)$.
\end{lemma}
\begin{proof} 
If $x \ge_K y$, then $y + t(x-y) \in K$ for every $t \ge 0$.
Suppose $f(x) < f(y)$; then by concavity, for every $t > 1$, we have
$f(y + t(x-y)) \le f(y) + t (f(x) - f(y))$, which becomes negative
for large positive~$t$. This contradicts nonnegativity of $f$. \qed 
\end{proof}

\begin{defi}\label{d40}
Let all operators of~$\cA$ be Metzler for a cone~$K$.
An antinorm $f(\cdot)$ on $K$ is called extremal if
for every trajectory of~(\ref{main}) starting in~$K$
we have ${f(x(t)) \, \ge \, e^{\, \check \sigma \, t}f(x(0))\, , \ t \ge 0}$.

An extremal antinorm is called invariant if for every $x_0 \in K$ there exists a generalized trajectory
$\bar x (t)$ with $\bar x (0) = x_0$ such that  $f(x(t)) \, = \, e^{\, \check \sigma \, t}\, f(x_0)\, , \ t \ge 0$.
\end{defi}

Thus,   for an extremal antinorm
the function $e^{-\, \check \sigma \, t}f(x(t))\, $ is non-decreasing in $t$ on every trajectory. For an invariant antinorm,
this function is identically constant on some trajectory, and for every point $x_0 \in K$ there is
such a trajectory starting in it. For $ \check \sigma = 0$, we have
\begin{cor}\label{c30}
Let $K$ be a given cone. In case~$ \check \sigma (\cA) = 0\, $ an  antinorm is extremal for~$\cA$ if and only if it is non-decreasing in~$t$ on every trajectory of~(\ref{main}) in the cone. An extremal antinorm is invariant if and only if for every $x_0 \in K$ there exists a
generalized trajectory $\bar x (t)$ with $\bar x (0) = x_0$ on which this antinorm is identically constant.
\end{cor}
Consider the unit level set $D = \{x \in K | \ f(x) \ge 1\}$ of this antinorm.
This is a convex unbounded subset of~$K$. The antinorm is extremal if and only if
every trajectory starting on the boundary $\partial_K D$ never leaves the set~$D$.
The antinorm is invariant if for each point of the boundary there exists a trajectory starting at this point
 that eternally  remains on the boundary.

\begin{theorem}\label{th20}
Let~$K$ be a given cone. Every compact set~$\cA$ of $K$-Metzler operators possesses an extremal antinorm on~$K$.
If, in addition, every operator from~$\cA$ is $K$-irreducible, then there exists a positive invariant antinorm on~$K$.
\end{theorem}
\begin{remark}{\rm In Theorem~\ref{th20}, in contrast to Theorem~\ref{th10}, there is no irreducibility assumption
for the existence of an extremal antinorm. It always exists for
a family of Metzler operators. This antinorm, however, may vanish on the
boundary of~$K$ and may not be invariant. An invariant positive
antinorm exists under a stronger irreducibility assumption: each operator from~$\cA$ is $K$-irreducible.
}
\end{remark}

The proof of Theorem~\ref{th20} is in Appendix. That is somewhat similar to the proof of Theorem~\ref{th10}, but with
 differences in several key points.
 The main one is the use of concept of embedded cone.

\begin{defi}\label{d50}
A cone $K'$ is embedded in a cone~$K$ if $ (K'\setminus \{0\}) \, \subset \, {\rm int} \, (K)$.
\end{defi}

\begin{prop}\label{p40}
If all operators from a compact family~$\cA$ are $K$-Metzler and each of them is~$K$-irreducible, then
they are Metzler with respect to some cone~$K'$ embedded in~$K$.
\end{prop}
The proof is in Appendix.

\begin{remark}\label{r10}
{\rm Thus, if all operators from a given compact set are Metzler and irreducible for a given cone~$K$, then~$K$ can be narrowed down to an embedded cone
so that all those operators stay Metzler. Note that an analogous statement for nonnegative operators (i.e. leaving a cone invariant)
does not hold. If a set of irreducible operators leaves a cone invariant, then it may not leave invariant any embedded cone.
 For example, the following pair of matrices
 $$
 A_1 \ = \
 \left(
 \begin{array}{cc}
 0& 2\\
 1& 0
 \end{array}
 \right)\ ; \qquad
 A_2 \ = \
 \left(
 \begin{array}{cc}
 0& 1\\
 1& 0
 \end{array}
 \right)\
 $$
 leaves invariant the positive orthant~$K = \re^2_+$, however, no embedded cone of~$K$
 is invariant, because, for every positive vector~$x$ the direction of the vector $(A_1A_2)^kx $
 converges to~$(1,0)^T$ as~$k \to \infty$. }
 \end{remark}


Applying Lemmas~\ref{l10} and~\ref{l20} we obtain

\begin{cor}\label{c40}
Under the assumptions of Proposition~\ref{p40}, every trajectory of~(\ref{main}) starting in~$K'$
is contained in~$K'$. In particular, for every~$t> 0$ the family~$e^{\, t \cA}$ leaves~$K'$ invariant.
\end{cor}

\subsection{Geometric conditions of stabilizability}

The conditions of stabilizability can be formulated in terms of vector
fields, similarly to Theorem~B and Proposition~\ref{p30}. To do this we need some more notation.

We are given a cone~$K$. A {\em monotone infinite body} (in short, {\em infinite body}) is a
convex closed  proper subset~$G \subset (K\setminus \{0\})$ such that~$x \in G, y \ge_K x\, \Rightarrow \, y \in G$.
Each infinite body defines an antinorm on~$K$ by the formula $f(x) = \sup\, \{\lambda > 0 \, | \,
\lambda^{-1}x \in G\}$. Conversely, for an arbitrary antinorm~$f$, its unit ball, i.e., the level set~$D = \{x \in K \ | \, f(x) \ge 1\}$
is an infinite body. The antinorm is positive precisely when its unit sphere $\partial_K\, D$ is bounded.

The {\em infinite convex hull} of a subset~$M \subset (K \setminus \{0\})$ is the smallest by inclusion
infinite body that contains~$M$. It can be defined by the formula
$$
 {\rm co}_{+}(M)\ = \  {\rm co}(M)\, + \, K\ = \
 \bigl\{ \, y \, + \, z\  \bigl| \  y \, \in \,
{\rm co}\, (M)\, , \   z \in K\bigr\}\,
$$
The infinite convex hull of a finite set of points is  called
{\em infinite polytope}. Some of these points are vertices of this polytope, i.e., its extreme points.

\begin{prop}\label{p60}
If there is an infinite  body~$Q$  such that, for every point~$x \, \in \, \partial_K \, Q$, all vectors
$Ax\, , \ A \in \cA$, are directed inside~$Q$, then $\check \sigma (\cA) > 0$, and $\cA$ is not stablilizable. If~$Q$ is an infinite polytope,
then it suffices to check this condition only for its vertices~$x$.

Conversely, if $\check \sigma (\cA) > 0 $, then there exists such an infinite body~$Q$.

\end{prop}
\begin{proof}  
Let $f$ be the antinorm generated by~$Q$. If $f$ is differentiable,
then the condition that $Ax$ is directed inside~$Q$ means that $(f'_x, Ax) \ge 0$. Consequently,
for almost all~$t$ (in Lebesgue measure), we have $f'_t\bigl(x(t)\bigr) = \bigl(f'_x(t), \dot x(t)\bigr) =
\bigl(f'_x(x), A(t)x\bigr) \ge 0$,
hence $f\bigl(x(t)\bigr)$ is non-decreasing in~$t$, and the system is not stabilizable. This proof is extended to nonsmooth~$f$
by the standard argument, as it is done for norms (see, for instance~\cite{MP2,P1}).

Let now~$Q$ be an infinite  polytope. If for every its vertex $v$ the vector $Av$ is directed inside~$Q$,
then there is $\eta > 0$ such that $v  + \eta A v \in {\rm int Q}$.
Rewriting this inclusion in the form $\eta (A  + \eta^{-1} I)v \in {\rm int}_K \, Q$,
we see that the operators  $A  + \eta^{-1} I$ is $K$-positive, whenever $\eta$ is small enough.
 Hence, this inclusion holds for every convex combination~$x$ of vertices of~$Q$ and
 for all points $y \ge_K x$, i.e., for all $y \in Q$.  Thus, the vector $Ay$ is directed inside~$Q$, for every~$y \in \partial_K \, Q$, and so $\check \sigma (\cA) \ge 0$.

To prove the existence, we invoke Theorem~\ref{th20} and consider an extremal antinorm~$f$
of the family~$\cA$. Let us show that its level set $Q = \{x \in K \ | \ f(x) \ge 1\}$
 is what we need. Take an arbitrary $\alpha \in (0, \check \sigma)$.
For any $x_0 \in \partial_K Q$, and for every trajectory~$x(t)$
with $x(0) = x_0$, we have $f(x(t)) \ge e^{\alpha t} f(x_0) = e^{\alpha t}\, , \ t > 0$.
This implies $f'_t(x(0)) \ge \alpha$. Hence, for every element $a^*$ from the
 subdifferential of the  function~$f$ at the point~$x_0$, we have
 $\alpha \le (a^*, \dot x (0)) = (a, Ax_0), \, A \in \cA$. Therefore, the vector $Ax_0$
 is directed inside~$Q$. \qed 
\end{proof}

\subsection{Stabilizability of discrete systems and the lower spectral radius}

Similarly to previous cases we discretize (\ref{main}) and are lead to a problem of so called lower spectral radius,
that is to determine the lowest rate of growth in the product semigroup generated by a set of matrices.

Before we formulate our results for stabilizability of continuous-time LSS, let us recall some facts on discrete ones.
Stabilizability of a discrete system is decided in terms of its lower spectral radius (LSR).
 \begin{defi}\label{d60}
For a given compact set of matrices~$\cB$, the lower spectral radius~$\check \rho(\cB)$ is
$$
\check \rho(\cB)\ = \ \lim_{k \to \infty}\min_{B_i \in \cB\, , \, i = 1,\ldots , k}\, \bigl\|B_k\ldots B_1\bigr\|^{\, 1/k}\, .
$$
\end{defi}
This limit exists for every compact set o matrices~$\cB$ and does not depend on the matrix norm~\cite{Gurv}.
See also~\cite{P3,GP} for properties and for more applications of LSR.
The discrete system
is stabilizable  if and only if $\check \rho(\cB) < 1$.
If one discretizes the continuous system with dwell time~$\tau > 0$ to the form~(\ref{main-discr})
by setting~$x_k = x(k\tau )\, , \, B = e^{\, \tau A},\, \, A \in \cA$, then we obtain
only those trajectories corresponding
to piecewise-constant control functions $A(\cdot)$ with the step size~$\tau$.
Hence, if there is $\tau > 0$
for which the discrete system is
stabilizable, i.e., $\check \rho(e^{\, \tau \cA}) < 1$, then the continuous system is stabilizable as well.

In~\cite{GP} we presented an algorithm for LSR computation that for most of families (also in high dimensions) gives 
the precise value of~$\, \check \rho(\cB)$.
The main idea is analogous to the JSR computation, but involving antinorms instead of norms.
The algorithm tries to find the {\em spectrum minimizing} (or {\em lowest}) product (s.l.p.),
of matrices from~$\cB$ for which the value
$[\rho(\Pi)]^{\, 1/n}$ is minimal, where $n = n(\Pi)$ is the length of~$\Pi$.
To this end, we first fix some reasonably large $l \in \n$ and check all products of lengths~$n \le l$ finding
 a product~$\Pi$ with the minimal value~$[\rho(\Pi)]^{\, 1/n}$. We denote this value by~$\check \rho_l$ and consider this product as a candidate for s.l.p. Then
the algorithm iteratively build an infinite  polytope~$Q$ for which $\, B\, Q \, \subset \,
\check \rho_l Q\, , \ B \in \cB$. If it terminates within
finitely many iterations, then the infinite polytope~$Q$ is extremal, $\Pi$ is an s.l.p., and $\check \rho(\cB) = \check \rho_l$. Let us denote $\tilde \cB \, = \, \{\tilde B = \rho_l^{-1}B | \  B \in \cB\}$.
\begin{defi}\label{d70}
A product $\Pi\in \cB^n$ is called under-dominant for the family
$\cB$  if there is $p> 1$ such that the spectral radius of every
product of operators of the normalized family $\tilde \cB$, that is
not a power of~$\widetilde \Pi$ nor a power of its cyclic
permutation, is greater than $\, p$.
\end{defi}
It is shown in~\cite[theorem 4]{GP} that the algorithm terminates within finite time if and only if
the product~$\Pi$ is under dominant for~$\cB$.

\subsection{Bounds for the lower Lyapunov exponent}

For a given~$\tau > 0$, we set~$\beta (\tau) = \tau^{-1}\ln \check \rho (e^{\tau A})$.
For a given infinite polytope~$Q \subset K$,  we define the value $\check \alpha (\cA , Q) \, = \, \check \alpha (Q)$ as follows:
\begin{eqnarray}
\check \alpha (Q) & = & \sup \ \Bigl\{ \alpha \in \re \ \Bigl| \
\mbox{for each vertex}\ v \in Q \ \mbox{and} \ A \in \cA, \Bigr.
\nonumber
\\
& & \qquad \mbox{the vector} \ (A - \alpha I) v \  \mbox{is directed inside} \, Q\, \Bigr\}\, .
\label{def-alpha+}
\end{eqnarray}

\begin{prop}\label{p50+}
For an arbitrary compact family~$\cA$ of Metzler operators, for each number $\tau > 0$ and for an
infinite polytope~$Q$, we have
\begin{equation}\label{lu+}
\check \alpha (P)\ \le \  \check \sigma \ \le \ \check \beta (\tau) \, .
\end{equation}
\end{prop}
\begin{proof} 
Is realized in the same way as the proofs of Propositions~\ref{p50} and~\ref{p50-}. \qed 
\end{proof}

The notions of extremal and  {\em $\eps$-extremal} polytope are extended
to infinite  polytopes, replacing
the joint spectral radius by lower spectral radius, and multiplying by~$e^{-\tau \eps}$
instead of~$e^{\, \tau \eps}$. Thus,  an infinite polytope~$Q$
is $\eps$-extremal for $e^{\tau \cA}$ if
$$
 e^{\tau A} \, Q \ \subset \ e^{- \tau \eps } \, \check \rho (\cA)\ Q\, , \qquad A \in \cA\, .
$$
For $\eps = 0$, we obtain an extremal infinite polytope.
Since $\rho (e^{\tau \cA}) \, = \, e^{\tau \check \beta (\tau)}$,
the  $\eps$-extremalily is equivalent to the inclusion
\begin{equation}\label{varep+}
 e^{\tau A} \, Q \ \subset \ e^{\tau (\check \beta - \eps )} \, Q\ , \qquad  A \in \cA\, .
\end{equation}
According to Proposition~\ref{p40}, if all operators of a family~$\cA$ are $K$-Metzler and $K$-irreducible,
 then they are all $K'$-Metzler, for some cone~$K'$ embedded in~$K$. The following theorem shows that
 both the upper and lower bound from~(\ref{lu+}) are close to each other, provided~$Q$ is $\eps$-extremal and has all its vertices in~$K'$.
\begin{theorem}\label{th8+}
For every compact family~$\cA$ of $K$-Metzler $K$-irreducible operators,  there is a constant~$C$ such that
for all~$\tau ,  \eps > 0$, we have
$$
\check \beta (Q) \ - \ \check \alpha (\tau) \ \le \ C\tau \, + \, \eps\, ,
$$
whenever the infinite polytope~$Q$ is $\eps$-extremal for the family $e^{\tau \cA}$ and
has all its vertices in the embedded cone~$K'$ from Proposition~\ref{p40}.
\end{theorem}

The proof is realized in a similar way as for Theorem~\ref{th8}, applying antinorms
instead of norms. In this case, however, Proposition~\ref{p20} on the equivalence of all
contractive norms, is inapplicable. In general, it does not hold for antinorms. Instead,
we involve embedded cones and use Proposition~\ref{p40}.
\smallskip

\begin{proof} 
Without loss of generality it can be assumed that
$\tau$ is small enough (otherwise we change the constant~$C$) and that
$\check \beta (\cA) = 0$ (otherwise, we replace the family~$\cA$ by $\cA- \beta I$).
Since all operators from~$\cA$ are $K'$-Metzler, it follows that
 $K'$ is invariant for the family~$e^{\, \tau A}$, for all~$\tau > 0$, and is invariant for the
  family~$I\, +\, \tau\cA$, for all sufficiently small~$\tau > 0$. Since $K'$ is embedded in~$K$,
 there is a constant~$C_0 > 0$
  such that for every antinorm~$f$ on~$K$ we have
\begin{equation}\label{b)}
\ \bigl|\, f(a) \, - \, f(b)\,   \bigr|\ \le \ C_0 \, \bigl\|\, a - b\, \bigr\|\, , \quad a, b \in K'\, .
\end{equation}
The proof can be easily derived or found in~\cite{GP}.
Let us  show that there is a constant~$C$
such that for every vertex~$v \in Q$  and for each~$A \in \cA$
the vector~$(\, A \, + \, (C\tau + \eps)\, I)\, v$ is directed inside~$Q$.
This will imply that $\check \alpha \, \ge \, -  C\, \tau \, - \, \eps$, which is required.
 Let $f$ be the antinorm on~$K$ generated by~$Q$. We denote~$\cA' = \cA + \eps I$ and consider arbitrary~$A' \in \cA'$.
By~(\ref{const1}) the distance between points~$e^{\, \tau A'}v$ and $(I + \tau A')v$
is smaller than~$C\tau^2$. Since both these points are in the embedded cone~$K'$,
inequality~(\ref{b)}) yields
$$
\bigl|\, f\bigl(e^{\, \tau A'}v\bigr) \, - \, f\bigl((I + \tau A')v\bigr)\, \bigr|
\ < \  C_0C \tau^2\, .
$$
 Let us now denote the value~$C_0C$ by a new constant~$C$.
Since $\check \beta = 0$ and  $Q$ is $\eps$-extremal for the family~$e^{\tau \cA}$,
we see that $e^{\tau A'}v =  e^{\eps \tau} e^{\tau A} v \, \in \, e^{\eps \tau}
e^{(\check \beta - \eps ) \tau} Q
\, = \, Q$. Thus, $f\bigl(e^{\tau A'}v\bigr) \ge 1$, and consequently $f\bigl((I + \tau A')v\bigr) >
1 - C_0C \tau^2$. We assume $\tau$ is small enough, and so $1 - C_0C \tau^2 > 0$.
Therefore, the point $\, y \, = \, 1/(1-C_0C\tau^2)\, (I + \tau A')\, v\, $ belongs to~${\rm int}_K\, Q$,
and hence, the vector from the point~$v$ to~$y$ is directed  inside~$Q$.
This vector is $\, y - v \, = \, \tau/(1 - C\tau^2)\, \bigl( A' \, + \, C\tau \, I \, \bigr)\, v$.
Whence, the vector~$( A'  +  C\tau I ) v = ( A  +  (C\tau + \eps) I ) v$  is directed  inside~$Q$.
\qed
\end{proof}


\subsection{The case $\mathbf{K = \re^d_+}$. Nonnegative matrices}

Let us recall that in the simplest case, when $K = \re^d_+$, an operator~$A$ is Metzler if it is written by a Metzler matrix,
i.e., a matrix with nonnegative off-diagonal elements.
All the results of Sections~3 and~4 hold true for $K = \re^d_+$ and for a compact set~$\cA$
of Metzler matrices. In this case the $K$-irreducibility coincides with
the usual positive irreducibility of nonnegative matrices. A set of  matrices is
{\em positively irreducible} if none of the coordinate planes (i.e., linear spans of several basis vectors)
is a common invariant subspace for those matrices. For positively reducible set of matrices, there always exists
a  permutation of basis vectors, after which they get a block upper triangular form.

\subsection{Algorithm~(L) for computing  the lower Lyapunov exponent
and constructing the polytope antinorm}

For  a given finite family~$\cA = \{A_1, \ldots , A_m\}$ of  $K$-Metzler operators, or for the corresponding  polytope
family ${\rm co} \, (\cA)$,
Algorithm~(L) approximates the lower Lyapunov exponent~$\check \sigma (\cA)$ by computing its lower and upper bonds,
and finds a polytope Lyapunov antinorm on~$K$. We begin with a brief description of the algorithm.
\smallskip

{\it Initialization}.  We choose a small number $\tau > 0$ (dwell time), a small number $\nu \ge 0$,
 and a reasonably large natural~$l$.
Among all products of length~$\le l$ of the operators $e^{\, \tau A_j}\, , \ j = 1, \ldots , m$,
we select a {\em starting} product $\Pi \, = \, e^{\, \tau A_{d_n}}\cdots e^{\, \tau A_{d_1}}\, , \ n \le l$,
for which the value $[\rho (\Pi)]^{1/n}$ ($n$ is the length of~$\Pi$) is as small as possible.
This is done by exhaustion of products of lengths at most~$l$. If $l$ is not too large, then we are
able to make a full exhaustion (this is preferable) and find the {\em minimal} product~$\Pi$
that gives the minimal value of~$[\rho (\Pi)]^{1/n}$ among all products of lengths at most~$l$.
We denote~$\check \rho_l = [\rho(\Pi)]^{1/n},
\, \check \beta_l(\tau) = \tau^{-1} \ln \check \rho_l$
and observe that $\check  \beta_l \ge \check  \beta$. Moreover, for the minimal starting products,
we have  $\check  \beta_l \to \check \beta$ as $l \to \infty$.
By the Krein--Rutman theorem, the leading eigenvalue~$\lambda_{\max}$ of $\Pi$ is positive.
 Let $v_1$ be the corresponding eigenvector.
 We normalize the family $\cA$ as 
$$
\tilde \cA = \cA - (\check \beta_l - \nu)I.
$$

{\it The first part. Construction of the infinite polytope~$Q$.}
We start with the finite set $\cV_0 = \{v_1, \ldots , v_n\}$
and the corresponding infinite polytope $Q_0 = {\rm co}_{+} (\cV_0)$, where
$v_i = e^{\, \tau A_{d_{i-1}}}\cdots e^{\, \tau A_{d_{1}}}v_1$. At the $k$th step, $k \ge 1$,
we have a finite set of points ${\mathcal V}_{k-1}$ and a polytope $Q_{k-1} = {\rm co}_{+} ({\mathcal V}_{k-1})$.
We take a point~$v \in {\mathcal V}_{k-1}$ added in the previous step, and for each $j = 1, \ldots , m$
check if $e^{\tau \tilde A_j}v \in P_{k-1}$, by solving the corresponding LP problem. If
the answer is affirmative, then we go to the next $j$; if $j=m$, we go to the next point from~${\mathcal V}_{k-1}$.
Otherwise, if $e^{\tau \tilde A_j}v \notin P_{k-1}$, we update the set ${\mathcal V}_{k-1}$ by adding the point $v$ and update
the polytope $Q_{k-1}$ by adding the vertex~$v$, and then go the next~$j$ and to the next point~$v$. After we exhaust all points
of the initial set~${\mathcal V}_{k-1}$, we start the $(k+1)$st step, and so on. This process
terminates after~$N$th step, when $\mathcal{V}_N = \mathcal{V}_{N-1}$, i.e., no new vertices are added
to the~$Q_{N-1}$. In this case, $e^{\tau \tilde A}Q_{N}\subset Q_N\, , \ \tilde A \in \cA - (\check \beta_l - \nu)I$.
This means that~$Q_N$ is an $\eps$-extremal infinite polytope for the family~$\cA$ with
\begin{equation}\label{vareps+}
\eps \quad = \quad \check \beta(\tau) \ - \ \check \beta_l(\tau) \ + \ \nu\, .
\end{equation}

{\it The second  part. Deriving the lower and upper bounds for $\check \sigma (\cA)$}.
The first part produces the $\eps$-extremal infinite polytope~$Q_N$.
We compute $\check \alpha(Q_N)$ by definition, as supremum of numbers~$\alpha$ such that
the vector $(A - \alpha I)w $ is directed inside $Q_N$, for each vertex $w \in Q_N$ and for every $A \in \cA$.
This is done by taking a small~$\delta> 0$ and solving the following  LP problem:
$$
\left\{
\begin{array}{l}
\alpha \ \to \ \sup \\
w + \delta (A - \alpha I)w \, \in \, Q_N,\\
\, w \in Q_N,\  A \in \cA\, .
\end{array}
\right.
$$
 Proposition~\ref{p50+} implies
\begin{equation}\label{result+}
\check \alpha(Q_N) \ \le \ \check \sigma(\cA) \ \le \ \check \beta_l(\tau)\, .
\end{equation}
Thus, we obtain the lower and upper bounds $\alpha(Q_N)$ and $\beta_l(\tau)$ respectively for
the Lyapunov exponent. The infinite polytope~$Q_N$ constitutes the Lyapunov antinorm for the family~$\cA$.
If $\check \beta_l (Q_N) < 0$, then we conclude that the system is stabilizable. If $\check \alpha(Q_N) \ge 0$,
 then it is not stabilizable and the joint Lyapunov function
is defined by the infinite polytope~$Q_N$.

If the first part of the algorithm does not terminate within finite time, then we either increase~$l$ or increase~$\nu$ and go
to the Initialization.

Now we present Algorithm (L) in a structured form:

\begin{algorithm}
\DontPrintSemicolon
\KwData{$\cB_{\dtt} = \e^{\dt \tilde \cA}$ (see (\ref{eq:normalized}))\; 
The product $\Pi$ (candidate s.l.p.) of length $n$ such that $[\rho(\Pi)]^{1/n}$ is as small as possible.
among all products of length smaller or equal to $l \ge n$
}
\KwResult{the polytope $Q$}
\Begin{
\nl Compute the leading eigenvector of $\Pi$ and its cyclic permutation $\{v_1,\ldots,v_n\}$ \;
\nl Set $\cV_0 := \{ v_j \}_{j=1}^{n}$ and $\cR_0 = \cV_0$\;
\nl Set $i=0$\;
\nl Set ${\rm term} = 0$\;
\nl \While{${\rm term} \neq 1$} {
\nl $\cW_{i+1} = \cB_{\dtt}\, \cR_{i}$\;
\nl Set $\cS_i = \emptyset$\;
\nl Set $m_i = {\rm Cardinality}(\cW_{i+1})$\;
\nl Set $n_i = {\rm Cardinality}(\cV_i)$\;
\For{$\ell=1,\ldots,m_i$}{
\nl Let $z$ the $\ell$-th element of $\cR_i$\;
\nl Check whether  $z \in {\rm co}_{+} (\cV_i)$, $\cV_i := \{ v_j \}_{j=1}^{n_i}$, i.e. solve the LP problem\;
$
\begin{array}{rcl}
\min & & f_\ell=\sum\limits_{j=1}^{n_i} \lambda_j
\\[0.25cm]
{\rm subject \ to}
     & & \sum\limits_{j=1}^{n_i} \lambda_j v_j  \!=\! z
\\[0.25cm]
{\rm and}
     & & \lambda_j \ge 0, \quad j=1,\ldots,n_i.
\end{array}
$\;
\nl \If{$f_\ell < 1$}{
\nl $\cS_i = \cS_i \cup z$\;
}
}
\eIf{$\cS_i = \emptyset$}{
\nl Set ${\rm term} = 1$\;}{
\nl $\cR_i = \cS_i$\;
\nl $\cV_{i+1} = \cV_{i} \cup \cR_{i}$\;}
\nl Set $i=i+1$\;
}
\nl Set $N=i$\;
\nl Return $Q := Q_{N}$ (extremal polytope)\;
}
\caption{Algorithm (L), part 1. Constructing of the infinite polytope \label{algoL}}
\end{algorithm}

\smallskip

\begin{algorithm}
\KwData{$\cA, Q_{N}, \cV_{N}$ (system of vertices of $Q_N$)}
\KwResult{$\alpha$}
\Begin{
\For{$i=1,\ldots,m$}{
\nl Solve the LP problems (w.r.t. $\{t_v\}$, $\alpha_i$)
\begin{eqnarray}
\label{LP.l}
\begin{array}{rl}
\max & \alpha_i
\\[0.1cm]
{\rm s.t.}
     & w + \delta (A_i - \alpha_i I) w \ge
		  \sum\limits_{v \in \cV_N}\, t_v\, v \quad \forall w \in \cV_N
\\[0.1cm]
{\rm and}
     & \sum\limits_{v \in \cV_N}\, t_v \ge 1, \qquad t_v \ge 0 \quad \forall v \in \cV_N
\end{array}
\nonumber
\end{eqnarray}
}
}
\nl Return $\alpha(Q_N) := \max\limits_{1 \le i \le m} \alpha_i$ \;
\caption{Algorithm to compute the best lower bound $\alpha (Q_N)$~\label{algoL2}}
\end{algorithm}


\begin{theorem}\label{th40+}
Let all operators of~$\cA$ be $K$-irreducible. Then
Algorithm~(L) terminates within finite time if one of the following conditions is satisfied:

1) $\nu \, > \, \check \beta_l(\tau)\, - \, \check \beta(\tau)$;

2) $\nu = 0$, the product $\Pi$ is under-dominant for the family $e^{\tau \cA}$ and
its leading eigenvalue~$\lambda_{\max}$ is unique and simple.

In case 1) the distance between the lower and upper bounds in~(\ref{result})
does not exceed  $\check \beta - \check \beta_l + \nu + C\tau$; in case 2) it does not exceed $C\tau$,
where $C = C(\cA)$ is a constant.
\end{theorem}
\begin{proof} 
By Proposition~\ref{p40} all operators  of~$\cA$ are $K'$-Metzler with respect to
an embedded cone~$K' \subset K$. Hence, by the Krein-Rutman theorem, the leading eigenvector $v_1$ of $\Pi$
 belongs to $K'$. Since all operators $e^{\tau \tilde A}, \ \tilde A \in \tilde \cA$,
  are $K'$-positive, all vertices of the polytope~$Q_N$ generated by the algorithm lie in~$K'$.

In the case 1) we have $\check \rho(e^{\tau \tilde A}) = e^{ (\check \beta - \check \beta_l + \nu)\tau} > 1$.
Since the vector $v_1$ belongs to the embedded cone~$K'$, its images by
products of operators from $e^{\tau \tilde A}$ of length~$k$ tend to infinity as $k \to \infty$
(see, for instance, \cite{P3}). Therefore, for all~$k \ge n$, where $n$ is a large  natural number,  those images
age greater (in the order of the cone~$K$) than~$v_1$, and hence, belong to~$Q_n$.
So, all vertices~$v$ produced by the algorithm  after the $n$th step belong to~$Q_n$, which means that the first part of
 terminates in~$n$th step or earlier.

In case 2) we have $\check \rho(e^{\tau \tilde A}) = 1$, the spectrum minimizing  product~$\Pi$ is
under-dominant and its leading eigenvalue is unique and simple. By theorem~7 of~\cite{GP} Algorithm~(L) terminates
within finite time.

Since the polytope~$Q_N$ is $\eps$-extremal for $\eps = \check \beta -  \check \beta_l  +  \nu$,
the upper bound for the difference $\check \beta_l(\tau) - \check \alpha(P_N)$ follows from Theorem~\ref{th8+}. \qed 
\end{proof}

\begin{cor}\label{c50+}
If all operators of the family~$\cA$ are $K$-irreducible and the starting products~$\Pi$ are maximal,
then the distance between the lower and upper bounds in~(\ref{result}) tends to zero as
$l\to \infty$ and $\nu \to 0\, , \, \tau \to 0$.
\end{cor}
\begin{proof} 
Since $\check \beta_l \, \to \, \check \beta$ as $l \to \infty$, the corollary follows
by applying Theorem~\ref{th40+}. \qed 
\end{proof}

\subsection{An illustrative example in dimension $2$.}
\label{ex:simpleL}

Let ${\cA}=\{ A_1, A_2 \}$ with 
\begin{eqnarray*}
A_1 & = & \left(\begin{array}{rr}
   1.94591\ldots &              0 \\
   0.42364\ldots &  1.09861\ldots
\end{array} \right)
\\
A_2 & = & \left(\begin{array}{rr}
   0.69314\ldots     &  0.92419\ldots \\
                   0 &  2.07944\ldots
\end{array} \right).
\end{eqnarray*}
For $\dt = 1$ we set $\cB = \{ \tilde{B}_1, \tilde{B}_2 \}$ with
\begin{eqnarray*}
\tilde{B}_1 = \left(\begin{array}{rr}
7 & 0
\\
2 & 3
\end{array} \right), \quad
\tilde{B}_2 = \left(\begin{array}{rr}
2 & 4
\\
0 & 8
\end{array} \right).
\end{eqnarray*}
i.e. $A_1=\logm(\tilde{B}_1)$ and $A_2=\logm(\tilde{B}_2)$.

By means of the mentioned algorithm for computing the l.s.r.
we are able to prove that the product of degree equal to $8$,
$$P = \tilde{B}_1\,\tilde{B}_2\,(\tilde{B}_1^2\, \tilde{B}_2)^2$$
is spectrum minimizing, so that
$\check\rho(\tilde\cB) = \rho(P)^{1/8} = 6.009313489\ldots$, giving the upper bound
$$\check{\beta}=1.793310513\ldots.$$
Then we set $\cB = e^{\cA - \check\beta I} = \{ B_1, B_2 \}$ 
with $B_1 = \tilde{B}_1/\check\rho(\tilde{\cB}), B_2 = \tilde{B}_2/\check\rho(\tilde{\cB})$
and apply Algorithm (L), part 1. 
As a result we obtain the polytope antinorm in Figure \ref{fig2}, whose unit ball is an infinite
polytope $\cQ_\dt$ with $9$ vertices.
\begin{figure}[ht]
\centering
\global\def\path{#1}\input{fig2.inp} \\[2mm]
      \caption{Polytope antinorm for the illustrative example with $\dt=1$. In red the vectors
			$B_1 v$ and in blue the vectors $B_2 v$,
			for $v \in V_{\dt}$, vertices of $\cQ_{\dt}$.}
      \label{fig2}
\end{figure}

\begin{figure}[ht]
\centering
\global\def\path{#1}\input{fig_anormcnt.inp} \hskip 0.5cm \global\def\path{#1}\input{fig_anormcnt2.inp}\\[2mm]
      \caption{In red the vectors $(A_1-\check\alpha I) v$ and in blue the vectors $(A_2-\check \alpha I) v$,
			for $v \in V_{\dt}$, vertices of $\cQ_{\dt}$.}
      \label{fig2b}
\end{figure}

Applying Algorithm \ref{algoL2} we obtain the optimal shift $\check \gamma = 0.1323026\ldots$ so that we have the estimate
$$
\check\alpha = 1.661007914\ldots \le \check\sigma \le 1.793310513\ldots = \check\beta.
$$
If, however, we take $\dt=1/16$ we obtain a polytope with $28$ vertices which gives the following interval 
of length $\check\gamma = 0.0189\ldots$,
$$
\check\alpha = 1.755426316\ldots \le \check\sigma \le 1.774326316\ldots = \check\alpha.
$$


Figure \ref{fig2b} illustrates the fact that the computed polytope $\cQ_\dt$ is positively invariant for the shifted family
$\cA - \alpha I$.

It also shows (right picture) that one of the vectorfields $(\cA - \check\alpha I) v$ is tangential to the boundary of the polytope,
in agreement with the property that $\check\alpha = 1.661007914\ldots$ cannot be increased (or equivalently $\gamma$ cannot be 
decreased).

\section{The phenomenon of fibrillation}

A natural question involves the existence of an optimal piecewise continuous control function
determining the upper/lower Lyapunov exponent. This is not always true, as we are showing.
We speak of fibrillation whenever as $\dt \rightarrow 0$
the spectrum maximizing product, say $\Pi_{\dt}$, has bounded degree
(independent of $\dt$).
This implies that the extremal control function 
oscillates more and more rapidly as $\dt \rightarrow 0$.

\subsection{An illuminating case}

We consider families of two matrices.
The following result is important to clarify the phenomenon.

\begin{lemma}
\label{lem:trans} Let ${\cB} = \{ B_1, B_2 \}$ with
$B_1=B_2^{\rm T}$, then it holds
$
\rho({\cB}) = \sqrt{\rho(B_1\,B_2)}
$
\end{lemma}

\begin{proof}
By well-known inequalities \cite{DL} 
we have
\begin{eqnarray*}
&& \sqrt{\rho\left( B_1\,B_2 \right)} \,\le\, \rho({\cB}) \,\le\,
\| \cB \|_2 \,=\, \max\{ \| B_1 \|_2, \| B_2 \|_2 \}.
\end{eqnarray*}
Using the assumption, the result follows from the equality
$\rho\left( B_1\,B_2 \right) = \rho\left( B_1^{\rm T}\,B_1 \right) = \| B_1 \|_2^2 = \| B_2 \|_2^2.$
\qed
\end{proof}

\begin{corollary}
Every $d \times d$ family of matrices $\cA = \{ A_1, A_2 \}$ with $A_2 = A_1^{\rm T}$
shows the phenomenon of fibrillation.
\end{corollary}
\begin{proof}
It is sufficient to observe that the family $\{ e^{\dt A_1}, e^{\dt A_2} \}$ fulfils assumptions of Lemma 
\ref{lem:trans} so that the s.m.p. has length $2$ independently of $\dt$. \qed
\end{proof}

\subsubsection*{Illustrative example.}

We consider (\ref{main}) with $d=2$, $m=2$ and
\begin{eqnarray}
&& A_1 = \left( \begin{array}{rr} 0 & 1 \\ 0 & 0 \end{array} \right), \qquad
A_2 = \left( \begin{array}{rr} 0 & 0 \\ 1 & 0 \end{array} \right),
\nonumber
\end{eqnarray}
with the aim to approximate $\sigma (\cA)$.
If we consider any $\dt$ we get for the family
\begin{eqnarray}
&& \cB_{\dt} = \{ B_{1,\dt}, B_{2,\dt} \} := \{ \e^{A_1 \dt}, \e^{A_2 \dt} \},
\qquad
B_{1,\dt} = \left( \begin{array}{rr} 1 & \dt \\ 0 & 1 \end{array} \right), \qquad
B_{2,\dt} = \left( \begin{array}{rr} 1 & 0 \\ \dt & 1 \end{array} \right),
\nonumber
\end{eqnarray}
that - due to the fact that $B_{2,\dt} = B_{1,\dt}^{\rm T}$ -
$$
\rho(\cB_{\dt}) = \sqrt{\rho(B_{1,\dt} B_{2,\dt})} =
\frac{\sqrt{\dt ^2+\sqrt{\dt^2+4} \dt +2}}{\sqrt{2}}.
$$
It follows that
$$
\frac{1}{\dt} \log \left( \rho(\cB_{\dt}) \right) =
\frac{1}{2} - \frac{1}{48} \dt^2 + \cO(\dt^4)
\nearrow \frac{1}{2} \ \mbox{as} \ \dt \rightarrow 0^+,
$$
yielding ${\sigma}(\cA) = \frac{1}{2}$.
To interpret this result we make use of the following Lemma, which follows frome the well-known 
Lie Trotter product formula.
\begin{lemma}
Let $A(\theta) = \theta A_1 + (1-\theta) A_2$. Then 
$$
\lim\limits_{k \rightarrow \infty} \left( \e^{\frac{\theta}{k} A_1}\,\e^{\frac{1-\theta}{k} A_2}
\right)^{k} =  \e^{A(\theta)}.
$$
\end{lemma}
Choosing $\dt = \displaystyle{ \frac{1}{2 k} }$ we have from the lemma,
$$
\displaystyle{
( B_{1,\dt} B_{2,\dt} )^k = ( B_{1,\dt} B_{2,\dt} )^\frac{1}{2 \dt}
\approx \e^{\frac{A_1 + A_2}{2}}}.
$$
which gives
$
{\sigma} (\cA)
= \sigma\left( A(1/2) \right) = \frac{1}{2},
$
where $\sigma(C)$ denotes the spectral abscissa of a matrix $C$,
that is the largest real part of eigenvalues of $C$.
We can interpret fibrillation as the fact that at every instant the maximal growth
would be obtained by taking both matrices, that is a multivalued control function,
which in turn is equivalent to consider a convex combination of the vector fields.
%

%
%
%

A natural open issue concerns the search of conditions which determine fibrillation and
understanding its possible non genericity.
\begin{remark}
Finally we observe that fibrillation cannot occur if we consider generalized trajectories
(obtained replacing $\cA$ by ${\rm co}(\cA)$ (convex hull of $\cA$)) as it follows
from  Theorem~A of N.Barabanov.
In the given example in fact, the critical control function would be constant with
$A(u(t)) = \ds{\frac{A_1+A_2}{2}} \ \mbox{for all} \ t.$
\end{remark}

\section{Illustrative cases and numerical examples.}

\subsection{Illustrative test problems}


We provide some illustrative examples and compare the obtained results
by those achieved by looking for a common quadratic Lyapunov function (CQLF). We have made use
of the Yalmip Matlab package to compute an optimal CQLF. 

\subsubsection{Example 1: (dimension $3$).}
\label{ex:1}

We consider the following example proposed by Jungers and Protasov \cite{PJ}.
Let ${\cA}=\{ A_1, A_2 \}$ with
\begin{eqnarray} 
A_1 & = & \left(\begin{array}{rrr}
   -0.0822  &  0.0349 &  -0.1182 \\
    0.0953  & -0.0897 &  -0.1719 \\
    0.0787  &  0.0223 &  -0.2781
\end{array}\right), \qquad
A_{2} \, = \, \left( \begin{array}{rrr} 
    0.1391  &  0.1397 &  -0.0916 \\
    0.0338  & -0.1769 &  -0.0707 \\
    0.7417  &  0.3028 &  -0.4621
\end{array}\right).
\nonumber
\end{eqnarray}


Looking for a starting product $P$ of length $\ell \le 100$ we obtain the results shown in Table \ref{tab:1}.
The first columns reports $\dt$, the second and third columns denote the computed lower and upper bounds, the fourth columns the amplitude $\gamma$ of the interval containing the exact value, the fifth column provides the starting product, the sixth column the used $\eps$-value and the
seventh column the number of vertices of the computed $\eps$-extremal polytope.
\begin{table}[h]
\caption{Approximation of the Lyapunov exponent}
\label{tab:1}
\begin{center}
\begin{tabular}{|c||c||c||c||c|c|c|}
\hline
$\dt$ & $\beta$ & $\alpha$ & $\gamma$ & $\Pi$ & $\eps$ & $\# V$ \\[0.07cm]
\hline
$1/2$  & $-0.0470$ & $0.0074$ & $0.0545$ & $B_1^{27} B_2^{29}$  & $0.05$ & $163$ \\[0.07cm]
\hline
$1/2$  & $-0.0470$ & $-0.0148$ & $0.0322$ & $B_1^{27} B_2^{29}$ & $0.025$ & $332$ \\[0.07cm]
\hline
$1/4$  & $-0.0410$ & $0.0089$ & $0.0550$ & $B_1^{55} B_2^{58}$   & $0.0125$ & $423$ \\[0.07cm]
\hline
$1/4$  & $-0.0410$ & $-0.0243$ & $0.0227$ & $B_1^{55} B_2^{58}$    & $0.005$  & $1655$ \\[0.07cm]
\hline
\end{tabular}
\end{center}
\end{table}

Applying the algorithm for the search of a CQLF, we obtain $\alpha=7 \cdot 10^{-5} > 0$
(which does not guarantee uniform stability); referring to the value $\beta=-0.0410$, this would correspond 
to the value $\gamma=0.0480$.

Example \ref{ex:1} puts in evidence that through the CQLF approach it is not possible to decide stability of
the system since the lower bound for the Lyapunov exponent is negative while the upper bound is positive.
This is due to the fact the extremal norm for $\cA$ is (in general) non quadratic.

\subsubsection{Example 2: (dimension $5$).}
\label{ex:2}

Let ${\cA}=\{ A_1, A_2 \}$ with
\begin{eqnarray} 
A_1 & = & \left(\begin{array}{rrrrr}
 -0.9 & -1.0 & -1.0 & -1.0 &    0 \\
    0 & -0.9 & -1.0 & -1.0 &    0 \\
 -1.0 & -1.0 & -0.9 &    0 &    0 \\
    0 &    0 & -1.0 & -1.9 & -1.0 \\
    0 &    0 & -1.0 &    0 & -1.9
\end{array}\right), \qquad
A_{2} \, = \, \left( \begin{array}{rrrrr} 
 -0.9 & -1.0 &    0 &    0 &    0 \\
    0 & -1.9 & -1.0 & -1.0 & -1.0 \\
    0 &    0 & -0.9 &    0 &    0 \\
 -1.0 &    0 &    0 & -1.9 &    0 \\
    0 & -1.0 &    0 & -1.0 & -0.9
\end{array}\right).
\nonumber
\end{eqnarray}

Applying the algorithm for the search of a CQLF we do not find a positive semidefinite 
matrix $M$ such that 
\[
A_i^\tp M + M A_i \preceq 0
\]
which means we cannot state the uniform stability of the associated switched system by
means of an ellipsoid norm. 

Nevertheless we can prove stability by means of Algorithm (R), but this
is obtained only for a small $\dt$ and at a high computational cost (the overall procedure
employed several hours of computation). 

\begin{table}[h]
\caption{Approximation of the Lyapunov exponent}
\label{tab:2}
\begin{center}
\begin{tabular}{|c||c||c||c||c|c|c|}
\hline
$\dt$ & $\beta$ & $\alpha$ & ${\gamma}$ & $\Pi$ & $\eps$ & $\# V$ \\[0.07cm]
\hline 
$1/10$  & $-0.1372$ & $0.3927$  & $0.530$ & $B_1^{90} B_2^{39}$      & $0.02$   & $772$ \\[0.07cm]
\hline
$1/20$  & $-0.1372$ & $0.1927$  & $0.320$ & $B_1^{180} B_2^{78}$     & $0.01$   & $2666$ \\[0.07cm]
\hline
$1/40$  & $-0.1372$ & $0.1030$  & $0.243$ & $B_1^{359} B_2^{156}$    & $0.005$  & $9351$ \\[0.07cm]
\hline
$1/100$ & $-0.1372$ & $-0.0422$ & $0.095$ & $B_1^{898} B_2^{390}$    & $0.0025$ & $23885$ \\[0.07cm]
\hline
\end{tabular}
\end{center}
\end{table}

Example \ref{ex:2} emphasizes that in order to achieve a high accuracy in the approximation of the Lyapunov
exponent it is necessary to accept a significant computational cost. However, since the CQLF method does not 
provide a negative upper bound for the Lyapunov exponent such a computational effort is necessary to obtain 
the stability result.

\subsubsection{Example 3: positive system (dimension $3$).}
\label{ex:3}

We consider the following well-known example proposed by Margallot et al \cite{FMC},
which for convenience we shift by the identity.
Let ${\cA}=\{ A_1, A_2 \}$ with
\begin{eqnarray} 
A_1 & = & \left(\begin{array}{rrr}
-2  &  0  &   0 \\
10  & -2  &   0 \\
 0  &  0  & -11
\end{array}\right), \qquad
A_{2} \, = \, \left( \begin{array}{rrr} 
-11  &   0 & 10 \\
  0  & -11 &  0 \\
  0  &  10 & -2
\end{array}\right).
\nonumber
\end{eqnarray}

In Table \ref{tab:3} we report the results obtained by applying Algorithms (P), part 1 and 2
(note that in all cases the s.m.p. is found to be of the form $B_1^k B_2^{n-k}$).


\begin{table}[h]
\caption{Approximation of the Lyapunov exponent}
\label{tab:3}
\begin{center}
\begin{tabular}{|c||c||c||c||c|c|c|}
\hline
$\dt$ & $\beta$ & $\alpha$ & ${\gamma}$ & $\Pi$ & $\eps$ & $\# V$ \\[0.07cm]
\hline
$1/16$  & $-0.0462$ & $0.7168$ & $0.763$ & $B_1^8 B_2^5$         & $0$ & $13$ \\[0.07cm]
\hline
$1/32$  & $-0.0442$ & $0.2548$ & $0.299$ & $B_1^{16} B_2^9$      & $0$ & $34$ \\[0.07cm]
\hline
$1/64$  & $-0.0428$ & $0.1302$ & $0.173$ & $B_1^{31} B_2^{19}$   & $0$ & $83$ \\[0.07cm]
\hline
$1/128$ & $-0.0427$ & $0.0425$ & $0.085$ & $B_1^{62} B_2^{37}$   & $0$ & $165$ \\[0.07cm]
\hline
$1/256$ & $-0.0426$ & $-0.0006$ & $0.042$ & $B_1^{125} B_2^{75}$  & $0$ & $587$ \\[0.07cm]
\hline
$1/512$ & $-0.0426$ & $-0.0175$ & $0.025$ & $B_1^{249} B_2^{149}$ & $0$ & $2228$ \\[0.12cm]
\hline
\end{tabular}
\end{center}
\end{table}

Table \ref{tab:3} shows that the system $\cA$ is stable (this is seen already for
$\dt=1/256$).

Applying the algorithm for the search of a CQLF, we obtain referring to $\beta=-0.0426$ a value
$\alpha=0.2894$ which implies $\gamma=0.332$. So, already for $\dt=1/32$, the polytope with $34$
vertices constructed by Algorithm (P) gives a better estimate than CQLF.
The value $\gamma$ associated to the CQLF method is quite large here and the polytope method
outperforms the quadratic one since it allows to assess stability of the system computing a polytope
with a moderate number of vertices and thus quite efficiently.

%
%


\subsection{Example 4: positive system (dimension $3$).}
\label{ex:4}

This example is inspired by \cite{AR}. Let ${\cA}=\{ A_1, A_2 \}$ with
\begin{eqnarray}
A_1 & = & \left(\begin{array}{rrr}
-1 & 1/10 & 1/10 \\
1/10 & -1 & 1/10 \\
1/6 & 1/6 & -1/3
\end{array}\right), \qquad
A_{2} \, = \, \left( \begin{array}{rrr}
-1/2 & 1/10 & 9/8 \\
1/6 & -1/3 & 7/8 \\
1/10 & 1/10 & -1
\end{array}\right).
\nonumber
\end{eqnarray}
In Table \ref{tab:4} we report the results obtained by applying Algorithms (P), part 1 and part 2
(note that in all cases the s.m.p. is found to be $B_{2}$. 

\subsubsection*{Lyapunov exponent.}

Table \ref{tab:4} shows that the system is stable (this is seen already for
$\dt=1/2$).

\begin{table}[h]
\caption{Approximation of the Lyapunov exponent}
\label{tab:4}
\begin{center}
\begin{tabular}{|c||c||c||c||c|c|c|}
\hline
$\dt$  & $\beta$     & $\alpha$  & ${\gamma}$ & $\Pi$ & $\eps$ & $\# V$ \\[0.07cm]
\hline
   $1$ & $-0.061107$ & $0.07500$   & $0.136$  & $B_2$ & $0$ & $3$  \\[0.07cm]
\hline
 $1/2$ & $-0.061107$ & $-0.003891$ & $0.0571$ & $B_2$ & $0$ & $4$  \\[0.07cm]
\hline
 $1/8$ & $-0.061107$ & $-0.047604$ & $0.0134$ & $B_2$ & $0$ & $13$ \\[0.07cm]
\hline
$1/16$ & $-0.061110$ & $-0.054375$ & $0.0067$ & $B_2$ & $0$ & $26$ \\[0.07cm]
\hline
$1/32$ & $-0.061107$ & $-0.057489$ & $0.0033$ & $B_2$ & $0$ & $50$ \\[0.07cm]
\hline
$1/64$ & $-0.061107$ & $-0.058563$ & $0.0025$ & $B_2$ & $0$ & $100$ \\[0.07cm]
\hline 
\end{tabular}
\end{center}
\end{table}

It is interesting to observe that several stability criteria, based on suitable sufficient conditions, are shown in \cite{AR} not to be effective for this problem.
In fact Theorem 4, Theorem 6 and Theorem 7 in \cite{AR} do not apply so that uniform stability cannot be inferred. Nevertheless, the results of our algorithm, reported in Table  \ref{tab:1} show that
the associated system of ODEs is uniformly asymptotically stable.

We also notice that in this case the CQLF method, referring to $\beta = -0.061107$, provides a value $\gamma$ of the same order
of that obtained with $\tau=1/64$.

\subsubsection*{Lower Lyapunov exponent}

Table \ref{tab:4l} provides the results obtained by applying Algorithms 
\ref{algoL} and \ref{algoL2}. 

\begin{table}[h]
\caption{Approximation of the lower Lyapunov exponent}
\label{tab:4l}
\begin{center}
\begin{tabular}{|c||c||c||c||c|c|c|}
\hline
$\dt$  & $\check\beta$     & $\check\alpha$  & ${\check\gamma}$ & $\Pi$ & $\eps$ & $\# V$ \\[0.07cm]
\hline
 $1/4$ & $-0.29023$ & $-0.33453$ & $0.0443$ & $B_1^5 B_2$ & $0.01$ & $24$  \\[0.07cm]
\hline
 $1/8$ & $-0.29073$ & $-0.30843$ & $0.0177$ & $B_1^5 B_2$ & $0.001$ & $37$ \\[0.07cm]
\hline
$1/16$ & $-0.29086$ & $-0.30076$ & $0.0099$ & $B_1^5 B_2$ & $0.0003$ & $134$ \\[0.07cm]
\hline
$1/32$ & $-0.29087$ & $-0.29869$ & $0.0079$ & $B_1^5 B_2B_2$ & $0.0003$ & $252$ \\[0.07cm]
\hline
\end{tabular}
\end{center}
\end{table}

The Lower Lyapunov exponent is estimated from below by $\check\alpha = -0.29869$.

\subsubsection{Example 5: positive system (dimension $8$).}
\label{ex:5}

Consider the randomly generated family $\cA = \{ A_1, A_2 \}$ with

\begin{eqnarray*}
&& \hskip -0.3cm A_1 = {
\left(\begin{array}{rrrrrrrr}
 -15 &  1 &  1 &  0 &  3 &  2 &  0 &  0 \\
   2 & -9 &  3 &  2 &  3 &  1 &  2 &  1 \\
   1 &  3 &-13 &  2 &  1 &  1 &  0 &  3 \\
   2 &  0 &  1 & -7 &  1 &  0 &  0 &  1 \\
   1 &  0 &  1 &  1 & -8 &  0 &  1 &  0 \\
   1 &  3 &  1 &  2 &  3 &-11 &  2 &  2 \\
   1 &  3 &  1 &  3 &  1 &  1 &-10 &  1 \\
   2 &  1 &  3 &  2 &  3 &  2 &  3 &-11
\end{array}\right)}
\\[0.2cm]
&& \hskip -0.3cm A_2 = {
\left(\begin{array}{rrrrrrrr}
 -10 &  2 &  2 &  0 &  1 &  3 &  2 &  0 \\
   0 &-16 &  2 &  1 &  2 &  3 &  1 &  2 \\
   2 &  2 &-14 &  3 &  1 &  2 &  3 &  1 \\
   0 &  3 &  3 &-13 &  3 &  2 &  0 &  0 \\
   3 &  2 &  1 &  2 & -9 &  0 &  1 &  3 \\
   1 &  3 &  0 &  0 &  1 & -7 &  0 &  0 \\
   0 &  2 &  3 &  2 &  2 &  3 &-17 &  2 \\
   2 &  2 &  2 &  2 &  2 &  3 &  2 &-17
\end{array}\right) }
\end{eqnarray*}
\normalsize
with spectral absciss\ae \
$\sigma(A_1) = -0.89470735\ldots$, $\sigma(A_2) = -1.22136422\ldots$.


\subsubsection*{Lyapunov exponent.}

Table \ref{tab:5} reports the obtained computational results.

\begin{table}[h]
\caption{Approximation of the Lyapunov exponent}
\label{tab:5}
\begin{center}
\begin{tabular}{|c||c||c||c||c|c|c|}
\hline
$\dt$  & $\beta$     & $\alpha$  & ${\gamma}$ & $\Pi$ & $\eps$ & $\# V$ \\[0.07cm]
\hline
$1/32$ & $-0.76212368$ & $-0.33813367$ & $0.4239900$
& $B_1^4 B_2^2$ & $0.001$  & $46$ \\[0.07cm]
\hline
$1/64$ & $-0.76207385$ & $-0.56012765$ & $0.2019462$
& $B_1^6 B_2^3$ & $0.001$ & $194$ \\[0.07cm]
\hline
$1/128$ & $-0.76207385$ & $-0.56776133$ & $0.1943125$
& $B_1^8 B_2^4$ & $0.001$ & $256$ \\[0.07cm]
\hline
\end{tabular}
\end{center}
\end{table}
Note that in the last three cases $\eps$-extremal polytopes have been computed (with $\eps=0.001$).

For the case $\dt = 1/32$ we also compute the optimal value $\alpha$ by using the
standard tetrahedron defining the unit ball of the one-norm and compare it to the polytope
${\cP}_{\dt}$ obtained applying Algorithm (P). 
In this case we get $\alpha^* = 11.171105\ldots$ (which would not allow to infer stability of the system) to
be compared to the much smaller value $\alpha_{\dt} = 0.42399\ldots$ computed by Algorithm \ref{algoP2}.

We can conclude asserting that the system is (uniformly)
stable and that the (upper) Lyapunov exponent is smaller than
$\gamma_{1/32} = -0.56012765\ldots.$

Applying the algorithm for the search of a CQLF, we obtain referring to $\beta=-0.7620$ a value
$\alpha=-0.7590$ which implies $\gamma=0.003$, that is an excellent value, outperforming 
the one obtained by our algorithm subject to the choice of parameters in Table \ref{tab:5}.

\subsubsection*{Lower Lyapunov exponent}

We also compute bounds for the lower Lyapunov exponent by computing a polytope
antinorm with Algorithm (L).

Table \ref{tab:5l} provides the results obtained by applying Algorithms 
\ref{algoL} and \ref{algoL2} applied to Example \ref{ex:5}.

The accuracy of the approximations of the classical and the lower Lyapunov exponents
appear to be comparable. The number of vertices is smaller for the lower Lyapunov
exponent also because of the shorter length of the s.l.p. with respect to the s.m.p.
for the considered values of $\dt$.

We may conjecture that the lower Lyapunov exponent of the system is the spectral abscissa
$\sigma(A_2)$.

\begin{table}[h]
\caption{Approximation of the lower Lyapunov exponent}
\label{tab:5l}
\begin{center}
\begin{tabular}{|c||c||c||c||c|c|c|}
\hline
$\dt$  & $\check\beta$     & $\check\alpha$  & ${\check\gamma}$ & $\Pi$ & $\eps$ & $\# V$ \\[0.07cm]
\hline
 $1/32$  & $-1.22136$ & $-1.68279$ & $0.4614$ & $B_2$  & $0.001$   & $16$ \\[0.07cm]
\hline
 $1/64$  & $-1.22136$ & $-1.58316$ & $0.3618$ & $B_2$  & $0.0001$  & $26$ \\[0.07cm]
\hline
 $1/128$ & $-1.22136$ & $-1.38686$ & $0.1655$ & $B_2$ & $0.00001$ & $73$ \\[0.07cm]
\hline
\end{tabular}
\end{center}
\end{table}

\subsubsection{Example 6: positive system (dimension $25$).}
\label{ex:6}

We consider the randomly generated family of sign matrices (where each diagonal element is chosen
uniformly and independently from the set $\{ -1,0,1 \}$ and each off-diagonal element is chosen from
the set $\{ 0,1 \}$), $\cA = \{ A_1, A_2 \}$ with
\begin{eqnarray*}
&& \hskip -0.3cm A_1 = { \scriptsize
\left(\begin{array}{rrrrrrrrrrrrrrrrrrrrrrrrr}
 0 & 1 & 0 & 0 & 1 & 0 & 1 & 1 & 0 & 1 & 1 & 1 & 0 & 0 & 1 & 0 & 0 & 1 & 1 & 1 & 0 & 1 & 1 & 1 & 0 \\ 
 0 &-1 & 0 & 0 & 0 & 1 & 0 & 1 & 1 & 1 & 0 & 0 & 0 & 0 & 0 & 0 & 0 & 0 & 1 & 0 & 1 & 0 & 0 & 1 & 1 \\
 1 & 1 &-1 & 0 & 1 & 0 & 1 & 0 & 0 & 0 & 1 & 0 & 0 & 1 & 0 & 0 & 0 & 1 & 0 & 0 & 1 & 1 & 0 & 1 & 1 \\
 1 & 0 & 1 & 0 & 1 & 0 & 0 & 0 & 1 & 0 & 0 & 0 & 1 & 0 & 0 & 1 & 1 & 0 & 0 & 1 & 0 & 1 & 1 & 0 & 0 \\
 1 & 0 & 0 & 0 & 0 & 1 & 0 & 0 & 1 & 0 & 1 & 0 & 0 & 1 & 1 & 0 & 0 & 0 & 0 & 1 & 1 & 0 & 1 & 0 & 1 \\
 1 & 1 & 1 & 0 & 1 & 1 & 1 & 0 & 0 & 0 & 0 & 0 & 1 & 0 & 1 & 0 & 0 & 0 & 1 & 1 & 0 & 1 & 1 & 1 & 1 \\
 1 & 0 & 0 & 1 & 0 & 1 & 1 & 1 & 1 & 1 & 0 & 1 & 1 & 1 & 1 & 1 & 0 & 1 & 0 & 0 & 0 & 1 & 0 & 0 & 0 \\
 0 & 0 & 0 & 0 & 1 & 0 & 0 &-1 & 0 & 1 & 0 & 0 & 0 & 0 & 0 & 0 & 1 & 0 & 1 & 0 & 1 & 0 & 1 & 0 & 1 \\
 0 & 0 & 1 & 0 & 0 & 1 & 0 & 0 & 0 & 1 & 0 & 0 & 0 & 1 & 0 & 1 & 1 & 1 & 1 & 1 & 1 & 0 & 0 & 1 & 1 \\
 1 & 0 & 0 & 1 & 0 & 1 & 1 & 1 & 0 & 1 & 1 & 1 & 1 & 0 & 0 & 0 & 0 & 1 & 0 & 1 & 1 & 0 & 1 & 0 & 1 \\
 1 & 1 & 0 & 1 & 0 & 1 & 0 & 1 & 0 & 1 &-1 & 1 & 0 & 0 & 1 & 0 & 1 & 0 & 0 & 0 & 0 & 0 & 0 & 1 & 1 \\
 0 & 1 & 0 & 0 & 1 & 0 & 1 & 0 & 1 & 0 & 1 & 0 & 1 & 0 & 1 & 0 & 0 & 0 & 0 & 1 & 1 & 0 & 0 & 0 & 1 \\
 1 & 0 & 1 & 0 & 0 & 0 & 0 & 1 & 1 & 0 & 1 & 1 & 1 & 1 & 1 & 0 & 0 & 1 & 0 & 1 & 1 & 0 & 1 & 0 & 0 \\
 1 & 1 & 1 & 1 & 1 & 0 & 0 & 1 & 0 & 1 & 0 & 0 & 0 & 1 & 1 & 0 & 1 & 1 & 0 & 1 & 1 & 1 & 0 & 0 & 0 \\
 1 & 1 & 1 & 1 & 0 & 0 & 0 & 0 & 0 & 0 & 0 & 1 & 1 & 1 &-1 & 1 & 1 & 1 & 0 & 0 & 1 & 1 & 0 & 1 & 1 \\
 1 & 1 & 0 & 0 & 0 & 0 & 1 & 1 & 0 & 1 & 1 & 1 & 0 & 0 & 1 & 1 & 1 & 1 & 1 & 0 & 1 & 0 & 0 & 1 & 1 \\
 1 & 1 & 1 & 0 & 0 & 1 & 1 & 1 & 1 & 1 & 0 & 1 & 1 & 1 & 1 & 1 &-1 & 0 & 0 & 1 & 1 & 1 & 1 & 1 & 0 \\
 1 & 0 & 1 & 1 & 1 & 0 & 1 & 0 & 1 & 0 & 0 & 1 & 0 & 0 & 1 & 0 & 0 & 0 & 1 & 1 & 0 & 0 & 1 & 1 & 0 \\
 1 & 1 & 1 & 0 & 0 & 0 & 0 & 0 & 1 & 1 & 0 & 0 & 0 & 0 & 0 & 0 & 0 & 1 & 0 & 0 & 0 & 1 & 0 & 0 & 0 \\
 1 & 0 & 0 & 0 & 1 & 0 & 0 & 0 & 1 & 0 & 0 & 1 & 1 & 1 & 0 & 0 & 0 & 0 & 0 & 1 & 1 & 1 & 1 & 1 & 1 \\
 1 & 1 & 0 & 0 & 0 & 0 & 1 & 1 & 0 & 0 & 1 & 1 & 0 & 1 & 1 & 1 & 1 & 1 & 0 & 1 & 1 & 0 & 1 & 0 & 0 \\
 0 & 1 & 0 & 1 & 0 & 1 & 1 & 1 & 1 & 1 & 0 & 0 & 1 & 0 & 1 & 1 & 1 & 0 & 1 & 1 & 1 & 1 & 0 & 0 & 0 \\
 1 & 1 & 1 & 0 & 1 & 1 & 0 & 1 & 1 & 0 & 1 & 1 & 0 & 0 & 1 & 0 & 0 & 1 & 0 & 0 & 0 & 0 & 1 & 1 & 1 \\
 1 & 1 & 1 & 0 & 1 & 0 & 0 & 0 & 1 & 1 & 1 & 0 & 0 & 0 & 1 & 1 & 0 & 0 & 1 & 1 & 0 & 0 & 1 &-1 & 0 \\
 0 & 1 & 1 & 1 & 0 & 0 & 0 & 0 & 1 & 1 & 1 & 0 & 0 & 1 & 0 & 0 & 1 & 1 & 0 & 1 & 0 & 0 & 0 & 0 & 0 \\
\end{array}\right)}
\\[0.2cm]
&& \hskip -0.3cm A_2 = {\scriptsize
\left(\begin{array}{rrrrrrrrrrrrrrrrrrrrrrrrr}
  -1 & 0 & 0 & 0 & 1 & 1 & 0 & 1 & 1 & 0 & 1 & 1 & 0 & 1 & 0 & 0 & 1 & 1 & 1 & 0 & 0 & 1 & 0 & 0 & 0 \\
   1 & 0 & 1 & 1 & 0 & 0 & 1 & 1 & 1 & 0 & 1 & 1 & 0 & 1 & 0 & 0 & 1 & 1 & 1 & 1 & 1 & 0 & 1 & 1 & 1 \\
   1 & 1 & 0 & 1 & 0 & 0 & 1 & 1 & 1 & 1 & 0 & 0 & 0 & 0 & 0 & 1 & 0 & 0 & 1 & 1 & 0 & 0 & 1 & 1 & 1 \\
   1 & 1 & 1 & 0 & 1 & 0 & 1 & 1 & 1 & 0 & 1 & 0 & 0 & 0 & 1 & 1 & 0 & 1 & 0 & 0 & 1 & 1 & 1 & 1 & 1 \\
   0 & 1 & 0 & 1 & 0 & 1 & 0 & 1 & 1 & 1 & 1 & 1 & 1 & 0 & 1 & 0 & 0 & 0 & 0 & 0 & 1 & 1 & 1 & 0 & 0 \\
   1 & 1 & 1 & 1 & 1 &-1 & 1 & 0 & 0 & 0 & 1 & 0 & 1 & 1 & 1 & 0 & 0 & 0 & 1 & 0 & 0 & 1 & 1 & 0 & 1 \\
   0 & 0 & 0 & 1 & 0 & 1 &-1 & 0 & 0 & 1 & 1 & 0 & 0 & 0 & 0 & 0 & 0 & 0 & 0 & 0 & 0 & 0 & 0 & 0 & 1 \\
   0 & 0 & 1 & 1 & 1 & 1 & 0 &-1 & 0 & 1 & 0 & 1 & 0 & 1 & 0 & 1 & 1 & 1 & 1 & 0 & 1 & 1 & 1 & 0 & 0 \\
   0 & 0 & 0 & 0 & 1 & 0 & 1 & 1 & 0 & 1 & 1 & 1 & 0 & 1 & 1 & 0 & 0 & 0 & 0 & 1 & 1 & 1 & 1 & 1 & 0 \\
   1 & 0 & 0 & 0 & 1 & 1 & 0 & 0 & 1 & 0 & 0 & 0 & 0 & 0 & 0 & 1 & 0 & 0 & 1 & 1 & 0 & 1 & 0 & 1 & 1 \\
   1 & 1 & 1 & 0 & 1 & 1 & 1 & 0 & 0 & 0 &-1 & 1 & 1 & 1 & 1 & 1 & 0 & 1 & 1 & 0 & 0 & 0 & 1 & 0 & 1 \\
   0 & 1 & 1 & 1 & 0 & 0 & 0 & 1 & 1 & 1 & 0 &-1 & 0 & 0 & 1 & 1 & 1 & 0 & 1 & 0 & 0 & 1 & 0 & 0 & 0 \\
   0 & 0 & 1 & 1 & 0 & 1 & 0 & 0 & 1 & 1 & 0 & 0 & 1 & 1 & 1 & 0 & 1 & 0 & 0 & 0 & 1 & 0 & 1 & 0 & 0 \\
   0 & 0 & 0 & 0 & 0 & 0 & 0 & 0 & 1 & 1 & 0 & 0 & 0 & 0 & 1 & 1 & 1 & 0 & 0 & 0 & 1 & 0 & 0 & 0 & 0 \\
   0 & 0 & 1 & 0 & 0 & 0 & 0 & 1 & 0 & 0 & 1 & 1 & 1 & 0 & 0 & 1 & 1 & 0 & 1 & 0 & 1 & 1 & 1 & 1 & 0 \\
   1 & 1 & 1 & 0 & 0 & 0 & 1 & 0 & 1 & 0 & 1 & 1 & 1 & 1 & 0 & 1 & 0 & 1 & 1 & 0 & 0 & 0 & 0 & 1 & 1 \\
   1 & 0 & 1 & 0 & 0 & 1 & 0 & 0 & 0 & 1 & 0 & 1 & 0 & 0 & 1 & 1 & 0 & 1 & 0 & 1 & 0 & 0 & 0 & 0 & 0 \\
   0 & 0 & 1 & 0 & 0 & 1 & 1 & 1 & 1 & 1 & 1 & 1 & 1 & 1 & 0 & 1 & 0 &-1 & 1 & 0 & 1 & 1 & 1 & 0 & 0 \\
   0 & 1 & 0 & 0 & 1 & 1 & 1 & 0 & 0 & 0 & 0 & 0 & 0 & 1 & 0 & 0 & 1 & 1 & 0 & 1 & 1 & 0 & 1 & 0 & 0 \\
   0 & 0 & 1 & 0 & 0 & 1 & 0 & 1 & 0 & 1 & 0 & 1 & 1 & 1 & 1 & 0 & 1 & 1 & 0 & 0 & 0 & 0 & 0 & 1 & 1 \\
   0 & 0 & 1 & 0 & 0 & 1 & 0 & 0 & 0 & 1 & 1 & 0 & 1 & 0 & 1 & 1 & 1 & 1 & 1 & 1 & 0 & 1 & 1 & 0 & 0 \\
   0 & 1 & 1 & 1 & 1 & 1 & 1 & 1 & 1 & 1 & 0 & 0 & 0 & 1 & 1 & 1 & 0 & 1 & 0 & 1 & 0 & 1 & 1 & 0 & 1 \\
   0 & 0 & 0 & 0 & 0 & 0 & 0 & 0 & 0 & 1 & 1 & 1 & 0 & 1 & 0 & 1 & 0 & 0 & 1 & 0 & 1 & 0 &-1 & 1 & 1 \\
   1 & 1 & 0 & 1 & 0 & 1 & 0 & 1 & 1 & 0 & 1 & 1 & 0 & 1 & 1 & 0 & 0 & 1 & 0 & 1 & 1 & 1 & 1 & 0 & 1 \\
   0 & 0 & 1 & 1 & 0 & 0 & 1 & 0 & 0 & 0 & 1 & 0 & 1 & 1 & 1 & 1 & 1 & 1 & 1 & 0 & 1 & 0 & 0 & 0 & 1 \\
\end{array}\right)}
\end{eqnarray*}

\subsubsection*{Joint Lyapunov exponent}

We consider a dwell time $\dt = 1/16, 1/32, 1/64$ and $1/128$ respectively. The obtained results are reported in Table
\ref{tab:6}.

\begin{table}[h]
\caption{Approximation of the Lyapunov exponent}
\label{tab:6}
\begin{center}
\begin{tabular}{|c||c||c||c||c|c|c|}
\hline
$\dt$  & $\beta$     & $\alpha$  & ${\gamma}$ & $\Pi$ & $\eps$ & $\# V$ \\[0.07cm]
\hline
 $1/16$ & $11.983$ & $12.541$    & $0.557$  & $B_1 B_2^2$ & $0.001$   & $30$  \\[0.07cm]
\hline
 $1/16$ & $11.983$ & $12.541$    & $0.557$  & $B_1 B_2^2$ & $0.0001$  & $30$  \\[0.07cm]
\hline
 $1/16$ & $11.983$ & $12.532$    & $0.549$  & $B_1 B_2^2$ & $0.00001$ & $164$  \\[0.07cm]
\hline
 $1/32$ & $11.985$ & $12.284$    & $0.300$  & $B_1 B_2^2$ & $0.001$ & $103$ \\[0.07cm]
\hline
 $1/32$ & $11.985$ & $12.274$    & $0.290$  & $B_1 B_2^2$ & $0.0001$ & $728$ \\[0.07cm]
\hline
 $1/64$ & $11.985$ & $12.143$ & $0.158$ & $B_1 B_2^2$ & $0.001$ & $210$ \\[0.07cm]
\hline 
 $1/128$ & $11.985$ & $12.078$ & $0.093$ & $B_1 B_2$ & $0.0003$ & $1470$ \\[0.07cm]
\hline 
\end{tabular}
\end{center}
\end{table}

\subsubsection*{Lower Lyapunov exponent}

Similarly we compute bounds for the lower Lyapunov exponent using the same dwell times.
The results are reported in Table \ref{tab:6L}.

\begin{table}[h]
\caption{Approximation of the lower Lyapunov exponent}
\label{tab:6L}
\begin{center}
\begin{tabular}{|c||c||c||c||c|c|c|}
\hline
$\dt$  & $\check\beta$     & $\check\alpha$  & ${\check\gamma}$ & $\Pi$ & $\eps$ & $\# V$ \\[0.07cm]
\hline
 $1/16$ & $11.943$ & $11.453$    & $0.490$  & $B_1$ & $0.001$   & $28$  \\[0.07cm]
\hline
 $1/16$ & $11.943$ & $11.454$    & $0.489$  & $B_1$ & $0.0001$  & $33$  \\[0.07cm]
\hline
 $1/32$ & $11.943$ & $11.681$    & $0.262$  & $B_1$ & $0.001$ & $86$ \\[0.07cm]
\hline
 $1/64$ & $11.943$ & $11.777$ & $0.166$ & $B_1$ & $0.001$ & $187$ \\[0.07cm]
\hline 
 $1/128$ & $11.943$ & $11.838$ & $0.105$ & $B_1$ & $0.0003$ & $598$ \\[0.07cm]
\hline 
\end{tabular}
\end{center}
\end{table}

Note that in both cases we achieve a good approximation of the classical and lower Lyapunov
exponents (an interval of length about $0.1$) at a moderate cost (the $\eps$-extremal polytopes 
have a few hundreds of vertices).

\subsection*{Discussion}

We have proposed some examples to analyze how our methods work in different cases. 
Statistically, we can see that when $\dt$ is smaller, the applied algorithms take more time to compute the spectral maximizing or minimizing product. 
We note that we need more time to compute a spectrum product when his degree is higher. 
In addition, when $\dt$ decreases to $0$, number of vertices of the extremal polytope norm (or antinorm) increases, and the algorithm takes a longer time to compute it. 
However, thanks to the proposed approach, we are able to get an approximation for the joint Lyapunov exponent, and respectively for the lower Lyapunov exponent, the accuracy of which depends on $\dt$.
The smaller $\dt$ is, the closer the bounds are, so that the estimation is better. 
Finally, applying  the just introduced algorithms to the matrices associated with a switched system, we are able to say if the latter is stable (respectively stabilizable) or not.

\subsection{Numerical results and statistics}
\label{sec:stat}

We summarize here the numerical results obtained on a set of test problems.
We remark that we are able to deal with problems with dimension equal to some tens.
The computational complexity is that of repeatedly solving a sequence of LP problems.

\subsubsection{Computation of the Lyapunov exponent}

We consider Metzler matrices of different dimensions, $D=25, 50$ and $100$. Moreover we consider both the cases where
the matrices are sign-matrices (i.e. whose entries belong to $\{-1,0,1\}$ and also the case where the entries are
normally distributed in $[-1,1]$).

We used the following notation:
\begin{itemize}
\item[(ia) ]   $L_{\min}$ denotes the minimal length of s.m.p.; 
\item[(ib) ]   $L_{\max}$ denotes the maximal length of s.m.p.;
\item[(iia) ]  $\# V_{\min}$ indicates the minimal number of vertices of the computed polytope;
\item[(iib) ]  $\# V_{\max}$ indicates the maximal number of vertices of the computed polytope; 
\item[(iic) ]  $< \# V >$ indicates the average number of vertices of the computed polytope; 
\item[(iiia) ] $\gamma_{\min}$ indicates the minimal amplitude of the computed interval containing the Lyapunov exponent; 
\item[(iiib) ] $\gamma_{\max}$ indicates the maximal amplitude of the computed interval containing the Lyapunov exponent; 
\item[(iiic) ] $< \gamma >$ indicates the mean amplitude of the computed interval containing the Lyapunov exponent.
\end{itemize}

We first consider Metzler matrices of dimension $d=25$, with randomly chosen entries in $\{-1,0,1\}$.
We report the results obtained on a set of $10$ examples, making use of $\eps = \frac{1}{250}$ when $\dt= \frac{1}{4}$, 
$\eps = \frac{1}{500}$ when $\dt= \frac{1}{16}$ and $\eps = \frac{1}{1000}$ when $\dt= \frac{1}{64}$.
In Table \ref{tab:NE1} we report the obtained results.
Similarly, in Table \ref{tab:NE2} we consider Metzler matrices still of dimension $d=25$ with randomly selected real entries in
the interval $[-1,1]$.
We report the results obtained on $10$ examples with $\eps = \frac{1}{250}$ when $\dt= \frac{1}{4}$, $\eps = \frac{1}{500}$ when 
$\dt= \frac{1}{16}$, $\eps = \frac{1}{1000}$ when $\dt= \frac{1}{64}$ and $\eps = \frac{1}{2000}$ when $\dt= \frac{1}{128}$.

\begin{samepage}
\begin{table}[h]
\caption{Statistics on LE computation for Metzler problems of dimension $d=25$ with integer entries in $\{ -1, 0 , 1 \}$}
\label{tab:NE1}
\begin{center}
\begin{tabular}{|c|c|c|c|c|c|c|c|c|}
  \hline
  $\dt$ & $L_{\min}$ & $L_{\max}$ & $\# V_{\min}$ & $\# V_{\max}$ & $< \# V >$ & $\gamma_{\min}$ & $\gamma_{\max}$ & $< \gamma >$ \\
  \hline \hline
  1/4 & 1 & 3 & 2 & 4 & 3 & 0.801 & 1.435 & 1.2195 \\
  \hline
  1/16 & 1 & 11 & 9 & 26 & 17 & 0.253 & 0.555 & 0.4197 \\
  \hline
  1/64 & 1 & 15 & 58 & 172 & 106 & 0.098 & 0.162 & 0.1371 \\
  \hline
\end{tabular}
\end{center}
\end{table}

\begin{table}[h]
\caption{Statistics on LE computation for Metzler problems of dimension $d=25$ with entries in $[-1,1]$}
\label{tab:NE2}
\begin{center}
\begin{tabular}{|c|c|c|c|c|c|c|c|c|}
  \hline
  $\dt$ & $L_{\min}$ & $L_{\max}$ & $\# V_{\min}$ & $\# V_{\max}$ & $< \# V >$ & $\gamma_{\min}$ & $\gamma_{\max}$ & $< \gamma >$ \\
  \hline \hline
  1/4 & 1 & 3 & 2 & 4 & 3 & 0.339 & 0.918 & 0.5124 \\
  \hline
  1/16 & 1 & 8 & 9 & 15 & 12 & 0.116 & 0.307 & 0.1829 \\
  \hline
  1/64 & 1 & 19 & 16 & 82 & 48 & 0.077 & 0.115 & 0.0941 \\
  \hline
	1/128 & 1 & 11 & 32 & 194 & 109 & 0.064 & 0.092 & 0.0809 \\
  \hline
\end{tabular}
\end{center}
\end{table}
\end{samepage}

Next we consider Metzler matrices of dimension $d=50$ and entries randomly chosen in $\{ -1,0,1 \}$.
We have analyzed 10 examples with $\eps = \frac{1}{250}$ when $\dt= \frac{1}{4}$, $\eps = \frac{1}{500}$ when $\dt= \frac{1}{16}$, 
$\eps = \frac{1}{1000}$ when $\dt= \frac{1}{64}$ and $\eps = \frac{1}{2000}$ when $\dt= \frac{1}{128}$.
The results are reported in Table \ref{tab:NE3}.
Then we consider Metzler matrices of dimension $d=50$ and entries randomly chosen in $[-1,1]$.
The results are reported in Table \ref{tab:NE4}.

\begin{samepage}
\begin{table}[h]
\caption{Statistics on Metzler problems of dimension $d=50$ with entries in $\{-1,0,1\}$}
\label{tab:NE3}
\begin{center}
\begin{tabular}{|c|c|c|c|c|c|c|c|c|}
  \hline
  $\dt$ & $L_{\min}$ & $L_{\max}$ & $\# V_{\min}$ & $\# V_{\max}$ & $< \# V >$ & $\gamma_{\min}$ & $\gamma_{\max}$ & $< \gamma >$ \\
  \hline \hline
  1/4 & 1 & 2 & 2 & 2 & 2 & 1.312 & 2.229 & 1.6485 \\
  \hline
  1/16 & 1 & 7 & 5 & 10 & 8 & 0.595 & 1.1195 & 0.8891 \\
  \hline
  1/64 & 1 & 11 & 21 & 71 & 51 & 0.178 & 0.387 & 0.2958 \\
  \hline
\end{tabular}
\end{center}
\end{table}

\begin{table}[h]
\caption{Statistics on LE computation for Metzler problems of dimension $d=50$ with entries in $[-1,1]$}
\label{tab:NE4}

\begin{center}
\resizebox{1\textwidth}{!}{
\begin{tabular}{|c|c|c|c|c|c|c|c|c|}
  \hline
  $\dt$ & $L_{\min}$ & $L_{\max}$ & $\# V_{\min}$ & $\# V_{\max}$ & $< \# V >$ & $\gamma_{\min}$ & $\gamma_{\max}$ & $< \gamma >$ \\
  \hline \hline
  1/4   & 1 & 2  & 2  & 2  & 2  & 0.321  & 0.624 & 0.5115 \\
  \hline
  1/16  & 1 & 3  & 4  & 7  & 5  & 0.1798 & 0.334 & 0.2775 \\
  \hline
  1/64  & 1 & 15 & 11 & 28 & 22 & 0.094  & 0.126 & 0.1143 \\
  \hline
  1/128 & 1 & 15 & 23 & 74 & 55 & 0.079  & 0.097 & 0.0894 \\
  \hline
\end{tabular}}
\end{center}
\end{table}
\end{samepage}

Finally we consider Metzler matrices of dimension $d=100$ and entries randomly chosen in $\{ -1,0,1 \}$.
We have analyzed 10 examples with $\eps = \frac{1}{1000}$ when $\dt= \frac{1}{32}$, $\eps = \frac{1}{2000}$ when $\dt= \frac{1}{64}$, 
$\eps = \frac{1}{4000}$ when $\dt= \frac{1}{128}$ and $\eps = \frac{1}{8000}$ when $\dt= \frac{1}{256}$.
The results are reported in Table \ref{tab:NE5}.
To conclude we consider Metzler matrices of dimension $d=100$ and entries randomly chosen in $[-1,1]$.
The results are reported in Table \ref{tab:NE6}.

\begin{samepage}
\begin{table}[h]
\caption{Statistics on LE computation for Metzler problems of dimension $d=100$ with entries in $\{-1,0,1\}$}
\label{tab:NE5}
\begin{center}
\begin{tabular}{|c|c|c|c|c|c|c|c|c|}
  \hline
  $\dt$ & $L_{\min}$ & $L_{\max}$ & $\# V_{\min}$ & $\# V_{\max}$ & $< \# V >$ & $\gamma_{\min}$ & $\gamma_{\max}$ & $< \gamma >$ \\
  \hline \hline
  1/32  & 1 & 14  & 16  & 86   & 42    & 0.9640  &  1.729  & 1.3485 \\
  \hline
  1/64  & 1 & 15  & 34  & 290  & 181   & 0.4951  &  0.9505 & 0.6910 \\
  \hline
  1/128 & 1 & 31  & 101 & 771  & 451   & 0.2628  &  0.4606 & 0.3402 \\
  \hline
	1/256 & 1 & 44  & 288 & 1490 & 944   & 0.1366  &  0.2880 & 0.2125 \\
  \hline
\end{tabular}
\end{center}
\end{table}

\begin{table}[h]
\caption{Statistics on LE computation for Metzler problems of dimension $d=100$ with entries in $[-1,1]$}
\label{tab:NE6}
\begin{center}
\begin{tabular}{|c|c|c|c|c|c|c|c|c|}
  \hline
  $\dt$ & $L_{\min}$ & $L_{\max}$ & $\# V_{\min}$ & $\# V_{\max}$ & $< \# V >$ & $\gamma_{\min}$ & $\gamma_{\max}$ & $< \gamma >$ \\
  \hline \hline
  1/32  & 1 & 3  & 2   & 2   & 5    & 1.312  & 2.2290  & 1.6490 \\
  \hline
  1/64  & 1 & 6  & 5   & 14  & 18   & 0.595  & 1.1195  & 0.8890 \\
  \hline
  1/128 & 1 & 12 & 28  & 114  & 66  & 0.218  & 0.3927  & 0.2758 \\
  \hline
	1/256 & 1 & 23 & 120 & 370 & 258  & 0.0842 & 0.1730  & 0.1158 \\
  \hline
\end{tabular}
\end{center}
\end{table}
\end{samepage}

\subsubsection{Computation of the Lower Lyapunov exponent}

The parameters in the Tables are the same as those considered in the computation of the Lyapunov exponent.
The only difference is that $L_{\min}$ denotes here the minimal length of the candidate l.m.p. and$L_{\max}$ denotes the 
maximal length of the candidate l.m.p..

We first consider Metzler matrices of dimension $d=25$, with randomly chosen entries in $\{-1,0,1\}$.
We report the results obtained on a set of $10$ examples, making use of $\eps = \frac{1}{250}$ when $\dt= \frac{1}{4}$, 
$\eps = \frac{1}{500}$ when $\dt= \frac{1}{16}$ and $\eps = \frac{1}{1000}$ when $\dt= \frac{1}{64}$.
Similarly, in Table \ref{tab:LE2} we consider Metzler matrices still of dimension $d=25$ with randomly selected real entries in
the interval $[-1,1]$. Again we have analyzed $10$ examples with the same dwell times $\dt$ and $\varepsilon$.

\begin{samepage}
\begin{table}[h]
\caption{Statistics on LLE computation for Metzler problems of dimension $d=25$ with entries in $\{-1,0,1\}$}
\label{tab:LE1}
\begin{center}
\begin{tabular}{|c|c|c|c|c|c|c|c|c|}
  \hline
	$\dt$ & $L_{\min}$ & $L_{\max}$ & $\# V_{\min}$ & $\# V_{\max}$ & $< \# V >$ & $\gamma_{\min}$ & $\gamma_{\max}$ & $< \gamma >$ \\
  \hline \hline
  1/4  & 1 & 2  & 2  & 4   & 3   & 0.847 & 1.522 & 1.1767 \\
  \hline
  1/16 & 1 & 7  & 8  & 30  & 18  & 0.321 & 0.606 & 0.4514 \\
  \hline
  1/64 & 1 & 15 & 25 & 198 & 116 & 0.096 & 0.181 & 0.1409 \\
  \hline
\end{tabular}
\end{center}
\end{table}

\begin{table}[h]
\caption{Statistics on LLE computation for Metzler problems of dimension $d=25$ with entries in $[-1,1]$}
\label{tab:LE2}
\begin{center}
\begin{tabular}{|c|c|c|c|c|c|c|c|c|}
  \hline
  $\dt$ & $L_{\min}$ & $L_{\max}$ & $\# V_{\min}$ & $\# V_{\max}$ & $< \# V >$ & $\gamma_{\min}$ & $\gamma_{\max}$ & $< \gamma >$ \\
  \hline \hline
  1/4   & 1 & 3  & 2  & 4   & 3   & 0.285 & 0.824   & 0.53829 \\
  \hline
  1/16  & 1 & 5  & 7  & 18  & 11  & 0.131 & 0.2792  & 0.20368 \\
  \hline
  1/64  & 1 & 11 & 22 & 90  & 44  & 0.078 & 0.1210  & 0.0955 \\
  \hline
  1/128 & 1 & 11 & 45 & 215 & 102 & 0.066 & 0.0970  & 0.081 \\
  \hline
\end{tabular}
\end{center}
\end{table}
\end{samepage}

Next we consider Metzler matrices of dimension $d=50$ and entries randomly chosen in $\{ -1,0,1 \}$.
We have considered $10$ examples, using $\epsilon = \frac{1}{250}$ for $\dt= \frac{1}{4}$, $\epsilon = \frac{1}{500}$ for $\dt= \frac{1}{16}$, $\epsilon = \frac{1}{1000}$ for $\dt= \frac{1}{64}$ and $\epsilon = \frac{1}{2000}$ for $\dt= \frac{1}{128}$.
The results are reported in Table \ref{tab:LE3}.
Afterwards we considered Metzler pairs of dimension $50$, with random entries in $[-1,1]$ using the same dwell times
$\dt$ and $\varepsilon$ as in the previous experiments.
The results are showin in Table \ref{tab:LE4}. 

\begin{samepage}
\begin{table}[h]
\caption{Statistics on LLE computation for Metzler problems of dimension $d=50$ with entries in $\{-1,0,1\}$}
\label{tab:LE3}
\begin{center}
\begin{tabular}{|c|c|c|c|c|c|c|c|c|}
  \hline
  $\dt$ & $L_{\min}$ & $L_{\max}$ & $\# V_{\min}$ & $\# V_{\max}$ & $< \# V >$ & $\gamma_{\min}$ & $\gamma_{\max}$ & $< \gamma >$ \\
  \hline \hline
  1/4 & 1 & 2 & 2 & 2 & 2 & 1.293 & 2.217 & 1.63201 \\
  \hline
  1/16 & 1 & 4 & 4 & 9 & 7 & 0.728 & 1.346 & 0.93261 \\
  \hline
  1/64 & 1 & 10 & 26 & 69 & 51 & 0.234 & 0.375 & 0.30608 \\
  \hline
  1/128 & 1 & 13 & 72 & 264 & 171 & 0.186 & 0.359 & 0.25354 \\
  \hline
\end{tabular}
\end{center}
\end{table}

\begin{table}[h]
\caption{Statistics on LLE computation for Metzler problems of dimension $d=50$ with entries in $[-1,1]$}
\label{tab:LE4}
\begin{center}
\begin{tabular}{|c|c|c|c|c|c|c|c|c|}
  \hline
  $\dt$ & $L_{\min}$ & $L_{\max}$ & $\# V_{\min}$ & $\# V_{\max}$ & $< \# V >$ & $\gamma_{\min}$ & $\gamma_{\max}$ & $< \gamma >$ \\
  \hline \hline
  1/4 & 1 & 2 & 2 & 2 & 2 & 0.408 & 0.826 & 0.5343 \\
  \hline 
  1/16 & 1 & 4 & 5 & 7 & 6 & 0.231 & 0.348 & 0.28757 \\
  \hline
  1/64 & 1 & 8 & 18 & 29 & 23 & 0.084 & 0.132 & 0.1187 \\
  \hline
  1/128 & 1 & 15 & 36 & 68 & 55 & 0.077 & 0.115 & 0.0921 \\
  \hline
  \end{tabular}
\end{center}
\end{table}
\end{samepage}

Finally we consider Metzler matrices of dimension $d=100$ and entries randomly chosen in $[ -1,1 ]$.
We have analyzed 10 examples with $\eps = \frac{1}{1000}$ when $\dt= \frac{1}{32}$, $\eps = \frac{1}{2000}$ when $\dt= \frac{1}{64}$, 
$\eps = \frac{1}{4000}$ when $\dt= \frac{1}{128}$ and $\eps = \frac{1}{8000}$ when $\dt= \frac{1}{256}$.
The results are reported in Table \ref{tab:LE5}.

\begin{table}[h]
\caption{Statistics on LLE computation for Metzler problems of dimension $d=100$ with entries in $[-1,1]$}
\label{tab:LE5}
\begin{center}
\begin{tabular}{|c|c|c|c|c|c|c|c|c|}
  \hline
  $\dt$ & $L_{\min}$ & $L_{\max}$ & $\# V_{\min}$ & $\# V_{\max}$ & $< \# V >$ & $\gamma_{\min}$ & $\gamma_{\max}$ & $< \gamma >$ \\
  \hline \hline
  1/32  & 1 & 6   & 5    & 10    & 7    & 1.048  & 1.690  & 1.3412 \\
  \hline
  1/64  & 1 & 11  & 11   & 28    & 18   & 0.441  & 0.983  & 0.7733 \\
  \hline
  1/128 & 1 & 10  & 36   & 166   & 90   & 0.181  & 0.272  & 0.2240 \\
  \hline
	1/256 & 1 & 18  & 121  & 672   & 358  & 0.033 & 0.1230  & 0.0658 \\
  \hline 
  \end{tabular}
\end{center}
\end{table}

\subsection*{Conclusion}

We have illustrated a large variety of examples in order to support  the efficiency of our method. 
For sufficiently small $\tau$ we have obtained quite good approximation results. Indeed, for sets of matrices with randomly chosen real entries in the interval [-1,1], for the smallest considered $\tau$, the mean amplitude of the computed interval containing the Lyapunov exponent is smaller than $0.1$ and the results are obtained in a reasonable short time.
On the other side, when we consider sets of matrices with randomly chosen entries in $\{-1,0,1\}$, the average amplitude of the computed interval is larger (it is usually smaller than $0.3$). Also in this case the results are obtained in a relatively short time (few minutes).



\section*{Appendix}

{\tt Proof of Theorem~\ref{th10}} 

It is split into four steps.
First we construct a special positively-homogeneous monotone convex functional~$\varphi$
and prove (step 1) that it is actually a norm. Then, in step 2, we establish
the extremality property of $\varphi$, i.e., that it is non-increasing on each trajectory.
Thus, we have an extremal norm. In steps 3 and 4 we show the existence of a
 generalized trajectory on the unit sphere starting at an arbitrary point. This is done by considering a special
  convex optimal control problem and applying the Banach--Alaoglu compactness theorem.

In view of~(\ref{plus}) it suffices to consider the case $ \sigma (\cA) = 0$.
We take an arbitrary  norm monotone with respect to~$K$, for example,
$\|x\| = (b,x)$, where $b \in {\rm int}\, K^*$, where $K^*$ is a dual cone~(\ref{dual}). For $t \ge 0$ and $z \in K$, denote
$l(z, t) \, = \, \sup\, \{\|x(t)\|\, , \ A \in \cU\, [0,t]\, , \, x(0) = z\}$.
For every fixed $t$, the function $l(\cdot , t)$ is a seminorm on~$K$, i.e., it is positively homogeneous
and convex, as a supremum of homogeneous convex functions. Moreover,
Corollary~\ref{c10} implies that $l$ is non-decreasing  in~$z$, i.e.,
if $z_1 \ge_K z_2$, then $l(z_1, t) \ge l(z_2, t)$ for each~$t$. The
function $\varphi(x) = \sup\limits_{t \in \re_+} l(x, t)$ is, therefore,
also a monotone seminorm on~$K$ as the supremum of monotone seminorms.
Moreover, $\varphi(x) \ge l(x,0) = \|x\|$, hence $\varphi$ is positive.

{\it Step 1}. Let us show that $\varphi(x) < \infty$ for all~$x$, i.e., $\varphi$ is a norm on~$K$.
Denote by~$\cL$ the set of points $x \in K$ such that $\varphi (x) < \infty$.
Since $\varphi$ is convex, homogeneous, and monotone on~$K$, it follows that either $\cL = K$ or
$\cL$ is a face of~$K$. Writing $\tilde \cL$ for the linear span of~$\cL$,
we are going to show that $A \tilde \cL \subset \tilde \cL$ for each $A \in \cA$.
If this is not the case, then $A_0z \notin \tilde \cL$ for some $z \in \cL$ and $A_0 \in \cA$,
hence $(I + tA_0)z \notin \tilde \cL$ for all $t \in (0, \tau]$, and therefore $e^{\, t\, A_0} z \notin \tilde \cL$
for all $t \in (0, \tau]$, where $\tau > 0$ is small enough.
Hence, for the control function $A(t) \equiv A_0, \, \ t \in [0, \tau]$ and for $x(0) = z$ we have
$x(\tau) \notin \tilde \cL$, and consequently, $\varphi(z) \ge \varphi(x(\tau)) = +\infty$, which is a contradiction.
Thus, unless $\cL = K$, the set $\tilde \cL$ is a face plane invariant for all operators from~$\cA$. From the $K$-irreducibility
it follows that either $\cL = K$ (in which case the proof is completed) or $\cL = \{0\}$. It remains to show that the latter is impossible.

 If $\cL = \{0\}$, then $\varphi (z) = +\infty$ for all $z \in K \setminus \{0\}$.
Let $K_1 = \{z \in K, \, \|z\| = 1\}$.
For every natural $n$, denote by $\cH_{\, n}$ the set of points $z \in K_1$,
for which there exists a number $\tau = \tau(z) \le n$ and a trajectory starting at~$z$
such that $\|x(\tau)\| > 2$.
Since $\cH_{\, n}$ is open and $\cup_{n =1}^{\infty}\cH_{\, n} \, = \, K_1$, from the compactness
we conclude that $\cup_{n = 1}^{N}\cU_{\, n} \, = \, K_1$ for some natural~$N$. Thus, $\tau(z) \le N$
for all $z \in K_1$. Whence, starting from an arbitrary point $x_0 \in K_1$
one can consequently build a trajectory $x(t)$
and sequences $\{x_n\}, \{t_n\}$ an $\{\tau_n\}$
such that $t_0 = 0, t_n = \sum_{k=0}^{n-1}\tau(x_k)$, $x_k = x(t_k)$.
For this trajectory, $\|x(t_n)\| \, > \,  2^n$ and $t_n \le nN$, hence $\|x(t_n)\| \, > \,
e^{t_n\ln 2 / N}$. Therefore, $ \sigma (\cA ) \ge \frac{\ln 2}{N} > 0 $, which contradicts the assumption.
Thus,  the case $\cL = \{0\}$ is impossible, hence $\varphi$ is a norm.
\smallskip

{\it Step 2}. By definition, for any trajectory
$x(t)$ the function $\varphi(x(t))$ is non-increasing in~$t$. Indeed, suppose $t_1 < t_2$;
then $\varphi (x(t_1))$ is the supremum of $\|y(t)\|\, , \, t \in [t_1, +\infty)$ over all
possible trajectories $y(\cdot)$ on the half-line $[t_1, +\infty)$ with the initial condition~$y(t_1) = x(t_1)$.
This set of trajectories includes $x(\cdot)$. Hence this supremum is not smaller than $\varphi(t_2)$,
which is the supremum over a narrower set of trajectories $y(\cdot)$ on the half-line $[t_2, +\infty)$ with the initial condition~$y(t_2) = x(t_2)$.
\smallskip

 {\it Step 3}. Thus, we have found a norm $\|x\| = \varphi (x)$ which is
 non-increasing in~$t$ on every trajectory~$x(t)$. In this norm
the function $l(z,t)$ is non-increasing in $t$ for each~$z \in K$.
Hence, the limit $F(z) = \lim\limits_{t \to +\infty}l(z, t)$ exists for every~$z \in K$.
 Let us show that~$F$ is a norm we are looking for. First of all, this is a monotone seminorm on~$K$ as a limit of monotone seminorms. Second, $F(x(t))$ is non-increasing in $t$ on every trajectory~$x(t)$.
 Finally, by definition of~$l(z, t)$, we have
\begin{equation}\label{equiv}
\sup_{A(\cdot) \in \cU\, [0,\tau]\, , \, x_0 \, = \, z }\ F(x(\tau)) \quad =  \quad F(z)\ , \qquad t > 0\, .
\end{equation}
For every $\tau > 0$, we denote
by $\cQ_{\, \tau}$ the set of control functions $A(\cdot) \in \cU\, [0, \tau]$, for which the supremum in the left hand side
of~(\ref{equiv}) is attained. To show that~$\cQ_{\tau}$ is nonempty, we consider, for an arbitrary $a \in K^*$, the following optimal control problem:
\begin{equation}\label{optim1}
\left\{
\begin{array}{l}
\bigl(a , x(\tau) \bigr)\ \to \ \max \\
\dot x \, = \, A\, x\\
x(0)\, = \, z\\
A(t) \in {\rm co}\, (\cA)\, , \ t \in [0,\tau]
\end{array}
\right.
\end{equation}
Since this problem is linear in the control function~$A(\cdot)$, the set ${\rm co}\, (\cA)$
is convex and compact, and the objective function~$\bigl(a , x(\tau) \bigr)$ is linear, it
possesses the optimal solution $(\bar A, \bar x) \in L_1[0, \tau] \times W^1_1 [0, \tau]$~(see, for instance,~\cite{Fil,IT}).
Now we take a maximizing sequence~$\{x_i(\cdot)\}_{i = 1}^{\infty}$, for which $F(x_i(\tau)) \to F(z)$ in~(\ref{equiv})
as $i \to \infty$. By the compactness, without loss of generality it can be assumed that the sequence $x_i(z)$ converges to some point~$y \in K$ as $i \to \infty$. Taking $a \in \partial \, F(y)$ (the subdifferential of $F$ at the point~$y$),
and solving problem~(\ref{optim1}) for that~$a$, we obtain $\bigl(a , \bar x(\tau) \bigr) = F(z)$, and hence~$F(\bar x(\tau)) = F(z)$. Thus, $\cQ_{\, \tau}\, \ne \, \emptyset$ for each $\tau > 0$. Note that this set is compact in
the weak-* topology of the space~$L_1[0, \tau]$ due to Banach--Alaoglu theorem.  Furthermore, the family~$\{\cQ_{\, \tau}\}_{\tau > 0}$ is embedded:
$\cQ_{\, \tau_2} \subset \cQ_{\, \tau_1}$ if $\tau_2 > \tau_1$. Indeed, if $\bar x \in \cQ_{\, \tau_2}$,
then $F(\bar x(\tau_2)) = F(z)$, and hence $F(\bar x(t))$ equals identically to~$F(z)$ on the segment~$[0, \tau_2]$.
Therefore, it equals identically to~$F(z)$ on a smaller segment~$[0, \tau_1]$, and so
$\bar x \in \cQ_{\, \tau_1}$. Since an embedded system of nonempty compact sets has a nonempty intersection,
it follows that there exists  a control function $\bar A$, whose
trajectory $\bar x $ with $\bar x (0) = z$ possesses the property $F(\bar x (t)) \equiv F(\bar x (0))\, , \
t \in [0, +\infty )$.

{\em Step 4}. Thus, we have proved that the seminorm~$F$ is invariant: it is non-increasing in~$t$ on any trajectory
$x(t)$, and for every starting point there is a trajectory, on which $F$ is identically constant.
It remains to show that $F$ is a norm, i.e., $F(z)$ is finite and positive
for every $z \in K\setminus \{0\}$. Since $F(z)$ is defined as a limit of a non-increasing function
as $t \to +\infty$,
we have $F(z) < \infty$. The positivity is proved by contradiction. Let $\cM = \{z \in K, \ F(z) = 0\}$.
Since $F$ is a seminorm, it follows that $\cM$ is either entire $K$ or a face of~$K$.
If this is a face of $K$, then as in Step 1 we conclude that its linear span is a common invariant face
plane for~$\cA$,
 which by the irreducibility implies $\cM = \{0\}$, and the proof is completed.
 If $\cM = K$, then $l(z, t) \to 0$ as $t \to \infty$ for every $z \in K$.
 Take an arbitrary $x_0 \in {\rm int}\, K$. There exists a constant $c > 0$ such that
for every $x \in K$ inequality $\|x\| \le c$ implies $x \, \le_K \, \frac{1}{2}\, x_0$.
Let $n$ be such that $l(x_0, n) < c$. Hence, $x(n) \, \le_K \, \frac{1}{2}\, x(0)$
for every trajectory $x(t)$ with $x(0) = x_0$. Applying now Corollary~\ref{c10} and iterating $k$ times, we get
$x(kn) \, \le_K \, 2^{-k}\, x_0$.  Since the norm is monotone, it follows that
$\|x(kn)\| \, \le \, 2^{-k}\|x_0\|\, , \, k \in \n$. On the other hand, since the norm is non-increasing in~$t$
on every trajectory, we see that $\|x(t)\| \, \le \, 2^{- \bigl[\frac{t}{n} \bigr]}\|x_0\|$, where the
brackets denote the integer part. Since for every $y_0 \in K$ there is a constant $C$ such that $y_0 \le_K
Cx_0$, it follows that for every trajectory $y(t)$ one has
$\|y(t)\| \, \le \, C 2^{- \bigl[\frac{t}{n} \bigr]}\|y_0\|$. Therefore, $ \sigma(\cA) \le - \frac{\ln 2}{n} < 0$.
The contradiction concludes the proof.
\qed
\smallskip

{\tt Proof of Proposition~\ref{p40}.} 

Replacing the family~$\cA$ by~$\cA + hI$, where $h>0$ is large enough,
it may be assumed that~$A \ge_K I$ for all~$A \in \cA$. In this case the family~$\cB \, = \,
\{B = A - I \ | \ A \in \cA\}$ consists of  $K$-irreducible operators that leave~$K$ invariant.
Take arbitrary vectors~$x \in K \, , \, x^*\in K^* \, , \ \|x\| = \|x^*\| = 1$.
For every~$B \in \cB$, we denote~$p_B(x^*, x)\, = \, \max\limits_{n=0, \ldots , d-1}(x^*,B^nx)$.
If~$p_B(x^*, x) = 0$, then~$x$ is contained in a face of~$K$ invariant with respect to~$B$.
This contradicts irreducibility of~$B$. Thus, the function~$p$ is strictly positive. Hence, by the
compactness, there is $a > 0$ such that~$p_B(x^*, x) \ge a$ for all~$B \in \cB$ and
all~$x \in K , x^*\in K^* \, , \, \|x\| = \|x^*\| = 1$.
Assume for the moment that the family~$\cA$ is finite: $\cA = \{A_1, \ldots , A_m\}$
and $A_k = I +B_k$. Every product~$P = \prod_{k=1}^{md}A_k$ of $md$ operators
contains at least $d$ equal terms, say, $A_i$. Then
$P = \prod_{k=1}^{md}(I+B_k) \ge \sum_{n=0}^{d-1}B_i^n$, and therefore,
$(x^*, Px) \, \ge \, p_{B_i}(x^*, x)\, \ge \, a$. In case of general compact set~$\cA$
we take its $\eps$-net $\cA_{\eps} = \{A_i\}_{i=1}^{m(\eps)}$,
and to every $A \in \cA$ we associate the closets element from~$\cA_{\, \eps}$
(if there are several closest elements, we take any of them). There is a function
$c(\eps)$ such that $c(\eps) \to 0, \eps \to 0$, and for every
product~$\Pi$ of length at most~$d$ of operators from~$\cA$, we have $\|\Pi - \Pi'\|\le c(\eps)$,
where~$\Pi'$ is the corresponding product of operators from~$\cA_{\eps}$. Hence,
for every
product~$P$ of length~$d m(\eps)$, we have~$(x^*, Px) \, \ge \, a - c(\eps)$.
Taking~$\eps$ small enough, so that $c(\eps) < a/2$, we
   see that~$(x^*, Px) \, \ge \, a/2$,  for every
product~$P$ of length~$d m(\eps)$ of operators from~$\cA$. Since this holds for all~$x^* \in K^*\, , \, \|x^*\| = 1$,
it follows that~$Px \in {\rm int}(K)$ for every~$x \in K\, , \, \|x\| = 1$. Consequently, the cone~$PK$ is embedded in~$K$.
Let $K'$ be the convex hull
of all cones~$PK$ taken over all products~$P \in \cA^{\, d m(\eps)}$. The cone~$K'$ is embedded in~$K$
and~$AK' \subset K'\, , \  A \in \cA$, hence~$A$ is $K'$-Metzler, which completes the prof.
\qed
\smallskip

{\tt Proof of Theorem~\ref{th20}.} 

In view of~(\ref{plus}) it suffices to consider the case $\check \sigma (\cA) = 0$.  Take a  positive  antinorm on~$K$, for example,
$g(x) = (b,x)$, where $b \in {\rm int}\, K^*$. For $t \ge 0$ and $z \in K$, denote
$r(z, t) \, = \, \inf\, \{g(x(t))\, , \ A \in \cU\, [0,t]\, , \, x(0) = z\}$.

\smallskip

{\em Step 1.} For every fixed $t$, the function $r(\cdot , t)$ is an antinorm on~$K$.
Corollary~\ref{c10} implies  that $r$ is non-decreasing  in~$z$. The
function $\psi(x) = \inf\limits_{t \in \re_+} r(x, t)$ is, therefore,
also an antinorm on~$K$ as the infimum of antinorms.
This antinorm is non-decreasing in~$t$ for every trajectory $x(t)$
(the proof is the same as in Theorem~\ref{th10}). It remains to show that~$\psi$ is not identically zero. 
For an arbitrary $x_0 \in {\rm int}\, K$, there is a constant $c > 0$ such that
 inequality $g(x) \le c$ implies $x \, \le_K \, \frac{1}{2}\, x_0$.
If $\psi (x_0) = 0$, then there is $n > 0$ such that $r(x_0, n) < c$. Hence, there is a control function $\bar A (\cdot )$ on the segment $[0, n]$
and the corresponding
trajectory $\bar x$ with $\bar x(0) = x_0$ and
 $\bar x(n) \, \le_K \, \frac{1}{2}\, \bar x(0)$. If now $A(\cdot )$ is the periodic extension of the control function
$\bar A (\cdot)$ to $\re_+$ with period~$n$, then the corresponding trajectory $x(\cdot)$
satisfies $x(kn) \, \le_K \, 2^{-k}\, \bar x(0)$. Since the antinorm~$g$ is monotone on~$K$ (Lemma~\ref{l50}), we have
$g(x(kn)) \le 2^{-k} g(x_0)\, , \, k \in \n$.   Moreover, $g$ is positive, hence it is equivalent to any norm on~$K$.
Therefore,  $\|x(kn)\| \le C\, 2^{-k}\, , \, k \in \n$, and hence $\check \sigma \le - \frac{\ln 2}{n}$, which contradicts the assumption. Thus, $\psi$ is an extremal antinorm, which concludes the proof of the first part.
 \smallskip

{\em Step 2.} Now let us show that if all operators of~$\cA$ are~$K$-irreducible, then  there is a positive invariant antinorm.
 Take the extremal antinorm~$\psi$ on~$K$ constructed in the previous step and consider
 the function $f(z) = \lim\limits_{t \to +\infty}r(z, t)$, where we now denote $r(z, t) \, = \, \inf\, \{\psi(x(t))\, , \ A \in \cU\, [0, t]\, , \, x(0) = z\}$. Since $\psi(x(t))\, $ is non-decreasing in~$t$, so is
 $\, r(z, t)$. Hence, this limit exists for every~$z \in K$,
 maybe it becomes~$+\infty$. If~$f(x) = +\infty$ for some~$x \in K$, then $f$ is equal to $+\infty$ in the whole interior of~$K$, since for every $z \in {\rm int}(K)$ there is a constant~$c$ such that $cz >_K x$.

Now we invoke Corollary~\ref{c40}: there exists a cone~$K'$ embedded in~$K$ such  that every trajectory~$x(t)$ starting
in~$K'$ remains in~$K'$.  We see that if~$f(x) = +\infty$ for some~$x \in K$, then $f(x_0) = +\infty$
for every point~$x_0 \in K'\setminus \{0\}$. Take an arbitrary~$x_0 \in K'\setminus \{0\}$.
As we saw,~$f(x) = +\infty$ for some~$x \in K$, then $f(x_0) = +\infty$.
 Since $\psi$ is positive and continuous on~$K'$, there exists a constant $C > 0$ such that
for every $x \in K'$ inequality $\psi (x) \ge C$ implies $x \, \ge_K \, 2\, x_0$.
If $f(x_0) = +\infty$, then there is $q > 0$ such that $r(x_0, q) > C$. Hence, $x(q) \, \ge_K \, 2\, x(0)$
for every trajectory $x(t)$ with $x(0) = x_0$.
Applying Corollary~\ref{c10} and iterating $k$ times, we get
$x(kq) \, \ge_K \, 2^{\, k}\, x_0$. Since $\psi$ is monotone, it follows that
$\psi(x(kq)) \, \ge \, 2^{\, k}\psi (x_0)$. On the other hand, $\psi$ is non-decreasing in~$t$
on every trajectory, consequently $\psi (x(t)) \, \ge \, 2^{\bigl[\frac{t}{q} \bigr]}\psi (x_0)$.
Since $\psi$ is equivalent to any norm on~$K'$, it follows that $\check \sigma(\cA) \ge  \frac{\ln 2}{n} > 0$.
The  contradiction shows that $f(x)< +\infty$ for every~$x \in K$, i.e.,~$f$ is an antinorm on~$K$.


Consider now the optimization problem~(\ref{optim1}), where the maximum is replaced by minimum.
It always has a solution $(\bar x, \bar A)\, \in \, W_1^1\times L_1[0, \tau]$ for which
$f(\bar x (\tau)) = F(z)$. Therefore, the set~$\cP_{\, \tau}$ of control functions~$A \in \cU\, [0, \tau]$
for which $f(x(t))$ is equal identically to $f(z)$ on the segment~$[0, \tau]$ is nonempty.
Since the sets $\{\cP_{\, \tau}\}_{\, \tau \in \re_+}$ form an embedded system of nonempty compact sets,
they have a common point~$\bar A (\cdot) \in \cU\, [0, +\infty )$, for which~$f(\bar x(t)) = f(z)\, , \
t \in [0, +\infty)$. Whence,~$f$ is an invariant antinorm. 
\qed

\section*{Acknowledgments}

Part of this work was developed during the stay of the third author at the University of L'Aquila
under the financial support of INdAM GNCS (Istituto Nazionale di Alta Matematica, Gruppo Nazionale
di Calcolo Scientifico) and GSSI (Gran Sasso Science Institute).


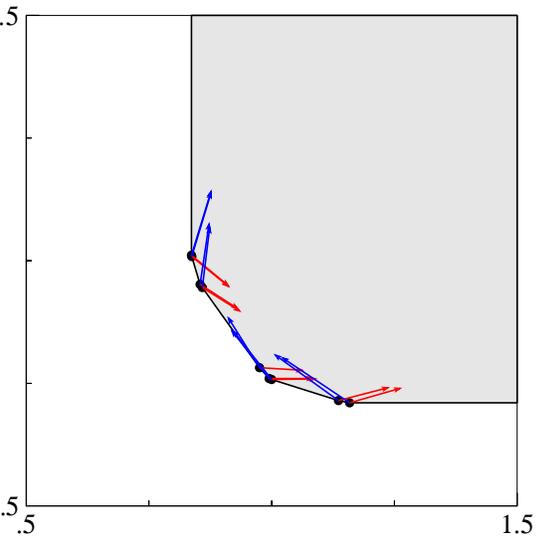
\end{document}